\documentclass
[numbers=enddot,12pt,final,onecolumn,notitlepage,abstracton]{scrartcl}%
\usepackage[headsepline,footsepline,manualmark]{scrlayer-scrpage}
\usepackage[all,cmtip]{xy}
\usepackage{amssymb}
\usepackage{amsmath}
\usepackage{amsthm}
\usepackage{framed}
\usepackage{comment}
\usepackage{color}
\usepackage{tabu}
\usepackage[breaklinks]{hyperref}
\usepackage[sc]{mathpazo}
\usepackage[T1]{fontenc}
\usepackage{needspace}
\usepackage{bbm}
\providecommand{\U}[1]{\protect\rule{.1in}{.1in}}
\newcounter{exer}
\newcounter{exera}
\theoremstyle{definition}
\newtheorem{theo}{Theorem}[section]
\newenvironment{theorem}[1][]
{\begin{theo}[#1]\begin{leftbar}}
{\end{leftbar}\end{theo}}
\newtheorem{lem}[theo]{Lemma}

\newtheorem{prop}[theo]{Proposition}
\newenvironment{proposition}[1][]
{\begin{prop}[#1]\begin{leftbar}}
{\end{leftbar}\end{prop}}
\newtheorem{defi}[theo]{Definition}
\newenvironment{definition}[1][]
{\begin{defi}[#1]\begin{leftbar}}
{\end{leftbar}\end{defi}}
\newtheorem{remk}[theo]{Remark}
\newenvironment{remark}[1][]
{\begin{remk}[#1]\begin{leftbar}}
{\end{leftbar}\end{remk}}
\newtheorem{coro}[theo]{Corollary}
\newenvironment{corollary}[1][]
{\begin{coro}[#1]\begin{leftbar}}
{\end{leftbar}\end{coro}}
\newtheorem{conv}[theo]{Convention}
\newenvironment{convention}[1][]
{\begin{conv}[#1]\begin{leftbar}}
{\end{leftbar}\end{conv}}
\newtheorem{quest}[theo]{Question}
\newenvironment{question}[1][]
{\begin{quest}[#1]\begin{leftbar}}
{\end{leftbar}\end{quest}}
\newtheorem{warn}[theo]{Warning}

\newtheorem{soln}{Solution}

\newtheorem{conj}[theo]{Conjecture}

\newtheorem{exam}[theo]{Example}
\newenvironment{example}[1][]
{\begin{exam}[#1]\begin{leftbar}}
{\end{leftbar}\end{exam}}
\newtheorem{exmp}[exer]{Exercise}

\newtheorem{exetwo}[exera]{Additional exercise}

\newenvironment{noncompile}{}{}
\excludecomment{verlong}
\includecomment{vershort}
\excludecomment{noncompile}
\newcommand{\kk}{\mathbf{k}}
\newcommand{\id}{\operatorname{id}}

\newcommand{\QSym}{\operatorname{QSym}_{\mathbf{k}}}

\newcommand{\undA}{\underline{A}}
\let\sumnonlimits\sum
\let\prodnonlimits\prod
\let\cupnonlimits\bigcup
\let\capnonlimits\bigcap
\renewcommand{\sum}{\sumnonlimits\limits}
\renewcommand{\prod}{\prodnonlimits\limits}
\renewcommand{\bigcup}{\cupnonlimits\limits}
\renewcommand{\bigcap}{\capnonlimits\limits}
\ihead{The Bernstein homomorphism}
\ohead{page \thepage}
\begin{document}

\title{The Bernstein homomorphism via Aguiar-Bergeron-Sottile universality}
\author{Darij Grinberg}
\date{version 2.1, {\today}}
\maketitle

\begin{abstract}
If $H$ is a commutative connected graded Hopf algebra over a commutative ring
$\mathbf{k}$, then a certain canonical $\mathbf{k}$-algebra homomorphism
$H\rightarrow H\otimes\operatorname*{QSym}\nolimits_{\mathbf{k}}$ is defined,
where $\operatorname*{QSym}\nolimits_{\mathbf{k}}$ denotes the Hopf algebra of
quasisymmetric functions. This homomorphism generalizes the \textquotedblleft
internal comultiplication\textquotedblright\ on $\operatorname*{QSym}%
\nolimits_{\mathbf{k}}$, and extends what Hazewinkel (in \S 18.24 of his
\textquotedblleft Witt vectors\textquotedblright) calls the Bernstein homomorphism.

We construct this homomorphism with the help of the universal property of
$\operatorname*{QSym}\nolimits_{\mathbf{k}}$ as a combinatorial Hopf algebra
(a well-known result by Aguiar, Bergeron and Sottile) and extension of scalars
(the commutativity of $H$ allows us to consider, for example, $H\otimes
\operatorname*{QSym}\nolimits_{\mathbf{k}}$ as an $H$-Hopf algebra, and this
change of viewpoint significantly extends the reach of the universal property).

\end{abstract}
\tableofcontents

\section*{***}

One of the most important aspects of $\operatorname*{QSym}$ (the Hopf algebra
of quasisymmetric functions) is a universal property discovered by Aguiar,
Bergeron and Sottile in 2003 \cite{ABS}; among other applications, it gives a
unifying framework for various quasisymmetric and symmetric functions
constructed from combinatorial objects (e.g., the chromatic symmetric function
of a graph).

On the other hand, let $\Lambda_{\mathbf{k}}$ be the Hopf algebra of symmetric
functions over a commutative ring $\mathbf{k}$. If $H$ is any commutative
cocommutative connected graded $\mathbf{k}$-Hopf algebra, then a certain
$\mathbf{k}$-algebra homomorphism $H\rightarrow H\otimes\Lambda_{\mathbf{k}}$
(not a Hopf algebra homomorphism!) was defined by Joseph N. Bernstein, and
used by Zelevinsky in \cite[\S 5.2]{Zelevi81} to classify PSH-algebras. In
\cite[\S 18.24]{HazeWitt}, Hazewinkel observed that this homomorphism
generalizes the second comultiplication of $\Lambda_{\mathbf{k}}$, and asked
for \textquotedblleft more study\textquotedblright\ and a better understanding
of this homomorphism.

In this note, I shall define an extended version of this homomorphism: a
$\mathbf{k}$-algebra homomorphism $H\rightarrow H\otimes\operatorname*{QSym}%
\nolimits_{\mathbf{k}}$ for any commutative (but not necessarily
cocommutative) connected graded $\mathbf{k}$-Hopf algebra $H$. This
homomorphism, which I will call the \textit{Bernstein homomorphism}, will
generalize the second comultiplication of $\operatorname*{QSym}%
\nolimits_{\mathbf{k}}$, or rather its variant with the two tensorands
flipped. When $H$ is cocommutative, this homomorphism has its image contained
in $H\otimes\Lambda_{\mathbf{k}}$ and thus becomes Bernstein's original homomorphism.

The Bernstein homomorphism $H\rightarrow H\otimes\operatorname*{QSym}%
\nolimits_{\mathbf{k}}$ is not fully new (although I have not seen it appear
explicitly in the literature). Its dual version is a coalgebra homomorphism
$H^{\prime}\otimes\operatorname*{NSym}\nolimits_{\mathbf{k}}\rightarrow
H^{\prime}$, where $H^{\prime}$ is a cocommutative connected graded Hopf
algebra; i.e., it is an action of $\operatorname*{NSym}\nolimits_{\mathbf{k}}$
on any such $H^{\prime}$. This action is implicit in the work of Patras and
Reutenauer on descent algebras, and a variant of it for Hopf monoids instead
of Hopf algebras appears in \cite[Propositions 84 and 88, and especially the
Remark after Proposition 88]{Aguiar13}. What I believe to be new in this note
is the way I will construct the Bernstein homomorphism: as a consequence of
the Aguiar-Bergeron-Sottile universal property of $\operatorname*{QSym}$, but
applied not to the $\mathbf{k}$-Hopf algebra $\operatorname*{QSym}%
\nolimits_{\mathbf{k}}$ but to the $H$-Hopf algebra $\operatorname*{QSym}%
\nolimits_{H}$. The commutativity of $H$ is being used here to deploy $H$ as
the base ring.

\subsection*{Acknowledgments}

Thanks to Marcelo Aguiar for enlightening discussions.

\section{Definitions and conventions}

For the rest of this note, we fix a commutative ring\footnote{The word
\textquotedblleft ring\textquotedblright\ always means \textquotedblleft
associative ring with $1$\textquotedblright\ in this note. Furthermore, a
$\mathbf{k}$-algebra (when $\mathbf{k}$ is a commutative ring) means a
$\mathbf{k}$-module $A$ equipped with a ring structure such that the
multiplication map $A\times A\rightarrow A$ is $\mathbf{k}$-bilinear.}
$\mathbf{k}$. All tensor signs ($\otimes$) without a subscript will mean
$\otimes_{\mathbf{k}}$. We shall use the notions of $\mathbf{k}$-algebras,
$\mathbf{k}$-coalgebras and $\mathbf{k}$-Hopf algebras as defined (e.g.) in
\cite[Chapter 1]{Reiner}. We shall also use the notions of graded $\mathbf{k}%
$-algebras, graded $\mathbf{k}$-coalgebras and graded $\mathbf{k}$-Hopf
algebras as defined in \cite[Chapter 1]{Reiner}; in particular, we shall not
use the topologists' sign conventions\footnote{Thus, the twist map $V\otimes
V\rightarrow V\otimes V$ for a graded $\mathbf{k}$-module $V$ sends $v\otimes
w\mapsto w\otimes v$, even if $v$ and $w$ are homogeneous of odd degree.}. The
comultiplication and the counit of a $\mathbf{k}$-coalgebra $C$ will be
denoted by $\Delta_{C}$ and $\varepsilon_{C}$, respectively; when the $C$ is
unambiguously clear from the context, we will omit it from the notation (so we
will just write $\Delta$ and $\varepsilon$).

If $V$ and $W$ are two $\mathbf{k}$-modules, then we let $\tau_{V,W}$ be the
$\mathbf{k}$-linear map $V\otimes W\rightarrow W\otimes V,\ v\otimes w\mapsto
w\otimes v$. This $\mathbf{k}$-linear map $\tau_{V,W}$ is called the
\textit{twist map}, and is a $\mathbf{k}$-module isomorphism.

The next two definitions are taken from \cite[\S 1.4]{Reiner}\footnote{The
objects we are defining are classical and standard; however, the notation we
are using for them is not. For example, what we call $\Delta^{\left(
k-1\right)  }$ in Definition \ref{def.iterated-comult} is denoted by
$\Delta_{k-1}$ in \cite{Sweedler-HA}, and is called $\Delta^{\left(  k\right)
}$ in \cite[\S 7.1]{Fresse-Op}.}:

\begin{definition}
\label{def.iterated-mult}Let $A$ be a $\mathbf{k}$-algebra. Let $m_{A}$ denote
the $\mathbf{k}$-linear map $A\otimes A\rightarrow A,\ a\otimes b\mapsto ab$.
Let $u_{A}$ denote the $\mathbf{k}$-linear map $\mathbf{k}\rightarrow
A,\ \lambda\mapsto\lambda\cdot1_{A}$. (The maps $m_{A}$ and $u_{A}$ are often
denoted by $m$ and $u$ when $A$ is unambiguously clear from the context.) For
any $k\in\mathbb{N}$, we define a $\mathbf{k}$-linear map $m^{\left(
k-1\right)  }:A^{\otimes k}\rightarrow A$ recursively as follows: We set
$m^{\left(  -1\right)  }=u_{A}$, $m^{\left(  0\right)  }=\operatorname*{id}%
\nolimits_{A}$ and%
\[
m^{\left(  k\right)  }=m\circ\left(  \operatorname*{id}\nolimits_{A}\otimes
m^{\left(  k-1\right)  }\right)  \ \ \ \ \ \ \ \ \ \ \text{for every }k\geq1.
\]
The maps $m^{\left(  k-1\right)  }:A^{\otimes k}\rightarrow A$ are called the
\textit{iterated multiplication maps} of $A$.

Notice that for every $k\in\mathbb{N}$, the map $m^{\left(
k-1\right)  }$ is the $\mathbf{k}$-linear map $A^{\otimes k}\rightarrow A$
which sends every $a_{1}\otimes a_{2}\otimes\cdots\otimes a_{k}\in A^{\otimes
k}$ to $a_{1}a_{2}\cdots a_{k}$.
\end{definition}

\begin{definition}
\label{def.iterated-comult}Let $C$ be a $\mathbf{k}$-coalgebra. For any
$k\in\mathbb{N}$, we define a $\mathbf{k}$-linear map $\Delta^{\left(
k-1\right)  }:C\rightarrow C^{\otimes k}$ recursively as follows: We set
$\Delta^{\left(  -1\right)  }=\varepsilon_{C}$, $\Delta^{\left(  0\right)
}=\operatorname*{id}\nolimits_{C}$ and%
\[
\Delta^{\left(  k\right)  }=\left(  \operatorname*{id}\nolimits_{C}%
\otimes\Delta^{\left(  k-1\right)  }\right)  \circ\Delta
\ \ \ \ \ \ \ \ \ \ \text{for every }k\geq1.
\]
The maps $\Delta^{\left(  k-1\right)  }:C\rightarrow C^{\otimes k}$ are called
the \textit{iterated comultiplication maps} of $C$.
\end{definition}

A \textit{composition} shall mean a finite sequence of positive integers. The
\textit{size} of a composition $\alpha=\left(  \alpha_{1},\alpha_{2}%
,\ldots,\alpha_{k}\right)  $ is defined to be the nonnegative integer
$\alpha_{1}+\alpha_{2}+\cdots+\alpha_{k}$, and is denoted by $\left\vert
\alpha\right\vert $. Let $\operatorname*{Comp}$ denote the set of all compositions.

Let $\mathbb{N}$ denote the set $\left\{  0,1,2,\ldots\right\}  $.

\begin{definition}
\label{def.pialpha}Let $H$ be a graded $\mathbf{k}$-module. For every
$n\in\mathbb{N}$, we let $\pi_{n}:H\rightarrow H$ be the canonical projection
of $H$ onto the $n$-th graded component $H_{n}$ of $H$. We shall always regard
$\pi_{n}$ as a map from $H$ to $H$, not as a map from $H$ to $H_{n}$, even
though its image is $H_{n}$.

For every composition $\alpha=\left(  a_{1},a_{2},\ldots,a_{k}\right)  $, we
let $\pi_{\alpha}:H^{\otimes k}\rightarrow H^{\otimes k}$ be the tensor
product $\pi_{a_{1}}\otimes\pi_{a_{2}}\otimes\cdots\otimes\pi_{a_{k}}$ of the
canonical projections $\pi_{a_{i}}:H\rightarrow H$. Thus, the image of
$\pi_{\alpha}$ can be identified with $H_{a_{1}}\otimes H_{a_{2}}\otimes
\cdots\otimes H_{a_{k}}$.
\end{definition}

Let $\operatorname*{QSym}\nolimits_{\mathbf{k}}$ denote the $\mathbf{k}$-Hopf
algebra of quasisymmetric functions defined over $\mathbf{k}$. (This is
defined and denoted by $\mathcal{Q}Sym$ in \cite[\S 3]{ABS}; it is also
defined and denoted by $\operatorname*{QSym}$ in \cite[Chapter 5]{Reiner}.) We
shall follow the notations and conventions of \cite[\S 5.1]{Reiner} as far as
$\operatorname*{QSym}\nolimits_{\mathbf{k}}$ is concerned; in particular, we
regard $\operatorname*{QSym}\nolimits_{\mathbf{k}}$ as a subring of the ring
$\mathbf{k}\left[  \left[  x_{1},x_{2},x_{3},\ldots\right]  \right]  $ of
formal power series in countably many indeterminates $x_{1},x_{2},x_{3}%
,\ldots$.

Let $\varepsilon_{P}$ denote the $\mathbf{k}$-linear map $\operatorname*{QSym}%
\nolimits_{\mathbf{k}}\rightarrow\mathbf{k}$ sending every $f\in
\operatorname*{QSym}\nolimits_{\mathbf{k}}$ to $f\left(  1,0,0,0,\ldots
\right)  \in\mathbf{k}$. (This map $\varepsilon_{P}$ is denoted by
$\zeta_{\mathcal{Q}}$ in \cite[\S 4]{ABS} and by $\zeta_{Q}$ in \cite[Example
7.1.2]{Reiner}.) Notice that $\varepsilon_{P}$ is a $\mathbf{k}$-algebra homomorphism.

\begin{definition}
\label{def.Malpha}For every composition $\alpha=\left(  \alpha_{1},\alpha
_{2},\ldots,\alpha_{\ell}\right)  $, we define a power series $M_{\alpha}%
\in\mathbf{k}\left[  \left[  x_{1},x_{2},x_{3},\ldots\right]  \right]  $ by%
\[
M_{\alpha}=\sum_{1\leq i_{1}<i_{2}<\cdots<i_{\ell}}x_{i_{1}}^{\alpha_{1}%
}x_{i_{2}}^{\alpha_{2}}\cdots x_{i_{\ell}}^{\alpha_{\ell}}%
\]
(where the sum is over all strictly increasing $\ell$-tuples $\left(
i_{1}<i_{2}<\cdots<i_{\ell}\right)  $ of positive integers). It is well-known
(and easy to check) that this $M_{\alpha}$ belongs to $\operatorname*{QSym}%
\nolimits_{\mathbf{k}}$. The power series $M_{\alpha}$ is called the
\textit{monomial quasisymmetric function} corresponding to $\alpha$. The
family $\left(  M_{\alpha}\right)  _{\alpha\in\operatorname*{Comp}}$ is a
basis of the $\mathbf{k}$-module $\operatorname*{QSym}\nolimits_{\mathbf{k}}$;
this is the so-called \textit{monomial basis} of $\operatorname*{QSym}%
\nolimits_{\mathbf{k}}$. (See \cite[\S 3]{ABS} and \cite[\S 5.1]{Reiner} for
more about this basis.)
\end{definition}

It is well-known that every $\left(  b_{1},b_{2},\ldots,b_{\ell}\right)
\in\operatorname*{Comp}$ satisfies%
\begin{equation}
\Delta\left(  M_{\left(  b_{1},b_{2},\ldots,b_{\ell}\right)  }\right)
=\sum_{i=0}^{\ell}M_{\left(  b_{1},b_{2},\ldots,b_{i}\right)  }\otimes
M_{\left(  b_{i+1},b_{i+2},\ldots,b_{\ell}\right)  } \label{eq.DeltaM}%
\end{equation}
and%
\[
\varepsilon\left(  M_{\left(  b_{1},b_{2},\ldots,b_{\ell}\right)  }\right)  =%
\begin{cases}
1, & \text{if }\ell=0;\\
0, & \text{if }\ell\neq0
\end{cases}
.
\]
These two equalities can be used as a definition of the $\mathbf{k}$-coalgebra
structure on $\operatorname*{QSym}\nolimits_{\mathbf{k}}$ (because $\left(
M_{\alpha}\right)  _{\alpha\in\operatorname*{Comp}}$ is a basis of the
$\mathbf{k}$-module $\operatorname*{QSym}\nolimits_{\mathbf{k}}$, and thus the
$\mathbf{k}$-linear maps $\Delta$ and $\varepsilon$ are uniquely determined by
their values on the $M_{\alpha}$).

\section{The Aguiar-Bergeron-Sottile theorem}

The cornerstone of the Aguiar-Bergeron-Sottile paper \cite{ABS} is the
following result:

\begin{theorem}
\label{thm.ABS.hopf}Let $\mathbf{k}$ be a commutative ring. Let $H$ be a
connected graded $\mathbf{k}$-Hopf algebra. Let $\zeta:H\rightarrow\mathbf{k}$
be a $\mathbf{k}$-algebra homomorphism.

\textbf{(a)} Then, there exists a unique graded $\mathbf{k}$-coalgebra
homomorphism $\Psi:H\rightarrow\operatorname*{QSym}\nolimits_{\mathbf{k}}$ for
which the diagram%
\[%
\xymatrix{
H \ar[rr]^-{\Psi} \ar[dr]_{\zeta} & & \QSym\ar[dl]^{\varepsilon_P} \\
& \kk&
}%
\]
is commutative.

\textbf{(b)} This unique $\mathbf{k}$-coalgebra homomorphism $\Psi
:H\rightarrow\operatorname*{QSym}\nolimits_{\mathbf{k}}$ is a $\mathbf{k}%
$-Hopf algebra homomorphism.

\textbf{(c)} For every composition $\alpha=\left(  a_{1},a_{2},\ldots
,a_{k}\right)  $, define a $\mathbf{k}$-linear map $\zeta_{\alpha
}:H\rightarrow\mathbf{k}$ as the composition%
\[%
\xymatrixcolsep{3pc}
\xymatrix{
H \ar[r]^-{\Delta^{(k-1)}} & H^{\otimes k} \ar[r]^-{\pi_\alpha} & H^{\otimes
k}
\ar[r]^-{\zeta^{\otimes k}} & \kk^{\otimes k} \ar[r]^{\cong} & \kk}%
.
\]
(Here, the map $\mathbf{k}^{\otimes k}\overset{\cong}{\longrightarrow
}\mathbf{k}$ is the canonical $\mathbf{k}$-algebra isomorphism from
$\mathbf{k}^{\otimes k}$ to $\mathbf{k}$. Recall also that $\Delta^{\left(
k-1\right)  }:H\rightarrow H^{\otimes k}$ is the \textquotedblleft iterated
comultiplication map\textquotedblright; see \cite[\S 1.4]{Reiner} for its
definition. The map $\pi_{\alpha}:H^{\otimes k}\rightarrow H^{\otimes k}$ is
the one defined in Definition \ref{def.pialpha}.)

Then, the unique $\mathbf{k}$-coalgebra homomorphism $\Psi$ of Theorem
\ref{thm.ABS.hopf} \textbf{(a)} is given by the formula
\[
\Psi\left(  h\right)  =\sum_{\substack{\alpha\in\operatorname*{Comp}%
;\\\left\vert \alpha\right\vert =n}}\zeta_{\alpha}\left(  h\right)  \cdot
M_{\alpha}\ \ \ \ \ \ \ \ \ \ \text{whenever }n\in\mathbb{N}\text{ and }h\in
H_{n}.
\]
(Recall that $H_{n}$ denotes the $n$-th graded component of $H$.)

\textbf{(d)} The unique $\mathbf{k}$-coalgebra homomorphism $\Psi$ of Theorem
\ref{thm.ABS.hopf} \textbf{(a)} is also given by%
\[
\Psi\left(  h\right)  =\sum_{\alpha\in\operatorname*{Comp}}\zeta_{\alpha
}\left(  h\right)  \cdot M_{\alpha}\ \ \ \ \ \ \ \ \ \ \text{for every }h\in
H
\]
(in particular, the sum on the right hand side of this equality has only
finitely many nonzero addends).

\textbf{(e)} Assume that the $\mathbf{k}$-coalgebra $H$ is cocommutative.
Then, the unique $\mathbf{k}$-coalgebra homomorphism $\Psi$ of Theorem
\ref{thm.ABS.hopf} \textbf{(a)} satisfies $\Psi\left(  H\right)
\subseteq\Lambda_{\mathbf{k}}$, where $\Lambda_{\mathbf{k}}$ is the
$\mathbf{k}$-algebra of symmetric functions over $\mathbf{k}$. (See
\cite[\S 2]{Reiner} for the definition of $\Lambda_{\mathbf{k}}$. We regard
$\Lambda_{\mathbf{k}}$ as a $\mathbf{k}$-subalgebra of $\operatorname*{QSym}%
\nolimits_{\mathbf{k}}$ in the usual way.)
\end{theorem}

Parts \textbf{(a)}, \textbf{(b)} and \textbf{(c)} of Theorem
\ref{thm.ABS.hopf} are proven in \cite[proof of Theorem 4.1]{ABS} and
\cite[proof of Theorem 7.1.3]{Reiner} (although we are using different
notations here\footnote{The paper \cite{ABS} defines a \textit{combinatorial
coalgebra} to be a pair $\left(  H,\zeta\right)  $ consisting of a connected
graded $\mathbf{k}$-coalgebra $H$ (where \textquotedblleft
connected\textquotedblright\ means that $\varepsilon\mid_{H_{0}}%
:H_{0}\rightarrow\mathbf{k}$ is a $\mathbf{k}$-module isomorphism) and a
$\mathbf{k}$-linear map $\zeta:H\rightarrow\mathbf{k}$ satisfying $\zeta
\mid_{H_{0}}=\varepsilon\mid_{H_{0}}$. Furthermore, it defines a
\textit{morphism} from a combinatorial coalgebra $\left(  H^{\prime}%
,\zeta^{\prime}\right)  $ to a combinatorial coalgebra $\left(  H,\zeta
\right)  $ to be a homomorphism $\alpha:H^{\prime}\rightarrow H$ of graded
$\mathbf{k}$-coalgebras for which the diagram%
\[%
\xymatrix{
H^\prime\ar[rr]^-{\alpha} \ar[dr]_{\zeta^\prime} & & H\ar[dl]^{\zeta} \\
& \kk&
}%
\]
is commutative. Theorem \ref{thm.ABS.hopf} \textbf{(a)} translates into this
language as follows: There exists a unique morphism from the combinatorial
coalgebra $\left(  H,\zeta\right)  $ to the combinatorial coalgebra $\left(
\operatorname*{QSym}\nolimits_{\mathbf{k}},\varepsilon_{P}\right)  $. (Apart
from this, \cite{ABS} is also using the notations $\mathbbm{k}$, $\mathcal{H}%
$, $\mathcal{Q}Sym$ and $\zeta_{\mathcal{Q}}$ for what we call $\mathbf{k}$,
$H$, $\operatorname*{QSym}\nolimits_{\mathbf{k}}$ and $\varepsilon_{P}$.)},
and avoiding the standing assumptions of \cite{ABS} which needlessly require
$\mathbf{k}$ to be a field and $H$ to be of finite type). Theorem
\ref{thm.ABS.hopf} \textbf{(d)} easily follows from Theorem \ref{thm.ABS.hopf}
\textbf{(c)}\footnote{\textit{Proof.} Let $\Psi$ be the unique $\mathbf{k}%
$-coalgebra homomorphism $\Psi$ of Theorem \ref{thm.ABS.hopf} \textbf{(a)}. It
is easy to see that every $n\in\mathbb{N}$, every composition $\alpha$ with
$\left\vert \alpha\right\vert \neq n$ and every $h\in H_{n}$ satisfy
$\zeta_{\alpha}\left(  h\right)  =0$ (because $\pi_{\alpha}\left(
\Delta^{\left(  k-1\right)  }\left(  \underbrace{h}_{\in H_{n}}\right)
\right)  \in\pi_{\alpha}\left(  \Delta^{\left(  k-1\right)  }\left(
H_{n}\right)  \right)  =0$ (for reasons of gradedness)). Hence, for every
$n\in\mathbb{N}$ and every $h\in H_{n}$, we have%
\begin{align*}
\sum_{\alpha\in\operatorname*{Comp}}\zeta_{\alpha}\left(  h\right)  \cdot
M_{\alpha}  &  =\sum_{\substack{\alpha\in\operatorname*{Comp};\\\left\vert
\alpha\right\vert =n}}\zeta_{\alpha}\left(  h\right)  \cdot M_{\alpha}%
+\sum_{\substack{\alpha\in\operatorname*{Comp};\\\left\vert \alpha\right\vert
\neq n}}\underbrace{\zeta_{\alpha}\left(  h\right)  }_{=0}\cdot M_{\alpha}\\
&  =\sum_{\substack{\alpha\in\operatorname*{Comp};\\\left\vert \alpha
\right\vert =n}}\zeta_{\alpha}\left(  h\right)  \cdot M_{\alpha}=\Psi\left(
h\right)  \ \ \ \ \ \ \ \ \ \ \left(  \text{by Theorem \ref{thm.ABS.hopf}
\textbf{(c)}}\right)  .
\end{align*}
Both sides of this equality are $\mathbf{k}$-linear in $h$; thus, it also
holds for every $h\in H$ (even if $h$ is not homogeneous). This proves Theorem
\ref{thm.ABS.hopf} \textbf{(d)}.}. Theorem \ref{thm.ABS.hopf} \textbf{(e)}
appears in \cite[Remark 7.1.4]{Reiner} (and something very close is proven in
\cite[Theorem 4.3]{ABS}). For the sake of completeness, let me give some
details on the proof of Theorem \ref{thm.ABS.hopf} \textbf{(e)}:

\begin{proof}
[Proof of Theorem \ref{thm.ABS.hopf} \textbf{(e)}.]Let $\varepsilon
_{p}:\Lambda_{\mathbf{k}}\rightarrow\mathbf{k}$ be the restriction of the
$\mathbf{k}$-algebra homomorphism $\varepsilon_{P}:\operatorname*{QSym}%
\nolimits_{\mathbf{k}}\rightarrow\mathbf{k}$ to $\Lambda_{\mathbf{k}}$. From
\cite[Theorem 4.3]{ABS}, we know that there exists a unique graded
$\mathbf{k}$-coalgebra homomorphism $\Psi^{\prime}:H\rightarrow\Lambda
_{\mathbf{k}}$ for which the diagram%
\begin{equation}%
\xymatrix{
H \ar[rr]^-{\Psi^\prime} \ar[dr]_{\zeta} & & \Lambda_{\kk} \ar
[dl]^{\varepsilon_p} \\
& \kk&
}
\label{pf.thm.ABS.hopf.e.1}%
\end{equation}
is commutative. Consider this $\Psi^{\prime}$. Let $\iota:\Lambda_{\mathbf{k}%
}\rightarrow\operatorname*{QSym}\nolimits_{\mathbf{k}}$ be the canonical
inclusion map; this is a $\mathbf{k}$-Hopf algebra homomorphism. Also,
$\varepsilon_{p}=\varepsilon_{P}\circ\iota$ (by the definition of
$\varepsilon_{p}$). The commutative diagram (\ref{pf.thm.ABS.hopf.e.1}) yields
$\zeta=\underbrace{\varepsilon_{p}}_{=\varepsilon_{P}\circ\iota}\circ
\Psi^{\prime}=\varepsilon_{P}\circ\iota\circ\Psi^{\prime}$.

Now, consider the unique $\mathbf{k}$-coalgebra homomorphism $\Psi$ of Theorem
\ref{thm.ABS.hopf} \textbf{(a)}. Due to its uniqueness, it has the following
property: If $\widetilde{\Psi}$ is any $\mathbf{k}$-coalgebra homomorphism
$H\rightarrow\operatorname*{QSym}\nolimits_{\mathbf{k}}$ for which the diagram%
\begin{equation}%
\xymatrix{
H \ar[rr]^-{\widetilde{\Psi}} \ar[dr]_{\zeta} & & \QSym\ar[dl]^{\varepsilon_P}
\\
& \kk&
}
\label{pf.thm.ABS.hopf.e.2}%
\end{equation}
is commutative, then $\widetilde{\Psi}=\Psi$. Applying this to
$\widetilde{\Psi}=\iota\circ\Psi^{\prime}$, we obtain $\iota\circ\Psi^{\prime
}=\Psi$ (since the diagram (\ref{pf.thm.ABS.hopf.e.2}) is commutative for
$\widetilde{\Psi}=\iota\circ\Psi^{\prime}$ (because $\zeta=\varepsilon
_{P}\circ\iota\circ\Psi^{\prime}$)). Hence, $\underbrace{\Psi}_{=\iota
\circ\Psi^{\prime}}\left(  H\right)  =\left(  \iota\circ\Psi^{\prime}\right)
\left(  H\right)  =\iota\left(  \underbrace{\Psi^{\prime}\left(  H\right)
}_{\subseteq\Lambda_{\mathbf{k}}}\right)  \subseteq\iota\left(  \Lambda
_{\mathbf{k}}\right)  =\Lambda_{\mathbf{k}}$. This proves Theorem
\ref{thm.ABS.hopf} \textbf{(e)}.
\end{proof}

\begin{noncompile}
\begin{proof}
[Proof of Theorem \ref{thm.ABS.hopf} \textbf{(e)}.]A \textit{partition} means
an infinite sequence $\left(  \lambda_{1},\lambda_{2},\lambda_{3}%
,\ldots\right)  $ of nonnegative integers satisfying $\lambda_{1}\geq
\lambda_{2}\geq\lambda_{3}\geq\cdots$ and $\left(  \lambda_{i}=0\text{ for all
sufficiently high }i\right)  $. We identify any partition $\left(  \lambda
_{1},\lambda_{2},\lambda_{3},\ldots\right)  $ with the finite sequence
$\left(  \lambda_{1},\lambda_{2},\ldots,\lambda_{m}\right)  $ whenever $m$ is
chosen such that $\lambda_{m+1}=\lambda_{m+2}=\lambda_{m+3}=\cdots=0$. Thus,
in particular, any partition is identified with a unique composition. The
\textit{size} of a partition $\lambda$ is thus well-defined (because the size
of a composition is well-defined); explicitly, the size $\left\vert
\lambda\right\vert $ of a partition $\lambda=\left(  \lambda_{1},\lambda
_{2},\lambda_{3},\ldots\right)  $ is $\lambda_{1}+\lambda_{2}+\lambda
_{3}+\cdots$. Let $\operatorname*{Par}$ be the set of all partitions.

For every monomial $\mathfrak{m}$, there exists a unique partition
$\lambda=\left(  \lambda_{1},\lambda_{2},\lambda_{3},\ldots\right)  $ such
that the monomial $\mathfrak{m}$ can be obtained from $x_{1}^{\lambda_{1}%
}x_{2}^{\lambda_{2}}x_{3}^{\lambda_{3}}\cdots$ by interchanging variables.
(Namely, this partition $\lambda$ is obtained by arranging the exponents of
$\mathfrak{m}$ in weakly decreasing order.) This partition $\lambda$ will be
denoted by $\operatorname*{profile}\lambda$.

For every partition $\lambda$, define $m_{\lambda}$ to be the sum of all
monomials which are obtained from $x_{1}^{\lambda_{1}}x_{2}^{\lambda_{2}}%
x_{3}^{\lambda_{3}}\cdots$ by interchanging variables\footnote{Each such
monomial appears only once in the sum, even if there are several ways to
obtain it from $x_{1}^{\lambda_{1}}x_{2}^{\lambda_{2}}x_{3}^{\lambda_{3}%
}\cdots$ by interchanging variables}. This $m_{\lambda}$ is a power series in
$\mathbf{k}\left[  \left[  x_{1},x_{2},x_{3},\ldots\right]  \right]  $, and
belongs to $\Lambda$; it is called a \textit{monomial symmetric function}. (It
agrees with the power series $m_{\lambda}$ in \cite[(2.1.1)]{Reiner}.)

For every composition $\alpha$, let $\operatorname*{sort}\alpha$ be the
partition obtained by arranging the entries of $\alpha$ in weakly decreasing
order. Then, it is easy to see that%
\begin{equation}
m_{\lambda}=\sum_{\substack{\alpha\in\operatorname*{Comp}%
;\\\operatorname*{sort}\alpha=\lambda}}M_{\alpha}\ \ \ \ \ \ \ \ \ \ \text{for
every }\lambda\in\operatorname*{Par}. \label{pf.thm.ABS.hopf.1}%
\end{equation}

[etc.]
\end{proof}
\end{noncompile}

\begin{remark}
\label{rmk.zeta.convolution}Let $\mathbf{k}$, $H$ and $\zeta$ be as in Theorem
\ref{thm.ABS.hopf}. Then, the $\mathbf{k}$-module $\operatorname*{Hom}\left(
H,\mathbf{k}\right)  $ of all $\mathbf{k}$-linear maps from $H$ to
$\mathbf{k}$ has a canonical structure of a $\mathbf{k}$-algebra; its unity is
the map $\varepsilon\in\operatorname*{Hom}\left(  H,\mathbf{k}\right)  $, and
its multiplication is the binary operation $\star$ defined by%
\[
f\star g=m_{\mathbf{k}}\circ\left(  f\otimes g\right)  \circ\Delta
_{H}:H\rightarrow\mathbf{k}\ \ \ \ \ \ \ \ \ \ \text{for every }%
f,g\in\operatorname*{Hom}\left(  H,\mathbf{k}\right)
\]
(where $m_{\mathbf{k}}$ is the canonical isomorphism $\mathbf{k}%
\otimes\mathbf{k}\rightarrow\mathbf{k}$). This $\mathbf{k}$-algebra is called
the \textit{convolution algebra} of $H$ and $\mathbf{k}$; it is a particular
case of the construction in \cite[Definition 1.4.1]{Reiner}. Using this
convolution algebra, we can express the map $\zeta_{\alpha}$ in Theorem
\ref{thm.ABS.hopf} \textbf{(c)} as follows: For every composition
$\alpha=\left(  a_{1},a_{2},\ldots,a_{k}\right)  $, the map $\zeta_{\alpha
}:H\rightarrow\mathbf{k}$ is given by%
\[
\zeta_{\alpha}=\left(  \zeta\circ\pi_{a_{1}}\right)  \star\left(  \zeta
\circ\pi_{a_{2}}\right)  \star\cdots\star\left(  \zeta\circ\pi_{a_{k}}\right)
.
\]
(This follows from \cite[Exercise 1.4.23]{Reiner}.)
\end{remark}

\begin{noncompile}
[... add an alternative proof of Theorem \ref{thm.ABS.hopf} sans uniqueness.]
\end{noncompile}

\section{Extension of scalars and $\left(  \mathbf{k},\protect\underline{A}%
\right)  $-coalgebra homomorphisms}

Various applications of Theorem \ref{thm.ABS.hopf} can be found in \cite{ABS}
and \cite[Chapter 7]{Reiner}. We are going to present another application,
which we will obtain by \textquotedblleft leveraging\textquotedblright%
\ Theorem \ref{thm.ABS.hopf} through an extension-of-scalars
argument\footnote{I have learned this extension-of-scalars trick from
Petracci's \cite[proof of Lemma 2.1.1]{Petracci}; similar ideas appear in
various other algebraic arguments.}. Let us first introduce some more notations.

\begin{definition}
\label{def.forget}Let $H$ be a $\mathbf{k}$-algebra (possibly with additional
structure, such as a grading or a Hopf algebra structure). Then,
$\underline{H}$ will mean the $\mathbf{k}$-algebra $H$ without any additional
structure (for instance, the $\mathbf{k}$-coalgebra structure on $H$ is
forgotten if $H$ was a $\mathbf{k}$-bialgebra, and the grading is forgotten if
$H$ was graded). Sometimes we will use the notation $\underline{H}$ even when
$H$ has no additional structure beyond being a $\mathbf{k}$-algebra; in this case, it
means the same as $H$, just stressing the fact that it is a plain $\mathbf{k}%
$-algebra with nothing up its sleeves.

In other words, $\underline{H}$ will denote the image of $H$ under the
forgetful functor from whatever category $H$ belongs to to the category of
$\mathbf{k}$-algebras. We shall often use $\underline{H}$ and $H$
interchangeably, whenever $H$ is merely a $\mathbf{k}$-algebra or the other
structures on $H$ cannot cause confusion.
\end{definition}

\begin{definition}
\label{def.extend-scalars}Let $A$ be a commutative $\mathbf{k}$-algebra.

\textbf{(a)} If $H$ is a $\mathbf{k}$-module, then $\underline{A}\otimes H$
will be understood to mean the $A$-module $A\otimes H$, on which $A$ acts by
the rule
\[
a\left(  b\otimes h\right)  =ab\otimes h\ \ \ \ \ \ \ \ \ \ \text{for all
}a\in A\text{, }b\in A\text{ and }h\in H.
\]
This $A$-module $\underline{A}\otimes H$ is called the $\mathbf{k}%
$\textit{-module }$H$\textit{ with scalars extended to }$\underline{A}$.

We can define a functor $\operatorname*{Mod}\nolimits_{\mathbf{k}}%
\rightarrow\operatorname*{Mod}\nolimits_{A}$ (where $\operatorname*{Mod}%
\nolimits_{B}$ denotes the category of $B$-modules) which sends every object
$H\in\operatorname*{Mod}\nolimits_{\mathbf{k}}$ to $\underline{A}\otimes H$
and every morphism $f\in\operatorname*{Mod}\nolimits_{\mathbf{k}}\left(
H_{1},H_{2}\right)  $ to $\operatorname*{id}\otimes f\in\operatorname*{Mod}%
\nolimits_{A}\left(  \underline{A}\otimes H_{1},\underline{A}\otimes
H_{2}\right)  $; this functor is called \textit{extension of scalars} (from
$\mathbf{k}$ to $A$).

\textbf{(b)} If $H$ is a graded $\mathbf{k}$-module, then the $A$-module
$\underline{A}\otimes H$ canonically becomes a graded $\underline{A}$-module
(namely, its $n$-th graded component is $\underline{A}\otimes H_{n}$, where
$H_{n}$ is the $n$-th graded component of $H$). Notice that even if $A$ is
graded, we disregard its grading when defining the grading on $\underline{A}%
\otimes H$; this is why we are calling it $\underline{A}\otimes H$ and not
$A\otimes H$.

As before, we can define a functor from the category of graded $\mathbf{k}%
$-modules to the category of graded $A$-modules (which functor sends every
object $H$ to $\underline{A}\otimes H$), which is called \textit{extension of
scalars}.

\textbf{(c)} If $H$ is a $\mathbf{k}$-algebra, then the $A$-module
$\underline{A}\otimes H$ becomes an $A$-algebra according to the rule%
\[
\left(  a\otimes h\right)  \left(  b\otimes g\right)  =ab\otimes
hg\ \ \ \ \ \ \ \ \ \ \text{for all }a\in A,\ b\in A\text{, }h\in H\text{ and
}g\in H.
\]
(This is, of course, the same rule as used in the standard definition of the
tensor product $A\otimes H$; but notice that we are regarding $\underline{A}%
\otimes H$ as an $A$-algebra, not just as a $\mathbf{k}$-algebra.) This
$A$-algebra $\underline{A}\otimes H$ is called the $\mathbf{k}$%
\textit{-algebra }$H$\textit{ with scalars extended to }$\underline{A}$.

As before, we can define a functor from the category of $\mathbf{k}$-algebras
to the category of $A$-algebras (which functor sends every object $H$ to
$\underline{A}\otimes H$), which is called \textit{extension of scalars}.

\textbf{(d)} If $H$ is a $\mathbf{k}$-coalgebra, then the $A$-module
$\underline{A}\otimes H$ becomes an $A$-coalgebra. Namely, its
comultiplication is defined to be%
\[
\operatorname*{id}\nolimits_{A}\otimes\Delta_{H}:A\otimes H\rightarrow
A\otimes\left(  H\otimes H\right)  \cong\left(  A\otimes H\right)  \otimes
_{A}\left(  A\otimes H\right)  ,
\]
and its counit is defined to be%
\[
\operatorname*{id}\nolimits_{A}\otimes\varepsilon_{H}:A\otimes H\rightarrow
A\otimes\mathbf{k}\cong A
\]
(recalling that $\Delta_{H}$ and $\varepsilon_{H}$ are the comultiplication
and the counit of $H$, respectively). Note that both the comultiplication and
the counit of $\underline{A}\otimes H$ are $A$-linear, so that this
$A$-coalgebra $\underline{A}\otimes H$ is well-defined. This $A$-coalgebra
$\underline{A}\otimes H$ is called the $\mathbf{k}$\textit{-coalgebra }%
$H$\textit{ with scalars extended to }$\underline{A}$.

As before, we can define a functor from the category of $\mathbf{k}%
$-coalgebras to the category of $A$-coalgebras (which functor sends every
object $H$ to $\underline{A}\otimes H$), which is called \textit{extension of
scalars}.

Notice that $\underline{A}\otimes H$ is an $A$-coalgebra, not a $\mathbf{k}%
$-coalgebra. If $A$ has a pre-existing $\mathbf{k}$-coalgebra structure, then
the $A$-coalgebra structure on $\underline{A}\otimes H$ usually has nothing to
do with the $\mathbf{k}$-coalgebra structure on $A\otimes H$ obtained by
tensoring the $\mathbf{k}$-coalgebras $A$ and $H$.

\textbf{(e)} If $H$ is a $\mathbf{k}$-bialgebra, then the $A$-module
$\underline{A}\otimes H$ becomes an $A$-bialgebra. (Namely, the $A$-algebra
structure and the $A$-coalgebra structure previously defined on $\underline{A}%
\otimes H$, combined, form an $A$-bialgebra structure.) This $A$-bialgebra
$\underline{A}\otimes H$ is called the $\mathbf{k}$\textit{-bialgebra }%
$H$\textit{ with scalars extended to }$\underline{A}$.

As before, we can define a functor from the category of $\mathbf{k}%
$-bialgebras to the category of $A$-bialgebras (which functor sends every
object $H$ to $\underline{A}\otimes H$), which is called \textit{extension of
scalars}.

\textbf{(f)} Similarly, extension of scalars is defined for $\mathbf{k}$-Hopf
algebras, graded $\mathbf{k}$-bialgebras, etc.. Again, all structures on $A$
that go beyond the $\mathbf{k}$-algebra structure are irrelevant and can be forgotten.
\end{definition}

\begin{definition}
\label{def.kA-coalg}Let $A$ be a commutative $\mathbf{k}$-algebra.

\textbf{(a)} Let $H$ be a $\mathbf{k}$-module, and let $G$ be an $A$-module.
For any $\mathbf{k}$-linear map $f:H\rightarrow G$, we let $f^{\sharp}$ denote
the $A$-linear map%
\[
\underline{A}\otimes H\rightarrow G,\ \ \ \ \ \ \ \ \ \ a\otimes h\mapsto
af\left(  h\right)  .
\]
(It is easy to see that this latter map is indeed well-defined and
$A$-linear.) For any $A$-linear map $g:\underline{A}\otimes H\rightarrow G$,
we let $g^{\flat}$ denote the $\mathbf{k}$-linear map%
\[
H\rightarrow G,\ \ \ \ \ \ \ \ \ \ h\mapsto g\left(  1\otimes h\right)  .
\]

Sometimes we will use the notations $f^{\sharp\left(  A,\mathbf{k}\right)  }$
and $g^{\flat\left(  A,\mathbf{k}\right)  }$ instead of $f^{\sharp}$ and
$g^{\flat}$ when the $A$ and the $\mathbf{k}$ are not clear from the context.

It is easy to see that $\left(  f^{\sharp}\right)  ^{\flat}=f$ for any
$\mathbf{k}$-linear map $f:H\rightarrow G$, and that $\left(  g^{\flat
}\right)  ^{\sharp}=g$ for any $A$-linear map $g:\underline{A}\otimes
H\rightarrow G$. Thus, the maps%
\begin{align}
\left\{  \mathbf{k}\text{-linear maps }H\rightarrow G\right\}   &
\rightarrow\left\{  A\text{-linear maps }\underline{A}\otimes H\rightarrow
G\right\}  ,\nonumber\\
f  &  \mapsto f^{\sharp} \label{eq.def.kA-coalg.dies}%
\end{align}
and%
\begin{align}
\left\{  A\text{-linear maps }\underline{A}\otimes H\rightarrow G\right\}   &
\rightarrow\left\{  \mathbf{k}\text{-linear maps }H\rightarrow G\right\}
,\nonumber\\
g  &  \mapsto g^{\flat} \label{eq.def.kA-coalg.bemol}%
\end{align}
are mutually inverse.

This is a particular case of an adjunction between functors (namely, the
Hom-tensor adjunction, with a slight simplification, also known as the
induction-restriction adjunction); this is also the reason why we are using
the $\sharp$ and $\flat$ notations. The maps (\ref{eq.def.kA-coalg.dies}) and
(\ref{eq.def.kA-coalg.bemol}) are natural in $H$ and $G$.

\textbf{(b)} Let $H$ be a $\mathbf{k}$-coalgebra, and let $G$ be an
$A$-coalgebra. A $\mathbf{k}$-linear map $f:H\rightarrow G$ is said to be a
$\left(  \mathbf{k},\underline{A}\right)  $\textit{-coalgebra homomorphism} if
the $A$-linear map $f^{\sharp}:\underline{A}\otimes H\rightarrow G$ is an
$A$-coalgebra homomorphism.
\end{definition}

\begin{proposition}
\label{prop.relative.alg}Let $A$ be a commutative $\mathbf{k}$-algebra. Let
$H$ be a $\mathbf{k}$-algebra. Let $G$ be an $A$-algebra. Let $f:H\rightarrow
G$ be a $\mathbf{k}$-linear map. Then, $f$ is a $\mathbf{k}$-algebra
homomorphism if and only if $f^{\sharp}$ is an $A$-algebra homomorphism.
\end{proposition}

\begin{proof}
[Proof of Proposition \ref{prop.relative.alg}.]Straightforward and left to the
reader. (The main step is to observe that $f^{\sharp}$ is an $A$-algebra
homomorphism if and only if it satisfies the following two conditions:

\begin{enumerate}
\item We have $f^{\sharp}\left(  1\otimes1\right)  =1$.

\item Every $a,b\in A$ and $h,g\in H$ satisfy $f^{\sharp}\left(  \left(
a\otimes h\right)  \left(  b\otimes g\right)  \right)  =f^{\sharp}\left(
a\otimes h\right)  f^{\sharp}\left(  b\otimes g\right)  $.
\end{enumerate}

\noindent This is because the tensor product $\underline{A}\otimes H$ is
spanned by pure tensors.)
\end{proof}

\begin{proposition}
\label{prop.relative.graded}Let $A$ be a commutative $\mathbf{k}$-algebra. Let
$H$ be a graded $\mathbf{k}$-module. Let $G$ be a graded $A$-module. Let
$f:H\rightarrow G$ be a $\mathbf{k}$-linear map. Then, the $\mathbf{k}$-linear
map $f$ is graded if and only if the $\mathbf{k}$-linear map $f^{\sharp}$ is graded.
\end{proposition}

\begin{proof}
[Proof of Proposition \ref{prop.relative.graded}.]Again, straightforward and
therefore omitted.
\end{proof}

Let us first prove some easily-checked properties of $\left(  \mathbf{k}%
,\underline{A}\right)  $-coalgebra homomorphisms.

\begin{proposition}
\label{prop.relative.composition.a}Let $\mathbf{k}$ be a commutative ring. Let
$A$ be a commutative $\mathbf{k}$-algebra. Let $H$ be a $\mathbf{k}%
$-coalgebra. Let $G$ and $I$ be two $A$-coalgebras. Let $f:H\rightarrow G$ be
a $\left(  \mathbf{k},\underline{A}\right)  $-coalgebra homomorphism. Let
$g:G\rightarrow I$ be an $A$-coalgebra homomorphism. Then, $g\circ f$ is a
$\left(  \mathbf{k},\underline{A}\right)  $-coalgebra homomorphism.
\end{proposition}

\begin{proof}
[Proof of Proposition \ref{prop.relative.composition.a}.]Since $f$ is a
$\left(  \mathbf{k},\underline{A}\right)  $-coalgebra homomorphism, the map
$f^{\sharp}:\underline{A}\otimes H\rightarrow G$ is an $A$-coalgebra
homomorphism. Now, straightforward elementwise computation (using the fact
that the map $f$ is $\mathbf{k}$-linear, and the map $g$ is $A$-linear) shows
that
\begin{equation}
\left(  g\circ f\right)  ^{\sharp}=g\circ f^{\sharp}.
\label{pf.prop.relative.composition.a.1}%
\end{equation}
Thus, $\left(  g\circ f\right)  ^{\sharp}$ is an $A$-coalgebra homomorphism
(since $g$ and $f^{\sharp}$ are $A$-coalgebra homomorphisms). In other words,
$g\circ f$ is a $\left(  \mathbf{k},\underline{A}\right)  $-coalgebra
homomorphism. This proves Proposition \ref{prop.relative.composition.a}.
\end{proof}

\begin{proposition}
\label{prop.relative.composition.b}Let $\mathbf{k}$ be a commutative ring. Let
$A$ be a commutative $\mathbf{k}$-algebra. Let $F$ and $H$ be two $\mathbf{k}%
$-coalgebras. Let $G$ be an $A$-coalgebra. Let $f:H\rightarrow G$ be a
$\left(  \mathbf{k},\underline{A}\right)  $-coalgebra homomorphism. Let
$e:F\rightarrow H$ be a $\mathbf{k}$-coalgebra homomorphism. Then, $f\circ e$
is a $\left(  \mathbf{k},\underline{A}\right)  $-coalgebra homomorphism.
\end{proposition}

\begin{proof}
[Proof of Proposition \ref{prop.relative.composition.b}.]Since $f$ is a
$\left(  \mathbf{k},\underline{A}\right)  $-coalgebra homomorphism, the map
$f^{\sharp}:\underline{A}\otimes H\rightarrow G$ is an $A$-coalgebra
homomorphism. The map $\operatorname*{id}\nolimits_{A}\otimes e:\underline{A}%
\otimes F\rightarrow\underline{A}\otimes H$ is an $A$-coalgebra homomorphism
(since $e:F\rightarrow H$ is a $\mathbf{k}$-coalgebra homomorphism). Now,
straightforward computation shows that $\left(  f\circ e\right)  ^{\sharp
}=f^{\sharp}\circ\left(  \operatorname*{id}\nolimits_{A}\otimes e\right)  $.
Hence, $\left(  f\circ e\right)  ^{\sharp}$ is an $A$-coalgebra homomorphism
(since $f^{\sharp}$ and $\operatorname*{id}\nolimits_{A}\otimes e$ are
$A$-coalgebra homomorphisms). In other words, $f\circ e$ is a $\left(
\mathbf{k},\underline{A}\right)  $-coalgebra homomorphism. This proves
Proposition \ref{prop.relative.composition.b}.
\end{proof}

\begin{proposition}
\label{prop.relative.composition.c1}Let $\mathbf{k}$ be a commutative ring.
Let $A$ be a commutative $\mathbf{k}$-algebra. Let $H$ be a $\mathbf{k}%
$-coalgebra. Let $G$ be an $A$-coalgebra. Let $B$ be a commutative
$A$-algebra. Let $p:A\rightarrow B$ be an $A$-algebra homomorphism. (Actually,
$p$ is uniquely determined by the $A$-algebra structure on $B$.) Let
$p_{G}:G\rightarrow B\otimes_{A}G$ be the canonical $A$-module homomorphism
defined as the composition%
\[
G\overset{\cong}{\longrightarrow}A\otimes_{A}G\overset{p\otimes_{A}%
\operatorname*{id}}{\longrightarrow}B\otimes_{A}G.
\]
Let $f:H\rightarrow G$ be a $\left(  \mathbf{k},\underline{A}\right)
$-coalgebra homomorphism. Then, $p_{G}\circ f:H\rightarrow\underline{B}%
\otimes_{A}G$ is a $\left(  \mathbf{k},\underline{B}\right)  $-coalgebra homomorphism.
\end{proposition}

\begin{proof}
[Proof of Proposition \ref{prop.relative.composition.c1}.]Since $f$ is a
$\left(  \mathbf{k},\underline{A}\right)  $-coalgebra homomorphism, the map
$f^{\sharp}=f^{\sharp\left(  A,\mathbf{k}\right)  }:\underline{A}\otimes
H\rightarrow G$ is an $A$-coalgebra homomorphism. Thus, the map
$\operatorname*{id}\nolimits_{B}\otimes_{A}f^{\sharp}:\underline{B}\otimes
_{A}\left(  \underline{A}\otimes H\right)  \rightarrow\underline{B}\otimes
_{A}G$ is a $B$-coalgebra homomorphism.

Let $\kappa:\underline{B}\otimes H\rightarrow\underline{B}\otimes_{A}\left(
\underline{A}\otimes H\right)  $ be the canonical $B$-module isomorphism
(sending each $b\otimes h\in\underline{B}\otimes H$ to $b\otimes_{A}\left(
1\otimes h\right)  $). It is well-known that $\kappa$ is a $B$-coalgebra
isomorphism\footnote{In fact, it is part of the natural isomorphism
$\operatorname*{Ind}\nolimits_{A}^{B}\circ\operatorname*{Ind}%
\nolimits_{\mathbf{k}}^{A}\cong\operatorname*{Ind}\nolimits_{\mathbf{k}}^{B}$,
where $\operatorname*{Ind}\nolimits_{P}^{Q}$ means extension of scalars from
$P$ to $Q$ (as a functor from the category of $P$-coalgebras to the category
of $Q$-coalgebras).}. Thus, $\left(  \operatorname*{id}\nolimits_{B}%
\otimes_{A}f^{\sharp}\right)  \circ\kappa$ is a $B$-coalgebra homomorphism
(since both $\operatorname*{id}\nolimits_{B}\otimes_{A}f^{\sharp}$ and
$\kappa$ are $B$-coalgebra homomorphisms).

The definition of $p_{G}$ yields that
\begin{equation}
p_{G}\left(  u\right)  =1\otimes_{A}u
\label{pf.prop.relative.composition.c1.1}%
\end{equation}
for every $u\in G$.

The map $p_{G}\circ f:H\rightarrow\underline{B}\otimes_{A}G$ gives rise to a
map $\left(  p_{G}\circ f\right)  ^{\sharp\left(  B,\mathbf{k}\right)
}:\underline{B}\otimes H\rightarrow\underline{B}\otimes_{A}G$. But easy
computations show that $\left(  p_{G}\circ f\right)  ^{\sharp\left(
B,\mathbf{k}\right)  }=\left(  \operatorname*{id}\nolimits_{B}\otimes
_{A}f^{\sharp}\right)  \circ\kappa$ (because the map $\left(  p_{G}\circ
f\right)  ^{\sharp\left(  B,\mathbf{k}\right)  }$ sends a pure tensor
$b\otimes h\in\underline{B}\otimes H$ to $b\underbrace{\left(  p_{G}\circ
f\right)  \left(  h\right)  }_{\substack{=p_{G}\left(  f\left(  h\right)
\right)  \\=1\otimes_{A}f\left(  h\right)  \\\text{(by
(\ref{pf.prop.relative.composition.c1.1}))}}}=b\left(  1\otimes_{A}f\left(
h\right)  \right)  =b\otimes_{A}f\left(  h\right)  $, whereas the map $\left(
\operatorname*{id}\nolimits_{B}\otimes_{A}f^{\sharp}\right)  \circ\kappa$
sends a pure tensor $b\otimes h\in\underline{B}\otimes H$ to
\begin{align*}
\left(  \left(  \operatorname*{id}\nolimits_{B}\otimes_{A}f^{\sharp}\right)
\circ\kappa\right)  \left(  b\otimes h\right)   &  =\left(  \operatorname*{id}%
\nolimits_{B}\otimes_{A}f^{\sharp}\right)  \left(  \underbrace{\kappa\left(
b\otimes h\right)  }_{=b\otimes_{A}\left(  1\otimes h\right)  }\right)
=\left(  \operatorname*{id}\nolimits_{B}\otimes_{A}f^{\sharp}\right)  \left(
b\otimes_{A}\left(  1\otimes h\right)  \right) \\
&  =b\otimes_{A}\underbrace{f^{\sharp}\left(  1\otimes h\right)  }_{=1f\left(
h\right)  =f\left(  h\right)  }=b\otimes_{A}f\left(  h\right)
\end{align*}
as well). Thus, $\left(  p_{G}\circ f\right)  ^{\sharp\left(  B,\mathbf{k}%
\right)  }$ is a $B$-coalgebra homomorphism (since $\left(  \operatorname*{id}%
\nolimits_{B}\otimes_{A}f^{\sharp}\right)  \circ\kappa$ is a $B$-coalgebra
homomorphism). In other words, $p_{G}\circ f$ is a $\left(  \mathbf{k}%
,\underline{B}\right)  $-coalgebra homomorphism. This proves Proposition
\ref{prop.relative.composition.c1}.
\end{proof}

\begin{proposition}
\label{prop.relative.composition.c2}Let $\mathbf{k}$ be a commutative ring.
Let $A$ and $B$ be two commutative $\mathbf{k}$-algebras. Let $H$ and $G$ be
two $\mathbf{k}$-coalgebras. Let $f:H\rightarrow\underline{A}\otimes G$ be a
$\left(  \mathbf{k},\underline{A}\right)  $-coalgebra homomorphism. Let
$p:A\rightarrow B$ be a $\mathbf{k}$-algebra homomorphism. Then, $\left(
p\otimes\operatorname*{id}\right)  \circ f:H\rightarrow\underline{B}\otimes G$
is a $\left(  \mathbf{k},\underline{B}\right)  $-coalgebra homomorphism.
\end{proposition}

\begin{proof}
[Proof of Proposition \ref{prop.relative.composition.c2}.]Consider $B$ as an
$A$-algebra by means of the $\mathbf{k}$-algebra homomorphism $p:A\rightarrow
B$. Thus, $p$ becomes an $A$-algebra homomorphism $A\rightarrow B$. Now,
$\underline{A}\otimes G$ is an $A$-coalgebra. Let $p_{\underline{A}\otimes
G}:\underline{A}\otimes G\rightarrow B\otimes_{A}\left(  \underline{A}\otimes
G\right)  $ be the canonical $A$-module homomorphism defined as the
composition%
\[
\underline{A}\otimes G\overset{\cong}{\longrightarrow}A\otimes_{A}\left(
\underline{A}\otimes G\right)  \overset{p\otimes_{A}\operatorname*{id}%
}{\longrightarrow}B\otimes_{A}\left(  \underline{A}\otimes G\right)  .
\]
Proposition \ref{prop.relative.composition.c1} (applied to $\underline{A}%
\otimes G$ and $p_{\underline{A}\otimes G}$ instead of $G$ and $p_{G}$) shows
that $p_{\underline{A}\otimes G}\circ f:H\rightarrow\underline{B}\otimes
_{A}\left(  \underline{A}\otimes G\right)  $ is a $\left(  \mathbf{k}%
,\underline{B}\right)  $-coalgebra homomorphism.

Now, let $\phi$ be the canonical $B$-module isomorphism
\[
\underline{B}\otimes_{A}\left(  \underline{A}\otimes G\right)  \rightarrow
\underbrace{\left(  \underline{B}\otimes_{A}\underline{A}\right)  }%
_{\cong\underline{B}}\otimes G\rightarrow\underline{B}\otimes G.
\]
Then, $\phi$ is a $B$-coalgebra homomorphism, and has the property that
$p\otimes\operatorname*{id}=\phi\circ p_{\underline{A}\otimes G}$ as maps
$\underline{A}\otimes G\rightarrow\underline{B}\otimes G$ (this can be checked
by direct computation). Now,%
\[
\underbrace{\left(  p\otimes\operatorname*{id}\right)  }_{=\phi\circ
p_{\underline{A}\otimes G}}\circ f=\phi\circ p_{\underline{A}\otimes G}\circ
f=\phi\circ\left(  p_{\underline{A}\otimes G}\circ f\right)
\]
must be a $\left(  \mathbf{k},\underline{B}\right)  $-coalgebra homomorphism
(by Proposition \ref{prop.relative.composition.a}, since $p_{\underline{A}%
\otimes G}\circ f$ is a $\left(  \mathbf{k},\underline{B}\right)  $-coalgebra
homomorphism and since $\phi$ is a $B$-coalgebra homomorphism). This proves
Proposition \ref{prop.relative.composition.c2}.
\end{proof}

\begin{proposition}
\label{prop.relative.composition.d}Let $\mathbf{k}$ be a commutative ring. Let
$A$ and $B$ be two commutative $\mathbf{k}$-algebras. Let $H$ be a
$\mathbf{k}$-coalgebra. Let $G$ be an $A$-coalgebra. Let $f:H\rightarrow G$ be
a $\left(  \mathbf{k},\underline{A}\right)  $-coalgebra homomorphism. Then,
$\operatorname*{id}\otimes f:\underline{B}\otimes H\rightarrow\underline{B}%
\otimes G$ is a $\left(  \underline{B},\underline{B}\otimes\underline{A}%
\right)  $-coalgebra homomorphism.
\end{proposition}

\begin{proof}
[Proof of Proposition \ref{prop.relative.composition.d}.]Since $f$ is a
$\left(  \mathbf{k},\underline{A}\right)  $-coalgebra homomorphism, the map
$f^{\sharp}=f^{\sharp\left(  A,\mathbf{k}\right)  }:\underline{A}\otimes
H\rightarrow G$ is an $A$-coalgebra homomorphism. Thus, the map
$\operatorname*{id}\nolimits_{B}\otimes f^{\sharp}:\underline{B}\otimes\left(
\underline{A}\otimes H\right)  \rightarrow\underline{B}\otimes G$ is a
$\underline{B}\otimes\underline{A}$-coalgebra homomorphism.

But the $\underline{B}$-linear map $\operatorname*{id}\otimes f:\underline{B}%
\otimes H\rightarrow\underline{B}\otimes G$ gives rise to a $\underline{B}%
\otimes\underline{A}$-linear map $\left(  \operatorname*{id}\otimes f\right)
^{\sharp\left(  \underline{B}\otimes\underline{A},\underline{B}\right)
}:\left(  \underline{B}\otimes\underline{A}\right)  \otimes_{B}\left(
\underline{B}\otimes H\right)  \rightarrow\underline{B}\otimes G$.

Now, let $\gamma$ be the canonical $\underline{B}\otimes\underline{A}$-module
isomorphism $\left(
\underline{B}\otimes\underline{A}\right)  \otimes_{B}\left(  \underline{B}%
\otimes H\right)  \rightarrow\underline{B}\otimes\left(  \underline{A}\otimes
H\right)  $ (sending each $\left(  b\otimes a\right)  \otimes_{B}\left(
b^{\prime}\otimes h\right)  \in\left(  \underline{B}\otimes\underline{A}%
\right)  \otimes_{B}\left(  \underline{B}\otimes H\right)  $ to $bb^{\prime
}\otimes\left(  a\otimes h\right)  $). Then, $\gamma$ is a $\underline{B}%
\otimes\underline{A}$-coalgebra isomorphism (this is easy to check). Hence,
$\left(  \operatorname*{id}\nolimits_{B}\otimes f^{\sharp}\right)  \circ
\gamma$ is a $\underline{B}\otimes\underline{A}$-coalgebra homomorphism (since
$\operatorname*{id}\nolimits_{B}\otimes f^{\sharp}$ and $\gamma$ are
$\underline{B}\otimes\underline{A}$-coalgebra homomorphisms).

Now, it is straightforward to see that $\left(  \operatorname*{id}\otimes
f\right)  ^{\sharp\left(  \underline{B}\otimes\underline{A},\underline{B}%
\right)  }=\left(  \operatorname*{id}\nolimits_{B}\otimes f^{\sharp}\right)
\circ\gamma$\ \ \ \ \footnote{Indeed, it suffices to check it on pure tensors,
i.e., to prove that
\[
\left(  \operatorname*{id}\otimes f\right)  ^{\sharp\left(  \underline{B}%
\otimes\underline{A},\underline{B}\right)  }\left(  \left(  b\otimes a\right)
\otimes_{B}\left(  b^{\prime}\otimes h\right)  \right)  =\left(  \left(
\operatorname*{id}\nolimits_{B}\otimes f^{\sharp}\right)  \circ\gamma\right)
\left(  \left(  b\otimes a\right)  \otimes_{B}\left(  b^{\prime}\otimes
h\right)  \right)
\]
for each $b\in B$, $a\in A$, $b^{\prime}\in B$ and $h\in H$. But this is easy
(both sides turn out to be $bb^{\prime}\otimes af\left(  h\right)  $).}.
Hence, the map $\left(  \operatorname*{id}\otimes f\right)  ^{\sharp\left(
\underline{B}\otimes\underline{A},\underline{B}\right)  }$ is a $\underline{B}%
\otimes\underline{A}$-coalgebra homomorphism (since $\left(
\operatorname*{id}\nolimits_{B}\otimes f^{\sharp}\right)  \circ\gamma$ is a
$\underline{B}\otimes\underline{A}$-coalgebra homomorphism). In other words,
$\operatorname*{id}\otimes f:\underline{B}\otimes H\rightarrow\underline{B}%
\otimes G$ is a $\left(  \underline{B},\underline{B}\otimes\underline{A}%
\right)  $-coalgebra homomorphism. This proves Proposition
\ref{prop.relative.composition.d}.
\end{proof}

\begin{proposition}
\label{prop.relative.composition.e}Let $\mathbf{k}$ be a commutative ring. Let
$A$ be a commutative $\mathbf{k}$-algebra. Let $B$ be a commutative
$A$-algebra. Let $H$ be a $\mathbf{k}$-coalgebra. Let $G$ be an $A$-coalgebra.
Let $I$ be a $B$-coalgebra. Let $f:H\rightarrow G$ be a $\left(
\mathbf{k},\underline{A}\right)  $-coalgebra homomorphism. Let $g:G\rightarrow
I$ be an $\left(  \underline{A},\underline{B}\right)  $-coalgebra
homomorphism. Then, $g\circ f:H\rightarrow I$ is a $\left(  \mathbf{k}%
,\underline{B}\right)  $-coalgebra homomorphism.
\end{proposition}

\begin{proof}
[Proof of Proposition \ref{prop.relative.composition.e}.]Since $f$ is a
$\left(  \mathbf{k},\underline{A}\right)  $-coalgebra homomorphism, the map
$f^{\sharp\left(  A,\mathbf{k}\right)  }:\underline{A}\otimes H\rightarrow G$
is an $A$-coalgebra homomorphism. Thus, the map $\operatorname*{id}%
\nolimits_{B}\otimes_{A}f^{\sharp\left(  A,\mathbf{k}\right)  }:\underline{B}%
\otimes_{A}\left(  \underline{A}\otimes H\right)  \rightarrow\underline{B}%
\otimes_{A}G$ is a $B$-coalgebra homomorphism.

Since $g:G\rightarrow I$ is an $\left(  \underline{A},\underline{B}\right)
$-coalgebra homomorphism, the map $g^{\sharp\left(  \underline{B}%
,\underline{A}\right)  }:\underline{B}\otimes_{A}G\rightarrow I$ is a
$B$-coalgebra homomorphism.

Let $\delta:\underline{B}\otimes H\rightarrow\underline{B}\otimes_{A}\left(
\underline{A}\otimes H\right)  $ be the canonical $B$-module isomorphism
(sending each $b\otimes h$ to $b\otimes_{A}\left(  1\otimes h\right)  $).
Then, $\delta$ is a $B$-coalgebra isomorphism. Straightforward elementwise
computation shows that $\left(  g\circ f\right)  ^{\sharp\left(
\underline{B},\mathbf{k}\right)  }=g^{\sharp\left(  \underline{B}%
,\underline{A}\right)  }\circ\left(  \operatorname*{id}\nolimits_{B}%
\otimes_{A}f^{\sharp\left(  A,\mathbf{k}\right)  }\right)  \circ\delta$.
Hence, $\left(  g\circ f\right)  ^{\sharp\left(  \underline{B},\mathbf{k}%
\right)  }$ is a $B$-coalgebra homomorphism (since $g^{\sharp\left(
\underline{B},\underline{A}\right)  }$, $\operatorname*{id}\nolimits_{B}%
\otimes_{A}f^{\sharp\left(  A,\mathbf{k}\right)  }$ and $\delta$ are
$B$-coalgebra homomorphisms). In other words, $g\circ f:H\rightarrow I$ is a
$\left(  \mathbf{k},\underline{B}\right)  $-coalgebra homomorphism. This
proves Proposition \ref{prop.relative.composition.e}.
\end{proof}

With these basics in place, we can now \textquotedblleft
escalate\textquotedblright\ Theorem \ref{thm.ABS.hopf} to the following setting:

\begin{corollary}
\label{cor.ABS.hopf.esc}Let $\mathbf{k}$ be a commutative ring. Let $H$ be a
connected graded $\mathbf{k}$-Hopf algebra. Let $A$ be a commutative
$\mathbf{k}$-algebra. Let $\xi:H\rightarrow\underline{A}$ be a $\mathbf{k}%
$-algebra homomorphism.

\textbf{(a)} Then, there exists a unique graded $\left(  \mathbf{k}%
,\underline{A}\right)  $-coalgebra homomorphism $\Xi:H\rightarrow
\underline{A}\otimes\operatorname*{QSym}\nolimits_{\mathbf{k}}$ for which the
diagram
\begin{equation}%
\xymatrix{
H \ar[rr]^-{\Xi} \ar[dr]_{\xi} & & \underline{A} \otimes\QSym\ar
[dl]^{\id_A \otimes\varepsilon_P} \\
& \underline{A} &
}
\label{eq.cor.ABS.hopf.esc.a.diag}%
\end{equation}
is commutative (where we regard $\operatorname*{id}\nolimits_{A}%
\otimes\varepsilon_{P}:\underline{A}\otimes\operatorname*{QSym}%
\nolimits_{\mathbf{k}}\rightarrow\underline{A}\otimes\mathbf{k}$ as a map from
$\underline{A}\otimes\operatorname*{QSym}\nolimits_{\mathbf{k}}$ to
$\underline{A}$, by canonically identifying $\underline{A}\otimes\mathbf{k}$
with $\underline{A}$).

\textbf{(b)} This unique $\left(  \mathbf{k},\underline{A}\right)  $-coalgebra
homomorphism $\Xi:H\rightarrow\underline{A}\otimes\operatorname*{QSym}%
\nolimits_{\mathbf{k}}$ is a $\mathbf{k}$-algebra homomorphism.

\textbf{(c)} For every composition $\alpha=\left(  a_{1},a_{2},\ldots
,a_{k}\right)  $, define a $\mathbf{k}$-linear map $\xi_{\alpha}:H\rightarrow
A$ (not to $\mathbf{k}$ !) as the composition%
\[%
\xymatrixcolsep{3pc}
\xymatrix{
H \ar[r]^-{\Delta^{(k-1)}} & H^{\otimes k} \ar[r]^-{\pi_\alpha} & H^{\otimes
k}
\ar[r]^-{\xi^{\otimes k}} & A^{\otimes k} \ar[r]^-{m^{(k-1)}} & A}%
.
\]
(Recall that $\Delta^{\left(  k-1\right)  }:H\rightarrow H^{\otimes k}$ and
$m^{\left(  k-1\right)  }:A^{\otimes k}\rightarrow A$ are the
\textquotedblleft iterated comultiplication and multiplication
maps\textquotedblright; see \cite[\S 1.4]{Reiner} for their definitions. The
map $\pi_{\alpha}:H^{\otimes k}\rightarrow H^{\otimes k}$ is the one defined
in Definition \ref{def.pialpha}.)

Then, the unique $\left(  \mathbf{k},\underline{A}\right)  $-coalgebra
homomorphism $\Xi$ of Corollary \ref{cor.ABS.hopf.esc} \textbf{(a)} is given
by
\[
\Xi\left(  h\right)  =\sum_{\alpha\in\operatorname*{Comp}}\xi_{\alpha}\left(
h\right)  \otimes M_{\alpha}\ \ \ \ \ \ \ \ \ \ \text{for every }h\in H
\]
(in particular, the sum on the right hand side of this equality has only
finitely many nonzero addends).

\textbf{(d)} If the $\mathbf{k}$-coalgebra $H$ is cocommutative, then
$\Xi\left(  H\right)  $ is a subset of the subring $\underline{A}%
\otimes\Lambda_{\mathbf{k}}$ of $\underline{A}\otimes\operatorname*{QSym}%
\nolimits_{\mathbf{k}}$, where $\Lambda_{\mathbf{k}}$ is the $\mathbf{k}%
$-algebra of symmetric functions over $\mathbf{k}$.
\end{corollary}

\begin{proof}
[Proof of Corollary \ref{cor.ABS.hopf.esc}.]We have $\underline{A}%
\otimes\operatorname*{QSym}\nolimits_{\mathbf{k}}\cong\operatorname*{QSym}%
\nolimits_{\underline{A}}$ as $A$-bialgebras canonically (since
$\operatorname*{QSym}\nolimits_{\mathbf{k}}$ is defined functorially in
$\mathbf{k}$, with a basis that is independent of $\mathbf{k}$).

Recall that we have defined a $\mathbf{k}$-algebra homomorphism $\varepsilon
_{P}:\operatorname*{QSym}\nolimits_{\mathbf{k}}\rightarrow\mathbf{k}$. We
shall now denote this $\varepsilon_{P}$ by $\varepsilon_{P,\mathbf{k}}$ in
order to stress that it depends on $\mathbf{k}$. Similarly, an $\mathbf{m}%
$-algebra homomorphism $\varepsilon_{P,\mathbf{m}}:\operatorname*{QSym}%
\nolimits_{\mathbf{m}}\rightarrow\mathbf{m}$ is defined for any commutative
ring $\mathbf{m}$. In particular, an $\underline{A}$-algebra homomorphism
$\varepsilon_{P,\underline{A}}:\operatorname*{QSym}\nolimits_{\underline{A}%
}\rightarrow\underline{A}$ is defined. The definitions of $\varepsilon
_{P,\mathbf{m}}$ for all $\mathbf{m}$ are essentially identical; thus, the map
$\varepsilon_{P,\underline{A}}:\operatorname*{QSym}\nolimits_{\underline{A}%
}\rightarrow\underline{A}$ can be identified with the map $\operatorname*{id}%
\nolimits_{A}\otimes\varepsilon_{P,\mathbf{k}}:\underline{A}\otimes
\operatorname*{QSym}\nolimits_{\mathbf{k}}\rightarrow\underline{A}%
\otimes\mathbf{k}$ (if we identify $\underline{A}\otimes\operatorname*{QSym}%
\nolimits_{\mathbf{k}}$ with $\operatorname*{QSym}\nolimits_{\underline{A}}$
and identify $\underline{A}\otimes\mathbf{k}$ with $\underline{A}$). We shall
use this identification below.

The $\mathbf{k}$-linear map $\xi:H\rightarrow\underline{A}$ induces an
$A$-linear map $\xi^{\sharp}:\underline{A}\otimes H\rightarrow\underline{A}$
(defined by $\xi^{\sharp}\left(  a\otimes h\right)  =a\xi\left(  h\right)  $
for all $a\in A$ and $h\in H$). Proposition \ref{prop.relative.alg} (applied
to $G=\underline{A}$ and $f=\xi$) shows that $\xi^{\sharp}$ is an $A$-algebra
homomorphism (since $\xi$ is a $\mathbf{k}$-algebra homomorphism).

Theorem \ref{thm.ABS.hopf} \textbf{(a)} (applied to $\underline{A}$,
$\underline{A}\otimes H$ and $\xi^{\sharp}$ instead of $\mathbf{k}$, $H$ and
$\zeta$) shows that there exists a unique graded $\underline{A}$-coalgebra
homomorphism $\Psi:\underline{A}\otimes H\rightarrow\operatorname*{QSym}%
\nolimits_{\underline{A}}$ for which the diagram%
\begin{equation}%
\xymatrix{
\underline{A} \otimes H \ar[rr]^-{\Psi} \ar[dr]_{\xi^{\sharp}}
& & \operatorname{QSym}_{\underline{A}} \ar[dl]^{\varepsilon_{P,\underline{A}%
}} \\
& \underline{A}&
}
\label{pf.cor.ABS.hopf.esc.a.1}%
\end{equation}
is commutative. Since we are identifying the map $\varepsilon_{P,\underline{A}%
}:\operatorname*{QSym}\nolimits_{\underline{A}}\rightarrow\underline{A}$ with
the map $\operatorname*{id}\nolimits_{A}\otimes\varepsilon_{P,\mathbf{k}%
}:\underline{A}\otimes\operatorname*{QSym}\nolimits_{\mathbf{k}}%
\rightarrow\underline{A}\otimes\mathbf{k}=\underline{A}$, we can rewrite this
as follows: There exists a unique graded $\underline{A}$-coalgebra
homomorphism $\Psi:\underline{A}\otimes H\rightarrow\underline{A}%
\otimes\operatorname*{QSym}\nolimits_{\mathbf{k}}$ for which the diagram%
\[%
\xymatrix{
\underline{A} \otimes H \ar[rr]^-{\Psi} \ar[dr]_{\xi^{\sharp}} & & \underline
{A} \otimes\QSym\ar[dl]^{\id_A \otimes\varepsilon_{P,\kk}} \\
& \underline{A}&
}%
\]
is commutative. In other words, there exists a unique graded $\underline{A}%
$-coalgebra homomorphism $\Psi:\underline{A}\otimes H\rightarrow
\underline{A}\otimes\operatorname*{QSym}\nolimits_{\mathbf{k}}$ such that
$\left(  \operatorname*{id}\nolimits_{A}\otimes\varepsilon_{P,\mathbf{k}%
}\right)  \circ\Psi=\xi^{\sharp}$. Let us refer to this observation as the
\textit{intermediate universal property}.

The $\left(  \mathbf{k},\underline{A}\right)  $-coalgebra homomorphisms
$H\rightarrow\underline{A}\otimes\operatorname*{QSym}\nolimits_{\mathbf{k}}$
are in a 1-to-1 correspondence with the $A$-coalgebra homomorphisms
$\underline{A}\otimes H\rightarrow\underline{A}\otimes\operatorname*{QSym}%
\nolimits_{\mathbf{k}}$, which is the same as the $A$-coalgebra homomorphisms
$\underline{A}\otimes H\rightarrow\operatorname*{QSym}\nolimits_{\underline{A}%
}$ (since $\underline{A}\otimes\operatorname*{QSym}\nolimits_{\mathbf{k}}%
\cong\operatorname*{QSym}\nolimits_{\underline{A}}$). The correspondence is
given by sending a $\left(  \mathbf{k},\underline{A}\right)  $-coalgebra
homomorphism $\Xi:H\rightarrow\underline{A}\otimes\operatorname*{QSym}%
\nolimits_{\mathbf{k}}$ to the $A$-coalgebra homomorphism $\Xi^{\sharp
}:\underline{A}\otimes H\rightarrow\underline{A}\otimes\operatorname*{QSym}%
\nolimits_{\mathbf{k}}$. Moreover, this correspondence has the property that
$\Xi$ is graded if and only if $\Xi^{\sharp}$ is (according to Proposition
\ref{prop.relative.graded}). Thus, this correspondence restricts to a
correspondence between the graded $\left(  \mathbf{k},\underline{A}\right)
$-coalgebra homomorphisms $H\rightarrow\underline{A}\otimes
\operatorname*{QSym}\nolimits_{\mathbf{k}}$ and the graded $A$-coalgebra
homomorphisms $\underline{A}\otimes H\rightarrow\underline{A}\otimes
\operatorname*{QSym}\nolimits_{\mathbf{k}}$. Using this correspondence, we can
rewrite the intermediate universal property as follows: There exists a unique
graded $\left(  \mathbf{k},\underline{A}\right)  $-coalgebra homomorphism
$\Xi:H\rightarrow\underline{A}\otimes\operatorname*{QSym}\nolimits_{\mathbf{k}%
}$ such that $\left(  \operatorname*{id}\nolimits_{A}\otimes\varepsilon
_{P,\mathbf{k}}\right)  \circ\Xi^{\sharp}=\xi^{\sharp}$. In other words, there
exists a unique graded $\left(  \mathbf{k},\underline{A}\right)  $-coalgebra
homomorphism $\Xi:H\rightarrow\underline{A}\otimes\operatorname*{QSym}%
\nolimits_{\mathbf{k}}$ such that $\left(  \left(  \operatorname*{id}%
\nolimits_{A}\otimes\varepsilon_{P,\mathbf{k}}\right)  \circ\Xi\right)
^{\sharp}=\xi^{\sharp}$ (since (\ref{pf.prop.relative.composition.a.1}) shows
that $\left(  \left(  \operatorname*{id}\nolimits_{A}\otimes\varepsilon
_{P,\mathbf{k}}\right)  \circ\Xi\right)  ^{\sharp}=\left(  \operatorname*{id}%
\nolimits_{A}\otimes\varepsilon_{P,\mathbf{k}}\right)  \circ\Xi^{\sharp}$). In
other words, there exists a unique graded $\left(  \mathbf{k},\underline{A}%
\right)  $-coalgebra homomorphism $\Xi:H\rightarrow\underline{A}%
\otimes\operatorname*{QSym}\nolimits_{\mathbf{k}}$ such that $\left(
\operatorname*{id}\nolimits_{A}\otimes\varepsilon_{P,\mathbf{k}}\right)
\circ\Xi=\xi$ (since the map (\ref{eq.def.kA-coalg.dies}) is a bijection). In
other words, there exists a unique graded $\left(  \mathbf{k},\underline{A}%
\right)  $-coalgebra homomorphism $\Xi:H\rightarrow\underline{A}%
\otimes\operatorname*{QSym}\nolimits_{\mathbf{k}}$ for which the diagram
(\ref{eq.cor.ABS.hopf.esc.a.diag}) is commutative. This proves Corollary
\ref{cor.ABS.hopf.esc} \textbf{(a)}.

By tracing back the above argument, we see that it yields an explicit
construction of the unique graded $\left(  \mathbf{k},\underline{A}\right)
$-coalgebra homomorphism $\Xi:H\rightarrow\underline{A}\otimes
\operatorname*{QSym}\nolimits_{\mathbf{k}}$ for which the diagram
(\ref{eq.cor.ABS.hopf.esc.a.diag}) is commutative: Namely, it is defined by
$\Xi^{\sharp}=\Psi$, where $\Psi$ is the unique graded $\underline{A}%
$-coalgebra homomorphism $\Psi:\underline{A}\otimes H\rightarrow
\operatorname*{QSym}\nolimits_{\underline{A}}$ for which the diagram
(\ref{pf.cor.ABS.hopf.esc.a.1}) is commutative. Consider these $\Xi$ and
$\Psi$.

Theorem \ref{thm.ABS.hopf} \textbf{(b)} (applied to $\underline{A}$,
$\underline{A}\otimes H$ and $\xi^{\sharp}$ instead of $\mathbf{k}$, $H$ and
$\zeta$) shows that $\Psi:\underline{A}\otimes H\rightarrow
\operatorname*{QSym}\nolimits_{\underline{A}}$ is an $\underline{A}$-Hopf
algebra homomorphism, thus an $\underline{A}$-algebra homomorphism. In other
words, $\Xi^{\sharp}:\underline{A}\otimes H\rightarrow\underline{A}%
\otimes\operatorname*{QSym}\nolimits_{\mathbf{k}}$ is an $\underline{A}%
$-algebra homomorphism (since $\Xi^{\sharp}:\underline{A}\otimes
H\rightarrow\underline{A}\otimes\operatorname*{QSym}\nolimits_{\mathbf{k}}$ is
the same as $\Psi:\underline{A}\otimes H\rightarrow\operatorname*{QSym}%
\nolimits_{\underline{A}}$, up to our identifications). Hence, $\Xi
:H\rightarrow\underline{A}\otimes\operatorname*{QSym}\nolimits_{\mathbf{k}}$
is a $\mathbf{k}$-algebra homomorphism as well (by Proposition
\ref{prop.relative.alg}, applied to $\underline{A}$, $\underline{A}%
\otimes\operatorname*{QSym}\nolimits_{\mathbf{k}}$ and $\Xi$ instead of $A$,
$G$ and $f$). This proves Corollary \ref{cor.ABS.hopf.esc} \textbf{(b)}.

\textbf{(c)} Theorem \ref{thm.ABS.hopf} \textbf{(d)} (applied to
$\underline{A}$, $\underline{A}\otimes H$ and $\xi^{\sharp}$ instead of
$\mathbf{k}$, $H$ and $\zeta$) shows that $\Psi$ is given by%
\begin{equation}
\Psi\left(  h\right)  =\sum_{\alpha\in\operatorname*{Comp}}\left(  \xi
^{\sharp}\right)  _{\alpha}\left(  h\right)  \cdot M_{\alpha}%
\ \ \ \ \ \ \ \ \ \ \text{for every }h\in\underline{A}\otimes H,
\label{pf.cor.ABS.hopf.esc.c.1}%
\end{equation}
where the map $\left(  \xi^{\sharp}\right)  _{\alpha}:\underline{A}\otimes
H\rightarrow\underline{A}$ is defined in the same way as the map
$\zeta_{\alpha}:H\rightarrow\mathbf{k}$ was defined in Theorem
\ref{thm.ABS.hopf} \textbf{(d)} (but with $\mathbf{k}$, $H$ and $\zeta$
replaced by $\underline{A}$, $\underline{A}\otimes H$ and $\xi^{\sharp}$).
Notice that (\ref{pf.cor.ABS.hopf.esc.c.1}) is an equality inside
$\operatorname*{QSym}\nolimits_{\underline{A}}$. Recalling that we are
identifying $\operatorname*{QSym}\nolimits_{\underline{A}}$ with
$\underline{A}\otimes\operatorname*{QSym}\nolimits_{\mathbf{k}}$, we can
rewrite it as an equality in $\underline{A}\otimes\operatorname*{QSym}%
\nolimits_{\mathbf{k}}$; it then takes the form%
\begin{equation}
\Psi\left(  h\right)  =\sum_{\alpha\in\operatorname*{Comp}}\left(  \xi
^{\sharp}\right)  _{\alpha}\left(  h\right)  \otimes M_{\alpha}%
\ \ \ \ \ \ \ \ \ \ \text{for every }h\in\underline{A}\otimes H.
\label{pf.cor.ABS.hopf.esc.c.2}%
\end{equation}

Let $\iota_{H}$ be the $\mathbf{k}$-module homomorphism%
\[
H\rightarrow\underline{A}\otimes H,\ \ \ \ \ \ \ \ \ \ h\mapsto1\otimes h.
\]
Also, for every $k\in\mathbb{N}$, we let $\iota_{k}$ be the $\mathbf{k}%
$-module homomorphism%
\[
H^{\otimes k}\rightarrow\left(  \underline{A}\otimes H\right)  ^{\otimes
_{\underline{A}}k},\ \ \ \ \ \ \ \ \ \ g\mapsto1\otimes g\in\underline{A}%
\otimes H^{\otimes k}\cong\left(  \underline{A}\otimes H\right)
^{\otimes_{\underline{A}}k}%
\]
(where $U^{\otimes_{\underline{A}}k}$ denotes the $k$-th tensor power of an
$\underline{A}$-module $U$); this homomorphism sends every $h_{1}\otimes
h_{2}\otimes\cdots\otimes h_{k}\in H^{\otimes k}$ to $\left(  1\otimes
h_{1}\right)  \otimes_{\underline{A}}\left(  1\otimes h_{2}\right)
\otimes_{\underline{A}}\cdots\otimes_{\underline{A}}\left(  1\otimes
h_{k}\right)  $.

On the other hand, fix some $\alpha\in\operatorname*{Comp}$. Write the
composition $\alpha$ in the form $\alpha=\left(  a_{1},a_{2},\ldots
,a_{k}\right)  $. The diagram%
\[%
\xymatrix@C=4pc@R=3pc{
H \ar[r]_{\Delta^{\left(k-1\right)}} \ar@/^3pc/[rrrr]^{\xi_\alpha}
\ar[d]^{\iota_H}
& H^{\otimes k} \ar[r]_{\pi_\alpha} \ar[d]^{\iota_k}
& H^{\otimes k} \ar[r]_{\xi^{\otimes k}} \ar[d]^{\iota_k}
& A^{\otimes k} \ar[r]_{m^{\left(k-1\right)}}
& A \ar[d]_{\id}
\\
\undA\otimes H \ar[r]^{\Delta^{\left(k-1\right)}} \ar@/_3pc/[rrrr]_{\left
(\xi^\sharp\right)_\alpha}
& \left(\undA\otimes H\right)^{\otimes_{\undA} k} \ar[r]^{\pi_\alpha}
& \left(\undA\otimes H\right)^{\otimes_{\undA} k} \ar[r]^-{\left(\xi
^\sharp\right)^{\otimes_{\undA} k}}
& \undA^{\otimes_{\undA} k} \ar[r]^{\cong}
& \undA}%
\]
is commutative\footnote{\textit{Proof.} In fact:
\par
\begin{itemize}
\item Its upper pentagon is commutative (by the definition of $\xi_{\alpha}$).
\par
\item Its lower pentagon is commutative (by the definition of $\left(
\xi^{\sharp}\right)  _{\alpha}$).
\par
\item Its left square is commutative (since the operation $\Delta^{\left(
k-1\right)  }$ on a $\mathbf{k}$-coalgebra is functorial with respect to the
base ring, i.e., commutes with extension of scalars).
\par
\item Its middle square is commutative (since the operation $\pi_{\alpha}$ on
a graded $\mathbf{k}$-module is functorial with respect to the base ring,
i.e., commutes with extension of scalars).
\par
\item Its right rectangle is commutative. (Indeed, every $h_{1},h_{2}%
,\ldots,h_{k}\in H$ satisfy%
\begin{align*}
&  \left(  \operatorname*{id}\circ m^{\left(  k-1\right)  }\circ\xi^{\otimes
k}\right)  \left(  h_{1}\otimes h_{2}\otimes\cdots\otimes h_{k}\right) \\
&  =m^{\left(  k-1\right)  }\left(  \underbrace{\xi^{\otimes k}\left(
h_{1}\otimes h_{2}\otimes\cdots\otimes h_{k}\right)  }_{=\xi\left(
h_{1}\right)  \otimes\xi\left(  h_{2}\right)  \otimes\cdots\otimes\xi\left(
h_{k}\right)  }\right)  =m^{\left(  k-1\right)  }\left(  \xi\left(
h_{1}\right)  \otimes\xi\left(  h_{2}\right)  \otimes\cdots\otimes\xi\left(
h_{k}\right)  \right) \\
&  =\xi\left(  h_{1}\right)  \xi\left(  h_{2}\right)  \cdots\xi\left(
h_{k}\right)
\end{align*}
and thus%
\begin{align*}
&  \left(  \left(  \xi^{\sharp}\right)  ^{\otimes_{\underline{A}}k}\circ
\iota_{k}\right)  \left(  h_{1}\otimes h_{2}\otimes\cdots\otimes h_{k}\right)
\\
&  =\left(  \xi^{\sharp}\right)  ^{\otimes_{\underline{A}}k}%
\underbrace{\left(  \iota_{k}\left(  h_{1}\otimes h_{2}\otimes\cdots\otimes
h_{k}\right)  \right)  }_{=\left(  1\otimes h_{1}\right)  \otimes
_{\underline{A}}\left(  1\otimes h_{2}\right)  \otimes_{\underline{A}}%
\cdots\otimes_{\underline{A}}\left(  1\otimes h_{k}\right)  }\\
&  =\left(  \xi^{\sharp}\right)  ^{\otimes_{\underline{A}}k}\left(  \left(
1\otimes h_{1}\right)  \otimes_{\underline{A}}\left(  1\otimes h_{2}\right)
\otimes_{\underline{A}}\cdots\otimes_{\underline{A}}\left(  1\otimes
h_{k}\right)  \right) \\
&  =\xi^{\sharp}\left(  1\otimes h_{1}\right)  \otimes_{\underline{A}}%
\xi^{\sharp}\left(  1\otimes h_{2}\right)  \otimes_{\underline{A}}%
\cdots\otimes_{\underline{A}}\xi^{\sharp}\left(  1\otimes h_{k}\right) \\
&  =\xi\left(  h_{1}\right)  \otimes_{\underline{A}}\xi\left(  h_{2}\right)
\otimes_{\underline{A}}\cdots\otimes_{\underline{A}}\xi\left(  h_{k}\right)
\ \ \ \ \ \ \ \ \ \ \left(  \text{since }\xi^{\sharp}\left(  1\otimes
y\right)  =\xi\left(  y\right)  \text{ for every }y\in H\right) \\
&  =\xi\left(  h_{1}\right)  \xi\left(  h_{2}\right)  \cdots\xi\left(
h_{k}\right)  \ \ \ \ \ \ \ \ \ \ \left(  \text{since }\underline{A}%
^{\otimes_{\underline{A}}k}\cong\underline{A}\right) \\
&  =\left(  \operatorname*{id}\circ m^{\left(  k-1\right)  }\circ\xi^{\otimes
k}\right)  \left(  h_{1}\otimes h_{2}\otimes\cdots\otimes h_{k}\right)  .
\end{align*}
Hence, $\left(  \xi^{\sharp}\right)  ^{\otimes_{\underline{A}}k}\circ\iota
_{k}=\operatorname*{id}\circ m^{\left(  k-1\right)  }\circ\xi^{\otimes k}$. In
other words, the right rectangle is commutative.)
\end{itemize}
}. Therefore, $\left(  \xi^{\sharp}\right)  _{\alpha}\circ\iota_{H}%
=\operatorname*{id}\circ\xi_{\alpha}=\xi_{\alpha}$.

Now, forget that we fixed $\alpha$. We thus have shown that%
\begin{equation}
\left(  \xi^{\sharp}\right)  _{\alpha}\circ\iota_{H}=\xi_{\alpha
}\ \ \ \ \ \ \ \ \ \ \text{for every }\alpha\in\operatorname*{Comp}.
\label{pf.cor.ABS.hopf.esc.c.5}%
\end{equation}

Now, every $h\in H$ satisfies%
\begin{align*}
\Xi\left(  h\right)   &  =\underbrace{\Xi^{\sharp}}_{=\Psi}\left(  1\otimes
h\right)  =\Psi\left(  1\otimes h\right) \\
&  =\sum_{\alpha\in\operatorname*{Comp}}\left(  \xi^{\sharp}\right)  _{\alpha
}\underbrace{\left(  1\otimes h\right)  }_{=\iota_{H}\left(  h\right)
}\otimes M_{\alpha}\ \ \ \ \ \ \ \ \ \ \left(  \text{by
(\ref{pf.cor.ABS.hopf.esc.c.2}), applied to }1\otimes h\text{ instead of
}h\right) \\
&  =\sum_{\alpha\in\operatorname*{Comp}}\underbrace{\left(  \xi^{\sharp
}\right)  _{\alpha}\left(  \iota_{H}\left(  h\right)  \right)  }_{=\left(
\left(  \xi^{\sharp}\right)  _{\alpha}\circ\iota_{H}\right)  \left(  h\right)
}\otimes M_{\alpha}=\sum_{\alpha\in\operatorname*{Comp}}\underbrace{\left(
\left(  \xi^{\sharp}\right)  _{\alpha}\circ\iota_{H}\right)  }_{\substack{=\xi
_{\alpha}\\\text{(by (\ref{pf.cor.ABS.hopf.esc.c.5}))}}}\left(  h\right)
\otimes M_{\alpha}\\
&  =\sum_{\alpha\in\operatorname*{Comp}}\xi_{\alpha}\left(  h\right)  \otimes
M_{\alpha}.
\end{align*}
This proves Corollary \ref{cor.ABS.hopf.esc} \textbf{(c)}.

\textbf{(d)} Assume that the $\mathbf{k}$-coalgebra $H$ is cocommutative.
Then, the $A$-coalgebra $\underline{A}\otimes H$ is cocommutative as well.

Let us first see why $\underline{A}\otimes\Lambda_{\mathbf{k}}$ is a subring
of $\underline{A}\otimes\operatorname*{QSym}\nolimits_{\mathbf{k}}$. Indeed,
recall that we are using the standard $A$-Hopf algebra isomorphism
$\underline{A}\otimes\operatorname*{QSym}\nolimits_{\mathbf{k}}\rightarrow
\operatorname*{QSym}\nolimits_{\underline{A}}$ to identify
$\operatorname*{QSym}\nolimits_{\underline{A}}$ with $\underline{A}%
\otimes\operatorname*{QSym}\nolimits_{\mathbf{k}}$. Similarly, let us use the
standard $A$-Hopf algebra isomorphism $\underline{A}\otimes\Lambda
_{\mathbf{k}}\rightarrow\Lambda_{\underline{A}}$ to identify $\Lambda
_{\underline{A}}$ with $\underline{A}\otimes\Lambda_{\mathbf{k}}$. Now,
$\underline{A}\otimes\Lambda_{\mathbf{k}}=\Lambda_{\underline{A}}%
\subseteq\operatorname*{QSym}\nolimits_{\underline{A}}=\underline{A}%
\otimes\operatorname*{QSym}\nolimits_{\mathbf{k}}$.

Theorem \ref{thm.ABS.hopf} \textbf{(e)} (applied to $\underline{A}$,
$\underline{A}\otimes H$ and $\xi^{\sharp}$ instead of $\mathbf{k}$, $H$ and
$\zeta$) shows that $\Psi\left(  \underline{A}\otimes H\right)  \subseteq
\Lambda_{\underline{A}}=\underline{A}\otimes\Lambda_{\mathbf{k}}$. Since
$\Psi=\Xi^{\sharp}$, this rewrites as $\Xi^{\sharp}\left(  \underline{A}%
\otimes H\right)  \subseteq\underline{A}\otimes\Lambda_{\mathbf{k}}$. But
$\Xi\left(  H\right)  \subseteq\Xi^{\sharp}\left(  \underline{A}\otimes
H\right)  $ (since every $h\in H$ satisfies $\Xi\left(  h\right)  =\Xi
^{\sharp}\left(  1\otimes h\right)  \in\Xi^{\sharp}\left(  \underline{A}%
\otimes H\right)  $). Hence, $\Xi\left(  H\right)  \subseteq\Xi^{\sharp
}\left(  \underline{A}\otimes H\right)  \subseteq\underline{A}\otimes
\Lambda_{\mathbf{k}}$. This proves Corollary \ref{cor.ABS.hopf.esc}
\textbf{(d)}.
\end{proof}

\begin{remark}
\label{rmk.xi1.convolution}Let $\mathbf{k}$, $H$, $A$ and $\xi$ be as in
Corollary \ref{cor.ABS.hopf.esc}. Then, the $\mathbf{k}$-module
$\operatorname*{Hom}\left(  H,A\right)  $ of all $\mathbf{k}$-linear maps from
$H$ to $A$ has a canonical structure of a $\mathbf{k}$-algebra; its unity is
the map $u_{A}\circ\varepsilon_{H}\in\operatorname*{Hom}\left(  H,A\right)  $
(where $u_{A}:\mathbf{k}\rightarrow A$ is the $\mathbf{k}$-linear map sending
$1$ to $1$), and its multiplication is the binary operation $\star$ defined by%
\[
f\star g=m_{A}\circ\left(  f\otimes g\right)  \circ\Delta_{H}:H\rightarrow
A\ \ \ \ \ \ \ \ \ \ \text{for every }f,g\in\operatorname*{Hom}\left(
H,A\right)
\]
(where $m_{A}$ is the $\mathbf{k}$-linear map $A\otimes A\rightarrow
A,\ a\otimes b\mapsto ab$). This $\mathbf{k}$-algebra is called the
\textit{convolution algebra} of $H$ and $A$; it is precisely the $\mathbf{k}%
$-algebra defined in \cite[Definition 1.4.1]{Reiner}. Using this $\mathbf{k}%
$-algebra, we can express the map $\xi_{\alpha}$ in Corollary \ref{cor.ABS.hopf.esc}
\textbf{(c)} as follows: For every composition $\alpha=\left(  a_{1}%
,a_{2},\ldots,a_{k}\right)  $, the map $\xi_{\alpha}:H\rightarrow A$ is given
by%
\[
\xi_{\alpha}=\left(  \xi\circ\pi_{a_{1}}\right)  \star\left(  \xi\circ
\pi_{a_{2}}\right)  \star\cdots\star\left(  \xi\circ\pi_{a_{k}}\right)  .
\]
(This follows easily from \cite[Exercise 1.4.23]{Reiner}.)
\end{remark}

\section{The second comultiplication on $\operatorname*{QSym}%
\nolimits_{\mathbf{k}}$}

\begin{convention}
In the following, we do \textbf{not} identify compositions with infinite
sequences, as several authors do. As a consequence, the composition $\left(
1,3\right)  $ does not equal the vector $\left(  1,3,0\right)  $ or the
infinite sequence $\left(  1,3,0,0,0,\ldots\right)  $.
\end{convention}

We now recall the definition of the \textit{second comultiplication} (a.k.a.
\textit{internal comultiplication}) of $\operatorname*{QSym}%
\nolimits_{\mathbf{k}}$. Several definitions of this operation appear in the
literature; we shall use the one in \cite[\S 11.39]{HazeWitt}:\footnote{The
second comultiplication seems to be as old as $\operatorname*{QSym}%
\nolimits_{\mathbf{k}}$; it first appeared in Gessel's \cite[\S 4]{Gessel}
(the same article where $\operatorname*{QSym}\nolimits_{\mathbf{k}}$ was first
defined).}

\begin{definition}
\label{def.QSym.secondcomult}\textbf{(a)} Given a $u\times v$-matrix
$A=\left(  a_{i,j}\right)  _{1\leq i\leq u,\ 1\leq j\leq v}\in\mathbb{N}%
^{u\times v}$ (where $u,v\in\mathbb{N}$) with nonnegative entries, we define
three tuples of nonnegative integers:

\begin{itemize}
\item The $v$-tuple $\operatorname*{column}A\in\mathbb{N}^{v}$ is the
$v$-tuple whose $j$-th entry is $\sum_{i=1}^{u}a_{i,j}$ (that is, the sum of
all entries in the $j$-th column of $A$) for each $j$. (In other words,
$\operatorname*{column}A$ is the sum of all rows of $A$, regarded as vectors.)

\item The $u$-tuple $\operatorname*{row}A\in\mathbb{N}^{u}$ is the $u$-tuple
whose $i$-th entry is $\sum_{j=1}^{v}a_{i,j}$ (that is, the sum of all entries
in the $i$-th row of $A$) for each $i$. (In other words, $\operatorname{row}A$
is the sum of all columns of $A$, regarded as vectors.)

\item The $uv$-tuple $\operatorname*{read}A\in\mathbb{N}^{uv}$ is the
$uv$-tuple whose $\left(  v\left(  i-1\right)  +j\right)  $-th entry is
$a_{i,j}$ for all $i\in\left\{  1,2,\ldots,u\right\}  $ and $j\in\left\{
1,2,\ldots,v\right\}  $. In other words,%
\begin{align*}
&  \operatorname*{read}A\\
&  =\left(  a_{1,1},a_{1,2},\ldots,a_{1,v},a_{2,1},a_{2,2},\ldots
,a_{2,v},\ldots,a_{u,1},a_{u,2},\ldots,a_{u,v}\right)  .
\end{align*}

\end{itemize}

We say that the matrix $A$ is \textit{column-reduced} if
$\operatorname*{column}A$ is a composition (i.e., contains no zero entries).
Equivalently, $A$ is column-reduced if and only if no column of $A$ is the $0$ vector.

We say that the matrix $A$ is \textit{row-reduced} if $\operatorname*{row}A$
is a composition (i.e., contains no zero entries). Equivalently, $A$ is
row-reduced if and only if no row of $A$ is the $0$ vector.

We say that the matrix $A$ is \textit{reduced} if $A$ is both column-reduced
and row-reduced.

\textbf{(b)} If $w\in\mathbb{N}^{k}$ is a $k$-tuple of nonnegative integers
(for some $k\in\mathbb{N}$), then $w^{\operatorname*{red}}$ shall mean the
composition obtained from $w$ by removing each entry that equals $0$. For
instance, $\left(  3,1,0,1,0,0,2\right)  ^{\operatorname*{red}}=\left(
3,1,1,2\right)  $.

\textbf{(c)} Let $\mathbb{N}_{\operatorname*{red}}^{\bullet,\bullet}$ denote
the set of all reduced matrices in $\mathbb{N}^{u\times v}$, where $u$ and $v$
both range over $\mathbb{N}$. In other words, we set
\[
\mathbb{N}_{\operatorname{red}}^{\bullet, \bullet} = \bigcup\limits_{\left(
u, v\right)  \in\mathbb{N}^{2}} \left\{  A \in\mathbb{N}^{u \times v} \mid A
\text{ is reduced} \right\}  .
\]

\textbf{(d)} Let $\Delta_{P}:\operatorname*{QSym}\nolimits_{\mathbf{k}%
}\rightarrow\operatorname*{QSym}\nolimits_{\mathbf{k}}\otimes
\operatorname*{QSym}\nolimits_{\mathbf{k}}$ be the $\mathbf{k}$-linear map
defined by setting%
\[
\Delta_{P}\left(  M_{\alpha}\right)  =\sum_{\substack{A\in\mathbb{N}%
_{\operatorname*{red}}^{\bullet,\bullet};\\\left(  \operatorname*{read}%
A\right)  ^{\operatorname*{red}}=\alpha}}M_{\operatorname*{row}A}\otimes
M_{\operatorname*{column}A}\ \ \ \ \ \ \ \ \ \ \text{for each }\alpha
\in\operatorname*{Comp}.
\]
This map $\Delta_{P}$ is called the \textit{second comultiplication} (or
\textit{internal comultiplication}) of $\operatorname*{QSym}%
\nolimits_{\mathbf{k}}$.

\textbf{(e)} Let $\tau$ denote the twist map $\tau_{\operatorname*{QSym}%
\nolimits_{\mathbf{k}},\operatorname*{QSym}\nolimits_{\mathbf{k}}%
}:\operatorname*{QSym}\nolimits_{\mathbf{k}}\otimes\operatorname*{QSym}%
\nolimits_{\mathbf{k}}\rightarrow\operatorname*{QSym}\nolimits_{\mathbf{k}%
}\otimes\operatorname*{QSym}\nolimits_{\mathbf{k}}$. Let $\Delta_{P}^{\prime
}=\tau\circ\Delta_{P}:\operatorname*{QSym}\nolimits_{\mathbf{k}}%
\rightarrow\operatorname*{QSym}\nolimits_{\mathbf{k}}\otimes
\operatorname*{QSym}\nolimits_{\mathbf{k}}$.
\end{definition}

\begin{example}
The matrix $\left(
\begin{array}
[c]{cccc}%
1 & 0 & 2 & 0\\
2 & 0 & 0 & 5\\
0 & 0 & 3 & 1
\end{array}
\right)  \in\mathbb{N}^{3\times4}$ is row-reduced but not column-reduced (and
thus not reduced). If we denote it by $A$, then $\operatorname*{row}A=\left(
3,7,4\right)  $ and $\operatorname*{column}A=\left(  3,0,5,6\right)  $ and
$\operatorname*{read}A=\left(  1,0,2,0,2,0,0,5,0,0,3,1\right)  $.
\end{example}

\begin{proposition}
\label{prop.QSym.secondbialg}The $\mathbf{k}$-algebra $\operatorname*{QSym}%
\nolimits_{\mathbf{k}}$, equipped with comultiplication $\Delta_{P}$ and
counit $\varepsilon_{P}$, is a $\mathbf{k}$-bialgebra (albeit not a connected
graded one, and not a Hopf algebra).
\end{proposition}

Proposition \ref{prop.QSym.secondbialg} is a well-known fact (appearing, for
example, in \cite[first paragraph of \S 3]{MalReu95}), but we shall actually
derive it further below using our results.

\section{The (generalized) Bernstein homomorphism}

Let us now define the Bernstein homomorphism of a commutative connected graded
$\mathbf{k}$-Hopf algebra, generalizing \cite[\S 18.24]{HazeWitt}:

\begin{definition}
\label{def.bernstein}Let $\mathbf{k}$ be a commutative ring. Let $H$ be a
commutative connected graded $\mathbf{k}$-Hopf algebra. For every composition
$\alpha=\left(  a_{1},a_{2},\ldots,a_{k}\right)  $, define a $\mathbf{k}%
$-linear map $\xi_{\alpha}:H\rightarrow H$ (not to $\mathbf{k}$ !) as the
composition%
\[%
\xymatrixcolsep{3pc}
\xymatrix{
H \ar[r]^{\Delta^{(k-1)}} & H^{\otimes k} \ar[r]^-{\pi_\alpha} & H^{\otimes k}
\ar[r]^-{m^{(k-1)}} & H
}%
.
\]
(Recall that $\Delta^{\left(  k-1\right)  }:H\rightarrow H^{\otimes k}$ and
$m^{\left(  k-1\right)  }:H^{\otimes k}\rightarrow H$ are the
\textquotedblleft iterated comultiplication and multiplication
maps\textquotedblright; see \cite[\S 1.4]{Reiner} for their definitions. The
map $\pi_{\alpha}:H^{\otimes k}\rightarrow H^{\otimes k}$ is the one defined
in Definition \ref{def.pialpha}.) Define a map $\beta_{H}:H\rightarrow
\underline{H}\otimes\operatorname*{QSym}\nolimits_{\mathbf{k}}$ by%
\begin{equation}
\beta_{H}\left(  h\right)  =\sum_{\alpha\in\operatorname*{Comp}}\xi_{\alpha
}\left(  h\right)  \otimes M_{\alpha}\ \ \ \ \ \ \ \ \ \ \text{for every }h\in
H. \label{eq.def.bernstein.def}%
\end{equation}
It is easy to see that this map $\beta_{H}$ is well-defined (i.e., the sum on
the right hand side of (\ref{eq.def.bernstein.def}) has only finitely many
nonzero addends\footnotemark) and $\mathbf{k}$-linear.
\end{definition}

\footnotetext{\textit{Proof.} Let $h\in H$. Then, there exists some
$N\in\mathbb{N}$ such that $h\in H_{0}+H_{1}+\cdots+H_{N-1}$ (since $h\in
H=\bigoplus\limits_{i\in\mathbb{N}}H_{i}$). Consider this $N$. Now, it is easy
to see that every composition $\alpha=\left(  a_{1},a_{2},\ldots,a_{k}\right)
$ of size $\geq N$ satisfies $\left(  \pi_{\alpha}\circ\Delta^{\left(
k-1\right)  }\right)  \left(  h\right)  =0$ (because $\Delta^{\left(
k-1\right)  }\left(  h\right)  $ is concentrated in the first $N$ homogeneous
components of the graded $\mathbf{k}$-module $H^{\otimes k}$, and all of these
components are annihilated by $\pi_{\alpha}$) and therefore $\xi_{\alpha
}\left(  h\right)  =0$. Thus, the sum on the right hand side of
(\ref{eq.def.bernstein.def}) has only finitely many nonzero addends (namely,
all its addends with $\left\vert \alpha\right\vert \geq N$ are $0$).}

\begin{remark}
\label{rmk.xi2.convolution}Let $\mathbf{k}$ and $H$ be as in Definition
\ref{def.bernstein}. Then, the $\mathbf{k}$-module $\operatorname*{Hom}\left(
H,H\right)  $ of all $\mathbf{k}$-linear maps from $H$ to $H$ has a canonical
structure of a $\mathbf{k}$-algebra, defined as in Remark
\ref{rmk.xi1.convolution} (for $A=H$). Using this $\mathbf{k}$-algebra, we can
express the map $\xi_{\alpha}$ from Definition \ref{def.bernstein} as
follows: For every composition $\alpha=\left(  a_{1},a_{2},\ldots
,a_{k}\right)  $, the map $\xi_{\alpha}:H\rightarrow H$ is given by%
\[
\xi_{\alpha}=\pi_{a_{1}}\star\pi_{a_{2}}\star\cdots\star\pi_{a_{k}}.
\]
(This follows easily from \cite[Exercise 1.4.23]{Reiner}.)
\end{remark}

The graded $\mathbf{k}$-Hopf algebra $\operatorname*{QSym}%
\nolimits_{\mathbf{k}}$ is commutative and connected; thus, Definition
\ref{def.bernstein} (applied to $H=\operatorname*{QSym}\nolimits_{\mathbf{k}}%
$) constructs a $\mathbf{k}$-linear map $\beta_{\operatorname*{QSym}%
\nolimits_{\mathbf{k}}}:\operatorname*{QSym}\nolimits_{\mathbf{k}}%
\rightarrow\underline{\operatorname*{QSym}\nolimits_{\mathbf{k}}}%
\otimes\operatorname*{QSym}\nolimits_{\mathbf{k}}$. We shall now prove that
this map is identical with the $\Delta_{P}^{\prime}$ from Definition
\ref{def.QSym.secondcomult} \textbf{(e)}:

\begin{proposition}
\label{prop.QSym.secondcomult.alt}We have $\beta_{\operatorname*{QSym}%
\nolimits_{\mathbf{k}}}=\Delta_{P}^{\prime}$.
\end{proposition}

Before we prove this, let us recall a basic formula for multiplication of
monomial quasisymmetric functions:

\begin{proposition}
\label{prop.QSym.M*M*M}Let $k\in\mathbb{N}$. Let $\alpha_{1},\alpha_{2}%
,\ldots,\alpha_{k}$ be $k$ compositions. Let $\mathbb{N}_{\operatorname*{Cred}%
}^{k,\bullet}$ denote the set of all column-reduced matrices in $\mathbb{N}%
^{k\times v}$ with $v$ ranging over $\mathbb{N}$. In other words, let
\[
\mathbb{N}_{\operatorname{Cred}}^{k, \bullet} = \bigcup\limits_{v
\in\mathbb{N}} \left\{  A \in\mathbb{N}^{k \times v} \mid A \text{ is
column-reduced} \right\}  .
\]
Then,%
\[
M_{\alpha_{1}}M_{\alpha_{2}}\cdots M_{\alpha_{k}}=\sum_{\substack{A\in
\mathbb{N}_{\operatorname*{Cred}}^{k,\bullet};\\\left(  A_{g,\bullet}\right)
^{\operatorname*{red}}=\alpha_{g}\text{ for each }g}}M_{\operatorname*{column}%
A}.
\]
Here, $A_{i,\bullet}$ denotes the $i$-th row of $A$ (regarded as a list of
nonnegative integers).
\end{proposition}

Notice that the $k=2$ case of Proposition \ref{prop.QSym.M*M*M} is a
restatement of the standard formula for the multiplication of monomial
quasisymmetric functions (e.g., \cite[Proposition 5.1.3]{Reiner}%
\footnote{Actually, \cite[Proposition 5.1.3]{Reiner} is slightly more general
(the $k=2$ case of Proposition \ref{prop.QSym.M*M*M} is obtained from
\cite[Proposition 5.1.3]{Reiner} by setting $I=\left\{  1,2,3,\ldots\right\}
$). That said, our proof can easily be extended to work in this greater
generality.} or \cite[\S 11.26]{HazeWitt}). The general case is still
classical, but since an explicit proof is hard to locate in the literature,
let me sketch it here.

\begin{proof}
[Proof of Proposition \ref{prop.QSym.M*M*M}.]We begin by introducing notations:

\begin{itemize}
\item Let $\mathbb{N}^{k,\infty}$ denote the set of all matrices with $k$ rows
(labelled $1,2,\ldots,k$) and countably many columns (labelled $1,2,3,\ldots$)
whose entries all belong to $\mathbb{N}$.

\item Let $\mathbb{N}_{\operatorname*{fin}}^{k,\infty}$ denote the set of all
matrices in $\mathbb{N}^{k,\infty}$ which have only finitely many nonzero entries.

\item Let $\mathbb{N}^{\infty}$ denote the set of all infinite sequences
$\left(  a_{1},a_{2},a_{3},\ldots\right)  $ of elements of $\mathbb{N}$.

\item Let $\mathbb{N}_{\operatorname*{fin}}^{\infty}$ denote the set of all
sequences in $\mathbb{N}^{\infty}$ which have only finitely many nonzero entries.

\item For every $B\in\mathbb{N}_{\operatorname*{fin}}^{k,\infty}$ and
$i\in\left\{  1,2,\ldots,k\right\}  $, we let $B_{i,\bullet}\in\mathbb{N}%
_{\operatorname*{fin}}^{\infty}$ be the $i$-th row of $B$.

\item For every $B=\left(  b_{i,j}\right)  _{1\leq i\leq k,\ 1\leq j}%
\in\mathbb{N}_{\operatorname*{fin}}^{k,\infty}$, we let
$\operatorname*{column}B\in\mathbb{N}_{\operatorname*{fin}}^{\infty}$ be the
sequence whose $j$-th entry is $\sum_{i=1}^{k}b_{i,j}$ (that is, the sum of
all entries in the $j$-th column of $B$) for each $j$. (In other words,
$\operatorname*{column}B$ is the sum of all rows of $B$, regarded as vectors.)

\item We extend Definition \ref{def.QSym.secondcomult} \textbf{(b)} to the
case when $w\in\mathbb{N}_{\operatorname*{fin}}^{\infty}$: If $w\in
\mathbb{N}_{\operatorname*{fin}}^{\infty}$, then $w^{\operatorname*{red}}$
shall mean the composition obtained from $w$ by removing each entry that
equals $0$\ \ \ \ \footnote{Here is a more rigorous definition of
$w^{\operatorname*{red}}$: Let $w=\left(  w_{1},w_{2},w_{3},\ldots\right)  $.
Let $\mathcal{J}$ be the set of all positive integers $j$ such that $w_{j}%
\neq0$. Let $\left(  j_{1}<j_{2}<\cdots<j_{h}\right)  $ be the list of all
elements of $\mathcal{J}$, in increasing order. Then, $w^{\operatorname*{red}%
}$ is defined to be the composition $\left(  w_{j_{1}},w_{j_{2}}%
,\ldots,w_{j_{h}}\right)  $.
\par
This rigorous definition of $w^{\operatorname*{red}}$ has the additional
advantage of making sense in greater generality than \textquotedblleft remove
each entry that equals $0$\textquotedblright; namely, it still works when
$w\in\mathbb{N}_{\operatorname*{fin}}^{I}$ for some totally ordered set $I$.}.

\item For every $\beta=\left(  b_{1},b_{2},b_{3},\ldots\right)  \in
\mathbb{N}_{\operatorname*{fin}}^{\infty}$, we define a monomial
$\mathbf{x}^{\beta}$ in the indeterminates $x_{1},x_{2},x_{3},\ldots$ by%
\[
\mathbf{x}^{\beta}=x_{1}^{b_{1}}x_{2}^{b_{2}}x_{3}^{b_{3}}\cdots.
\]

\end{itemize}

Then, it is easy to see that%
\begin{equation}
M_{\alpha}=\sum_{\substack{\beta\in\mathbb{N}_{\operatorname*{fin}}^{\infty
};\\\beta^{\operatorname*{red}}=\alpha}}\mathbf{x}^{\beta}%
\ \ \ \ \ \ \ \ \ \ \text{for every composition }\alpha.
\label{pf.prop.QSym.M*M*M.1}%
\end{equation}
Now,%
\begin{align}
&  M_{\alpha_{1}}M_{\alpha_{2}}\cdots M_{\alpha_{k}}\nonumber\\
&  =\prod_{g=1}^{k}\underbrace{M_{\alpha_{g}}}_{\substack{=\sum
_{\substack{\beta\in\mathbb{N}_{\operatorname*{fin}}^{\infty};\\\beta
^{\operatorname*{red}}=\alpha_{g}}}\mathbf{x}^{\beta}\\\text{(by
(\ref{pf.prop.QSym.M*M*M.1}))}}}=\prod_{g=1}^{k}\ \ \sum_{\substack{\beta
\in\mathbb{N}_{\operatorname*{fin}}^{\infty};\\\beta^{\operatorname*{red}%
}=\alpha_{g}}}\mathbf{x}^{\beta}\nonumber\\
&  =\sum_{\substack{\left(  \beta_{1},\beta_{2},\ldots,\beta_{k}\right)
\in\left(  \mathbb{N}_{\operatorname*{fin}}^{\infty}\right)  ^{k};\\\left(
\beta_{g}\right)  ^{\operatorname*{red}}=\alpha_{g}\text{ for each }%
g}}\mathbf{x}^{\beta_{1}}\mathbf{x}^{\beta_{2}}\cdots\mathbf{x}^{\beta_{k}%
}\ \ \ \ \ \ \ \ \ \ \left(  \text{by the product rule}\right) \nonumber\\
&  =\sum_{\substack{B\in\mathbb{N}_{\operatorname*{fin}}^{k,\infty};\\\left(
B_{g,\bullet}\right)  ^{\operatorname*{red}}=\alpha_{g}\text{ for each }%
g}}\underbrace{\mathbf{x}^{B_{1,\bullet}}\mathbf{x}^{B_{2,\bullet}}%
\cdots\mathbf{x}^{B_{k,\bullet}}}_{\substack{=\mathbf{x}%
^{\operatorname*{column}B}\\\text{(since }\operatorname*{column}B\text{ is the
sum of the}\\\text{rows of }B\text{ (as vectors))}}}\nonumber\\
&  \ \ \ \ \ \ \ \ \ \ \ \ \ \ \ \ \ \ \ \ \left(
\begin{array}
[c]{c}%
\text{here, we have substituted }\left(  B_{1,\bullet},B_{2,\bullet}%
,\ldots,B_{k,\bullet}\right)  \text{ for}\\
\left(  \beta_{1},\beta_{2},\ldots,\beta_{k}\right)  \text{ in the sum, since
the map}\\
\mathbb{N}_{\operatorname*{fin}}^{k,\infty}\rightarrow\left(  \mathbb{N}%
_{\operatorname*{fin}}^{\infty}\right)  ^{k},\ B\mapsto\left(  B_{1,\bullet
},B_{2,\bullet},\ldots,B_{k,\bullet}\right)  \text{is a bijection}%
\end{array}
\right) \nonumber\\
&  =\sum_{\substack{B\in\mathbb{N}_{\operatorname*{fin}}^{k,\infty};\\\left(
B_{g,\bullet}\right)  ^{\operatorname*{red}}=\alpha_{g}\text{ for each }%
g}}\mathbf{x}^{\operatorname*{column}B}. \label{pf.prop.QSym.M*M*M.2}%
\end{align}

Now, let us introduce one more notation: For every matrix $B\in\mathbb{N}%
_{\operatorname*{fin}}^{k,\infty}$, let $B^{\operatorname*{Cred}}$ be the
matrix obtained from $B$ by removing all zero columns (i.e., all columns
containing only zeroes)\footnote{Again, we can define $B^{\operatorname*{Cred}%
}$ more rigorously as follows: Let $\mathcal{J}$ be the set of all positive
integers $j$ such that the $j$-th column of $B$ is nonzero. Let $\left(
j_{1}<j_{2}<\cdots<j_{h}\right)  $ be the list of all elements of
$\mathcal{J}$, in increasing order. Then, $B^{\operatorname*{Cred}}$ is
defined to be the $k\times h$-matrix whose columns (from left to right) are
the $j_{1}$-th column of $B$, the $j_{2}$-nd column of $B$, $\ldots$, the
$j_{h}$-th column of $B$.}. It is easy to see that $B^{\operatorname*{Cred}%
}\in\mathbb{N}_{\operatorname*{Cred}}^{k,\bullet}$ for every $B\in
\mathbb{N}_{\operatorname*{fin}}^{k,\infty}$. Moreover, every $B\in
\mathbb{N}_{\operatorname*{fin}}^{k,\infty}$ satisfies the following fact: If
$A=B^{\operatorname*{Cred}}$, then
\begin{equation}
\left(  B_{g,\bullet}\right)  ^{\operatorname*{red}}=\left(  A_{g,\bullet
}\right)  ^{\operatorname*{red}}\text{ for each }g
\label{pf.prop.QSym.M*M*M.Cred-props}%
\end{equation}
(indeed, $A_{g,\bullet}$ is obtained from $B_{g,\bullet}$ by removing some
zero entries).

Now, (\ref{pf.prop.QSym.M*M*M.2}) becomes%
\begin{align}
M_{\alpha_{1}}M_{\alpha_{2}}\cdots M_{\alpha_{k}}  &  =\sum_{\substack{B\in
\mathbb{N}_{\operatorname*{fin}}^{k,\infty};\\\left(  B_{g,\bullet}\right)
^{\operatorname*{red}}=\alpha_{g}\text{ for each }g}}\mathbf{x}%
^{\operatorname*{column}B}\nonumber\\
&  =\sum_{A\in\mathbb{N}_{\operatorname*{Cred}}^{k,\bullet}}\underbrace{\sum
_{\substack{B\in\mathbb{N}_{\operatorname*{fin}}^{k,\infty};\\\left(
B_{g,\bullet}\right)  ^{\operatorname*{red}}=\alpha_{g}\text{ for each
}g;\\B^{\operatorname*{Cred}}=A}}}_{\substack{=\sum_{\substack{B\in
\mathbb{N}_{\operatorname*{fin}}^{k,\infty};\\B^{\operatorname*{Cred}%
}=A;\\\left(  B_{g,\bullet}\right)  ^{\operatorname*{red}}=\alpha_{g}\text{
for each }g}}=\sum_{\substack{B\in\mathbb{N}_{\operatorname*{fin}}^{k,\infty
};\\B^{\operatorname*{Cred}}=A;\\\left(  A_{g,\bullet}\right)
^{\operatorname*{red}}=\alpha_{g}\text{ for each }g}}\\\text{(because if
}B^{\operatorname*{Cred}}=A\text{, then }\left(  B_{g,\bullet}\right)
^{\operatorname*{red}}=\left(  A_{g,\bullet}\right)  ^{\operatorname*{red}%
}\\\text{for each }g\text{ (because of (\ref{pf.prop.QSym.M*M*M.Cred-props}%
)))}}}\mathbf{x}^{\operatorname*{column}B}\nonumber\\
&  \ \ \ \ \ \ \ \ \ \ \ \ \ \ \ \ \ \ \ \ \left(  \text{since }%
B^{\operatorname*{Cred}}\in\mathbb{N}_{\operatorname*{Cred}}^{k,\bullet}\text{
for each }B\in\mathbb{N}_{\operatorname*{fin}}^{k,\infty}\right) \nonumber\\
&  =\underbrace{\sum_{A\in\mathbb{N}_{\operatorname*{Cred}}^{k,\bullet}}%
\sum_{\substack{B\in\mathbb{N}_{\operatorname*{fin}}^{k,\infty}%
;\\B^{\operatorname*{Cred}}=A;\\\left(  A_{g,\bullet}\right)
^{\operatorname*{red}}=\alpha_{g}\text{ for each }g}}}_{=\sum_{\substack{A\in
\mathbb{N}_{\operatorname*{Cred}}^{k,\bullet};\\\left(  A_{g,\bullet}\right)
^{\operatorname*{red}}=\alpha_{g}\text{ for each }g}}\ \ \sum_{\substack{B\in
\mathbb{N}_{\operatorname*{fin}}^{k,\infty};\\B^{\operatorname*{Cred}}=A}%
}}\mathbf{x}^{\operatorname*{column}B}\nonumber\\
&  =\sum_{\substack{A\in\mathbb{N}_{\operatorname*{Cred}}^{k,\bullet
};\\\left(  A_{g,\bullet}\right)  ^{\operatorname*{red}}=\alpha_{g}\text{ for
each }g}}\ \ \sum_{\substack{B\in\mathbb{N}_{\operatorname*{fin}}^{k,\infty
};\\B^{\operatorname*{Cred}}=A}}\mathbf{x}^{\operatorname*{column}B}.
\label{pf.prop.QSym.M*M*M.5}%
\end{align}

However, for every matrix $A\in\mathbb{N}_{\operatorname*{Cred}}^{k,\bullet}$,
we have%
\begin{equation}
\sum_{\substack{B\in\mathbb{N}_{\operatorname*{fin}}^{k,\infty}%
;\\B^{\operatorname*{Cred}}=A}}\mathbf{x}^{\operatorname*{column}%
B}=M_{\operatorname*{column}A}. \label{pf.prop.QSym.M*M*M.3}%
\end{equation}

\textit{Proof of (\ref{pf.prop.QSym.M*M*M.3}):} Let $A\in\mathbb{N}%
_{\operatorname*{Cred}}^{k,\bullet}$. We need to prove
(\ref{pf.prop.QSym.M*M*M.3}).

For every $B\in\mathbb{N}_{\operatorname*{fin}}^{k,\infty}$, we have $\left(
\operatorname*{column}B\right)  ^{\operatorname*{red}}=\operatorname*{column}%
\left(  B^{\operatorname*{Cred}}\right)  $ (because first taking the sum of
each column of $B$ and then removing the zeroes among these sums results in
the same list as first removing the zero columns of $B$ and then taking the
sum of each remaining column). Thus, for every $B\in\mathbb{N}%
_{\operatorname*{fin}}^{k,\infty}$ satisfying $B^{\operatorname*{Cred}}=A$, we
have $\operatorname*{column}B\in\mathbb{N}_{\operatorname*{fin}}^{\infty}$ and
$\left(  \operatorname*{column}B\right)  ^{\operatorname*{red}}%
=\operatorname*{column}\underbrace{\left(  B^{\operatorname*{Cred}}\right)
}_{=A}=\operatorname*{column}A$. Hence, the map%
\begin{align}
\left\{  B\in\mathbb{N}_{\operatorname*{fin}}^{k,\infty}\ \mid
\ B^{\operatorname*{Cred}}=A\right\}   &  \rightarrow\left\{  \beta
\in\mathbb{N}_{\operatorname*{fin}}^{\infty}\ \mid\ \beta^{\operatorname*{red}%
}=\operatorname*{column}A\right\}  ,\nonumber\\
B  &  \mapsto\operatorname*{column}B \label{pf.prop.QSym.M*M*M.3.pf.1}%
\end{align}
is well-defined.

On the other hand, if $\beta\in\mathbb{N}_{\operatorname*{fin}}^{\infty}$
satisfies $\beta^{\operatorname*{red}}=\operatorname*{column}A$, then there
exists a unique $B\in\mathbb{N}_{\operatorname*{fin}}^{k,\infty}$ satisfying
$B^{\operatorname*{Cred}}=A$ and $\operatorname*{column}B=\beta$%
\ \ \ \ \footnote{Namely, this $B$ can be computed as follows: Write the
sequence $\beta$ in the form $\beta=\left(  \beta_{1},\beta_{2},\beta
_{3},\ldots\right)  $. Let $\left(  i_{1}<i_{2}<\cdots<i_{h}\right)  $ be the
list of all $c$ satisfying $\beta_{c}\neq0$, written in increasing order.
Then, $B$ shall be the matrix whose $i_{1}$-st, $i_{2}$-nd, $\ldots$, $i_{h}%
$-th columns are the columns of $A$ (from left to right), whereas all its
other columns are $0$.
\par
Let us briefly sketch a proof of the fact that this $B$ is indeed an element
of $\mathbb{N}_{\operatorname*{fin}}^{k,\infty}$ satisfying
$B^{\operatorname*{Cred}}=A$ and $\operatorname*{column}B=\beta$:
\par
Indeed, it is clear that $B\in\mathbb{N}_{\operatorname*{fin}}^{k,\infty}$.
\par
We shall now show that%
\begin{equation}
\left(  \text{the }j\text{-th entry of }\operatorname*{column}B\right)
=\beta_{j} \label{pf.prop.QSym.M*M*M.3.pf.1.bij.fn1.1}%
\end{equation}
for every $j\in\left\{  1,2,3,\ldots\right\}  $.
\par
\textit{Proof of (\ref{pf.prop.QSym.M*M*M.3.pf.1.bij.fn1.1}):} Let
$j\in\left\{  1,2,3,\ldots\right\}  $. We must prove
(\ref{pf.prop.QSym.M*M*M.3.pf.1.bij.fn1.1}). We are in one of the following
two cases:
\par
\textit{Case 1:} We have $j\in\left\{  i_{1},i_{2},\ldots,i_{h}\right\}  $.
\par
\textit{Case 2:} We have $j\notin\left\{  i_{1},i_{2},\ldots,i_{h}\right\}  $.
\par
Let us first consider Case 1. In this case, we have $j\in\left\{  i_{1}%
,i_{2},\ldots,i_{h}\right\}  $. Hence, there exists a $g\in\left\{
1,2,\ldots,h\right\}  $ such that $j=i_{g}$. Consider this $g$. Now,
\begin{align*}
&  \left(  \text{the }j\text{-th entry of }\operatorname*{column}B\right) \\
&  =\left(  \text{the sum of the entries of the }\underbrace{j}_{=i_{g}%
}\text{-th column of }B\right) \\
&  =\left(  \text{the sum of the entries of }\underbrace{\text{the }%
i_{g}\text{-th column of }B}_{\substack{=\left(  \text{the }g\text{-th column
of }A\right)  \\\text{(by the definition of }B\text{)}}}\right) \\
&  =\left(  \text{the sum of the entries of the }g\text{-th column of
}A\right) \\
&  =\left(  \text{the }g\text{-th entry of }\underbrace{\operatorname*{column}%
A}_{\substack{=\beta^{\operatorname*{red}}=\left(  \beta_{i_{1}},\beta_{i_{2}%
},\ldots,\beta_{i_{h}}\right)  \\\text{(by the definition of }\beta
^{\operatorname*{red}}\text{)}}}\right) \\
&  =\left(  \text{the }g\text{-th entry of }\left(  \beta_{i_{1}},\beta
_{i_{2}},\ldots,\beta_{i_{h}}\right)  \right)  =\beta_{i_{g}}=\beta
_{j}\ \ \ \ \ \ \ \ \ \ \left(  \text{since }i_{g}=j\right)  .
\end{align*}
Thus, (\ref{pf.prop.QSym.M*M*M.3.pf.1.bij.fn1.1}) is proven in Case 1.
\par
Let us now consider Case 2. In this case, we have $j\notin\left\{  i_{1}%
,i_{2},\ldots,i_{h}\right\}  $. Hence, $j$ does not belong to the list
$\left(  i_{1}<i_{2}<\cdots<i_{h}\right)  $. In other words, $j$ does not
belong to the list of all $c\in\left\{  1,2,3,\ldots\right\}  $ satisfying
$\beta_{c}\neq0$ (since this list is $\left(  i_{1}<i_{2}<\cdots<i_{h}\right)
$). Hence, $\beta_{j}=0$.
\par
Recall that $j\notin\left\{  i_{1},i_{2},\ldots,i_{h}\right\}  $. Hence, the
$j$-th column of $B$ is the $0$ vector (by the definition of $B$). Now,
\begin{align*}
&  \left(  \text{the }j\text{-th entry of }\operatorname*{column}B\right) \\
&  =\left(  \text{the sum of the entries of }\underbrace{\text{the }j\text{-th
column of }B}_{=\left(  \text{the }0\text{ vector}\right)  }\right) \\
&  =\left(  \text{the sum of the entries of the }0\text{ vector}\right) \\
&  =0=\beta_{j}.
\end{align*}
Thus, (\ref{pf.prop.QSym.M*M*M.3.pf.1.bij.fn1.1}) is proven in Case 2.
\par
Hence, (\ref{pf.prop.QSym.M*M*M.3.pf.1.bij.fn1.1}) is proven in both Cases 1
and 2. Thus, the proof of (\ref{pf.prop.QSym.M*M*M.3.pf.1.bij.fn1.1}) is
complete.
\par
Now, from (\ref{pf.prop.QSym.M*M*M.3.pf.1.bij.fn1.1}), we immediately obtain
$\operatorname*{column}B=\left(  \beta_{1},\beta_{2},\beta_{3},\ldots\right)
=\beta$.
\par
It remains to prove that $B^{\operatorname*{Cred}}=A$. This can be done as
follows: We have $A\in\mathbb{N}_{\operatorname*{Cred}}^{k,\bullet}$; thus,
the matrix $A$ is column-reduced. Hence, no column of $A$ is the zero vector.
Therefore, none of the $i_{1}$-st, $i_{2}$-nd, $\ldots$, $i_{h}$-th columns of
$B$ is the zero vector (since these columns are the columns of $A$). On the
other hand, each of the remaining columns of $B$ is the zero vector (due to
the definition of $B$). Thus, the set of all positive integers $j$ such that
the $j$-th column of $B$ is nonzero is precisely $\left\{  i_{1},i_{2}%
,\ldots,i_{h}\right\}  $. The list of all elements of this set, in increasing
order, is $\left(  i_{1}<i_{2}<\cdots<i_{h}\right)  $. Hence, the definition
of $B^{\operatorname*{Cred}}$ shows that $B^{\operatorname*{Cred}}$ is the
$k\times h$-matrix whose columns (from left to right) are the $i_{1}$-th
column of $B$, the $i_{2}$-nd column of $B$, $\ldots$, the $i_{h}$-th column
of $B$. Since these columns are precisely the columns of $A$, this entails
that $B^{\operatorname*{Cred}}$ is the matrix $A$. In other words,
$B^{\operatorname*{Cred}}=A$.
\par
Thus, we have proven that $B$ is an element of $\mathbb{N}%
_{\operatorname*{fin}}^{k,\infty}$ satisfying $B^{\operatorname*{Cred}}=A$ and
$\operatorname*{column}B=\beta$. It is fairly easy to see that it is the only
such element (because the condition $\operatorname*{column}B=\beta$ determines
which columns of $B$ are nonzero, whereas the condition
$B^{\operatorname*{Cred}}=A$ determines the precise values of these
columns).}. In other words, the map (\ref{pf.prop.QSym.M*M*M.3.pf.1}) is
bijective. Thus, we can substitute $\beta$ for $\operatorname*{column}B$ in
the sum $\sum_{\substack{B\in\mathbb{N}_{\operatorname*{fin}}^{k,\infty
};\\B^{\operatorname*{Cred}}=A}}\mathbf{x}^{\operatorname*{column}B}$, and
obtain%
\[
\sum_{\substack{B\in\mathbb{N}_{\operatorname*{fin}}^{k,\infty}%
;\\B^{\operatorname*{Cred}}=A}}\mathbf{x}^{\operatorname*{column}B}%
=\sum_{\substack{\beta\in\mathbb{N}_{\operatorname*{fin}}^{\infty}%
;\\\beta^{\operatorname*{red}}=\operatorname*{column}A}}\mathbf{x}^{\beta
}=M_{\operatorname*{column}A}%
\]
(by (\ref{pf.prop.QSym.M*M*M.1}), applied to $\alpha=\operatorname*{column}%
A$). This proves (\ref{pf.prop.QSym.M*M*M.3}).

Now, (\ref{pf.prop.QSym.M*M*M.5}) becomes%
\begin{align*}
M_{\alpha_{1}}M_{\alpha_{2}}\cdots M_{\alpha_{k}}  &  =\sum_{\substack{A\in
\mathbb{N}_{\operatorname*{Cred}}^{k,\bullet};\\\left(  A_{g,\bullet}\right)
^{\operatorname*{red}}=\alpha_{g}\text{ for each }g}}\ \ \underbrace{\sum
_{\substack{B\in\mathbb{N}_{\operatorname*{fin}}^{k,\infty}%
;\\B^{\operatorname*{Cred}}=A}}\mathbf{x}^{\operatorname*{column}B}%
}_{\substack{=M_{\operatorname*{column}A}\\\text{(by
(\ref{pf.prop.QSym.M*M*M.3}))}}}\\
&  =\sum_{\substack{A\in\mathbb{N}_{\operatorname*{Cred}}^{k,\bullet
};\\\left(  A_{g,\bullet}\right)  ^{\operatorname*{red}}=\alpha_{g}\text{ for
each }g}}M_{\operatorname*{column}A}.
\end{align*}
This proves Proposition \ref{prop.QSym.M*M*M}.
\end{proof}

We need one more piece of notation:

\begin{definition}
\label{def.Comp.monoid}We define a (multiplicative) monoid structure on the
set $\operatorname*{Comp}$ as follows: If $\alpha=\left(  a_{1},a_{2}%
,\ldots,a_{n}\right)  $ and $\beta=\left(  b_{1},b_{2},\ldots,b_{m}\right)  $
are two compositions, then we set $\alpha\beta=\left(  a_{1},a_{2}%
,\ldots,a_{n},b_{1},b_{2},\ldots,b_{m}\right)  $. Thus, $\operatorname*{Comp}$
becomes a monoid with neutral element $\varnothing=\left(  {}\right)  $ (the
empty composition). (This monoid is actually the free monoid on the set
$\left\{  1,2,3,\ldots\right\}  $.)
\end{definition}

\begin{proposition}
\label{prop.DeltaMalpha}Let $\gamma\in\operatorname*{Comp}$ and $k\in
\mathbb{N}$. Then,%
\[
\Delta^{\left(  k-1\right)  }M_{\gamma}=\sum_{\substack{\left(  \gamma
_{1},\gamma_{2},\ldots,\gamma_{k}\right)  \in\operatorname*{Comp}%
\nolimits^{k};\\\gamma_{1}\gamma_{2}\cdots\gamma_{k}=\gamma}}M_{\gamma_{1}%
}\otimes M_{\gamma_{2}}\otimes\cdots\otimes M_{\gamma_{k}}.
\]

\end{proposition}

\begin{proof}
[Proof of Proposition \ref{prop.DeltaMalpha} (sketched).]We can rewrite
(\ref{eq.DeltaM}) as follows:%
\begin{equation}
\Delta M_{\beta}=\sum_{\substack{\left(  \sigma,\tau\right)  \in
\operatorname*{Comp}\times\operatorname*{Comp};\\\sigma\tau=\beta}}M_{\sigma
}\otimes M_{\tau}\ \ \ \ \ \ \ \ \ \ \text{for every }\beta\in
\operatorname*{Comp}. \label{pf.prop.DeltaMalpha.eq.DeltaM}%
\end{equation}
Proposition \ref{prop.DeltaMalpha} can easily be proven by induction using
(\ref{pf.prop.DeltaMalpha.eq.DeltaM}).
\end{proof}

\begin{proof}
[Proof of Proposition \ref{prop.QSym.secondcomult.alt}.]Fix $\alpha
\in\operatorname*{Comp}$ and $\gamma\in\operatorname*{Comp}$. Write $\alpha$
in the form $\alpha=\left(  a_{1},a_{2},\ldots,a_{k}\right)  $; thus, a
$\mathbf{k}$-linear map $\xi_{\alpha}:\operatorname*{QSym}%
\nolimits_{\mathbf{k}}\rightarrow\operatorname*{QSym}\nolimits_{\mathbf{k}}$
is defined (as in Definition \ref{def.bernstein}, applied to
$H=\operatorname*{QSym}\nolimits_{\mathbf{k}}$).

We shall prove that
\begin{equation}
\xi_{\alpha}\left(  M_{\gamma}\right)  =\sum_{\substack{A\in\mathbb{N}%
_{\operatorname*{red}}^{\bullet,\bullet}\text{;}\\\left(  \operatorname*{read}%
A\right)  ^{\operatorname*{red}}=\gamma\text{;}\\\operatorname*{row}A=\alpha
}}M_{\operatorname*{column}A}. \label{pf.prop.QSym.secondcomult.alt.4}%
\end{equation}

\textit{Proof of (\ref{pf.prop.QSym.secondcomult.alt.4}):} The definition of
$\xi_{\alpha}$ yields $\xi_{\alpha}=m^{\left(  k-1\right)  }\circ\pi_{\alpha
}\circ\Delta^{\left(  k-1\right)  }$. Thus,%
\begin{align*}
\xi_{\alpha}\left(  M_{\gamma}\right)   &  =\left(  m^{\left(  k-1\right)
}\circ\pi_{\alpha}\circ\Delta^{\left(  k-1\right)  }\right)  \left(
M_{\gamma}\right) \\
&  =\left(  m^{\left(  k-1\right)  }\circ\pi_{\alpha}\right)
\underbrace{\left(  \Delta^{\left(  k-1\right)  }\left(  M_{\gamma}\right)
\right)  }_{\substack{=\sum_{\substack{\left(  \gamma_{1},\gamma_{2}%
,\ldots,\gamma_{k}\right)  \in\operatorname*{Comp}\nolimits^{k};\\\gamma
_{1}\gamma_{2}\cdots\gamma_{k}=\gamma}}M_{\gamma_{1}}\otimes M_{\gamma_{2}%
}\otimes\cdots\otimes M_{\gamma_{k}}\\\text{(by Proposition
\ref{prop.DeltaMalpha})}}}\\
&  =\left(  m^{\left(  k-1\right)  }\circ\pi_{\alpha}\right)  \left(
\sum_{\substack{\left(  \gamma_{1},\gamma_{2},\ldots,\gamma_{k}\right)
\in\operatorname*{Comp}\nolimits^{k};\\\gamma_{1}\gamma_{2}\cdots\gamma
_{k}=\gamma}}M_{\gamma_{1}}\otimes M_{\gamma_{2}}\otimes\cdots\otimes
M_{\gamma_{k}}\right) \\
&  =\sum_{\substack{\left(  \gamma_{1},\gamma_{2},\ldots,\gamma_{k}\right)
\in\operatorname*{Comp}\nolimits^{k};\\\gamma_{1}\gamma_{2}\cdots\gamma
_{k}=\gamma}}m^{\left(  k-1\right)  }\underbrace{\left(  \pi_{\alpha}\left(
M_{\gamma_{1}}\otimes M_{\gamma_{2}}\otimes\cdots\otimes M_{\gamma_{k}%
}\right)  \right)  }_{\substack{=\pi_{a_{1}}\left(  M_{\gamma_{1}}\right)
\otimes\pi_{a_{2}}\left(  M_{\gamma_{2}}\right)  \otimes\cdots\otimes
\pi_{a_{k}}\left(  M_{\gamma_{k}}\right)  \\\text{(since }\pi_{\alpha}%
=\pi_{a_{1}}\otimes\pi_{a_{2}}\otimes\cdots\otimes\pi_{a_{k}}\text{)}}}\\
&  =\sum_{\substack{\left(  \gamma_{1},\gamma_{2},\ldots,\gamma_{k}\right)
\in\operatorname*{Comp}\nolimits^{k};\\\gamma_{1}\gamma_{2}\cdots\gamma
_{k}=\gamma}}\underbrace{m^{\left(  k-1\right)  }\left(  \pi_{a_{1}}\left(
M_{\gamma_{1}}\right)  \otimes\pi_{a_{2}}\left(  M_{\gamma_{2}}\right)
\otimes\cdots\otimes\pi_{a_{k}}\left(  M_{\gamma_{k}}\right)  \right)  }%
_{=\pi_{a_{1}}\left(  M_{\gamma_{1}}\right)  \cdot\pi_{a_{2}}\left(
M_{\gamma_{2}}\right)  \cdot\cdots\cdot\pi_{a_{k}}\left(  M_{\gamma_{k}%
}\right)  =\prod_{g=1}^{k}\pi_{a_{g}}\left(  M_{\gamma_{g}}\right)  }\\
&  =\sum_{\substack{\left(  \gamma_{1},\gamma_{2},\ldots,\gamma_{k}\right)
\in\operatorname*{Comp}\nolimits^{k};\\\gamma_{1}\gamma_{2}\cdots\gamma
_{k}=\gamma}}\ \ \prod_{g=1}^{k}\underbrace{\pi_{a_{g}}\left(  M_{\gamma_{g}%
}\right)  }_{\substack{=%
\begin{cases}
M_{\gamma_{g}}, & \text{if }\left\vert \gamma_{g}\right\vert =a_{g};\\
0, & \text{if }\left\vert \gamma_{g}\right\vert \neq a_{g}%
\end{cases}
\\\text{(since the power series }M_{\gamma_{g}}\text{ is}\\\text{homogeneous
of degree }\left\vert \gamma_{g}\right\vert \text{)}}}
\end{align*}%
\begin{align*}
&  =\sum_{\substack{\left(  \gamma_{1},\gamma_{2},\ldots,\gamma_{k}\right)
\in\operatorname*{Comp}\nolimits^{k};\\\gamma_{1}\gamma_{2}\cdots\gamma
_{k}=\gamma}}\underbrace{\prod_{g=1}^{k}%
\begin{cases}
M_{\gamma_{g}}, & \text{if }\left\vert \gamma_{g}\right\vert =a_{g};\\
0, & \text{if }\left\vert \gamma_{g}\right\vert \neq a_{g}%
\end{cases}
}_{=%
\begin{cases}
\prod_{g=1}^{k}M_{\gamma_{g}}, & \text{if }\left\vert \gamma_{g}\right\vert
=a_{g}\text{ for all }g;\\
0, & \text{otherwise}%
\end{cases}
}\\
&  =\sum_{\substack{\left(  \gamma_{1},\gamma_{2},\ldots,\gamma_{k}\right)
\in\operatorname*{Comp}\nolimits^{k};\\\gamma_{1}\gamma_{2}\cdots\gamma
_{k}=\gamma}}%
\begin{cases}
\prod_{g=1}^{k}M_{\gamma_{g}}, & \text{if }\left\vert \gamma_{g}\right\vert
=a_{g}\text{ for all }g;\\
0, & \text{otherwise}%
\end{cases}
\\
&  =\sum_{\substack{\left(  \gamma_{1},\gamma_{2},\ldots,\gamma_{k}\right)
\in\operatorname*{Comp}\nolimits^{k};\\\gamma_{1}\gamma_{2}\cdots\gamma
_{k}=\gamma;\\\left\vert \gamma_{g}\right\vert =a_{g}\text{ for all }%
g}}\underbrace{\prod_{g=1}^{k}M_{\gamma_{g}}}_{\substack{=M_{\gamma_{1}%
}M_{\gamma_{2}}\cdots M_{\gamma_{k}}\\=\sum_{\substack{A\in\mathbb{N}%
_{\operatorname*{Cred}}^{k,\bullet};\\\left(  A_{g,\bullet}\right)
^{\operatorname*{red}}=\gamma_{g}\text{ for each }g}}M_{\operatorname*{column}%
A}\\\text{(by Proposition \ref{prop.QSym.M*M*M}, applied to }\alpha_{i}%
=\gamma_{i}\text{)}}}\\
&  \ \ \ \ \ \ \ \ \ \ \ \ \ \ \ \ \ \ \ \ \left(  \text{here, we have
filtered out the zero addends}\right)
\end{align*}%
\begin{align}
&  =\underbrace{\sum_{\substack{\left(  \gamma_{1},\gamma_{2},\ldots
,\gamma_{k}\right)  \in\operatorname*{Comp}\nolimits^{k};\\\gamma_{1}%
\gamma_{2}\cdots\gamma_{k}=\gamma;\\\left\vert \gamma_{g}\right\vert
=a_{g}\text{ for all }g}}\ \ \sum_{\substack{A\in\mathbb{N}%
_{\operatorname*{Cred}}^{k,\bullet};\\\left(  A_{g,\bullet}\right)
^{\operatorname*{red}}=\gamma_{g}\text{ for each }g}}}_{=\sum_{\left(
\gamma_{1},\gamma_{2},\ldots,\gamma_{k}\right)  \in\operatorname*{Comp}%
\nolimits^{k}}\ \ \sum_{\substack{A\in\mathbb{N}_{\operatorname*{Cred}%
}^{k,\bullet};\\\left(  A_{g,\bullet}\right)  ^{\operatorname*{red}}%
=\gamma_{g}\text{ for each }g;\\\gamma_{1}\gamma_{2}\cdots\gamma_{k}%
=\gamma;\\\left\vert \gamma_{g}\right\vert =a_{g}\text{ for all }g}%
}}M_{\operatorname*{column}A}\nonumber\\
&  =\sum_{\left(  \gamma_{1},\gamma_{2},\ldots,\gamma_{k}\right)
\in\operatorname*{Comp}\nolimits^{k}}\underbrace{\sum_{\substack{A\in
\mathbb{N}_{\operatorname*{Cred}}^{k,\bullet};\\\left(  A_{g,\bullet}\right)
^{\operatorname*{red}}=\gamma_{g}\text{ for each }g;\\\gamma_{1}\gamma
_{2}\cdots\gamma_{k}=\gamma;\\\left\vert \gamma_{g}\right\vert =a_{g}\text{
for all }g}}}_{\substack{=\sum_{\substack{A\in\mathbb{N}_{\operatorname*{Cred}%
}^{k,\bullet};\\\left(  A_{g,\bullet}\right)  ^{\operatorname*{red}}%
=\gamma_{g}\text{ for each }g;\\\left(  A_{1,\bullet}\right)
^{\operatorname*{red}}\left(  A_{2,\bullet}\right)  ^{\operatorname*{red}%
}\cdots\left(  A_{k,\bullet}\right)  ^{\operatorname*{red}}=\gamma
;\\\left\vert \left(  A_{g,\bullet}\right)  ^{\operatorname*{red}}\right\vert
=a_{g}\text{ for all }g}}\\\text{(here, we have replaced every }\gamma
_{g}\\\text{in the conditions }\left(  \gamma_{1}\gamma_{2}\cdots\gamma
_{k}=\gamma\right)  \\\text{and }\left(  \left\vert \gamma_{g}\right\vert
=a_{g}\text{ for all }g\right)  \text{ by the}\\\text{corresponding }\left(
A_{g,\bullet}\right)  ^{\operatorname*{red}}\text{, because}\\\text{of the
condition that }\left(  A_{g,\bullet}\right)  ^{\operatorname*{red}}%
=\gamma_{g}\text{)}}}M_{\operatorname*{column}A}\nonumber\\
&  =\underbrace{\sum_{\left(  \gamma_{1},\gamma_{2},\ldots,\gamma_{k}\right)
\in\operatorname*{Comp}\nolimits^{k}}\sum_{\substack{A\in\mathbb{N}%
_{\operatorname*{Cred}}^{k,\bullet};\\\left(  A_{g,\bullet}\right)
^{\operatorname*{red}}=\gamma_{g}\text{ for each }g;\\\left(  A_{1,\bullet
}\right)  ^{\operatorname*{red}}\left(  A_{2,\bullet}\right)
^{\operatorname*{red}}\cdots\left(  A_{k,\bullet}\right)
^{\operatorname*{red}}=\gamma;\\\left\vert \left(  A_{g,\bullet}\right)
^{\operatorname*{red}}\right\vert =a_{g}\text{ for all }g}}}_{=\sum
_{\substack{A\in\mathbb{N}_{\operatorname*{Cred}}^{k,\bullet};\\\left(
A_{1,\bullet}\right)  ^{\operatorname*{red}}\left(  A_{2,\bullet}\right)
^{\operatorname*{red}}\cdots\left(  A_{k,\bullet}\right)
^{\operatorname*{red}}=\gamma;\\\left\vert \left(  A_{g,\bullet}\right)
^{\operatorname*{red}}\right\vert =a_{g}\text{ for all }g}}}%
M_{\operatorname*{column}A}\nonumber\\
&  =\sum_{\substack{A\in\mathbb{N}_{\operatorname*{Cred}}^{k,\bullet
};\\\left(  A_{1,\bullet}\right)  ^{\operatorname*{red}}\left(  A_{2,\bullet
}\right)  ^{\operatorname*{red}}\cdots\left(  A_{k,\bullet}\right)
^{\operatorname*{red}}=\gamma;\\\left\vert \left(  A_{g,\bullet}\right)
^{\operatorname*{red}}\right\vert =a_{g}\text{ for all }g}%
}M_{\operatorname*{column}A}. \label{pf.prop.QSym.secondcomult.alt.8}%
\end{align}

Now, we observe that every $A\in\mathbb{N}_{\operatorname*{Cred}}^{k,\bullet}$
satisfies
\begin{equation}
\left(  A_{1,\bullet}\right)  ^{\operatorname*{red}}\left(  A_{2,\bullet
}\right)  ^{\operatorname*{red}}\cdots\left(  A_{k,\bullet}\right)
^{\operatorname*{red}}=\left(  \operatorname*{read}A\right)
^{\operatorname*{red}} \label{pf.prop.QSym.secondcomult.alt.helper1}%
\end{equation}
\footnote{\textit{Proof of (\ref{pf.prop.QSym.secondcomult.alt.helper1}):} Let
$A\in\mathbb{N}_{\operatorname*{Cred}}^{k,\bullet}$.
\par
Let $\mathbb{N}^{\bullet}$ be the set of all finite lists of nonnegative
integers. Then, $\operatorname*{Comp}\subseteq\mathbb{N}^{\bullet}$. In
Definition \ref{def.Comp.monoid}, we have defined a monoid structure on the
set $\operatorname*{Comp}$. We can extend this monoid structure to the set
$\mathbb{N}^{\bullet}$ (by the same rule: namely, if $\alpha=\left(
a_{1},a_{2},\ldots,a_{n}\right)  $ and $\beta=\left(  b_{1},b_{2},\ldots
,b_{m}\right)  $, then $\alpha\beta=\left(  a_{1},a_{2},\ldots,a_{n}%
,b_{1},b_{2},\ldots,b_{m}\right)  $). (Of course, this monoid $\mathbb{N}%
^{\bullet}$ is just the free monoid on the set $\mathbb{N}$.) Using the latter
structure, we can rewrite the definition of $\operatorname*{read}A$ as
follows:%
\[
\operatorname*{read}A=A_{1,\bullet}A_{2,\bullet}\cdots A_{k,\bullet}.
\]
\par
Clearly, the map $\mathbb{N}^{\bullet}\rightarrow\operatorname*{Comp}%
,\ \beta\mapsto\beta^{\operatorname*{red}}$ is a monoid homomorphism. Thus,%
\[
\left(  A_{1,\bullet}\right)  ^{\operatorname*{red}}\left(  A_{2,\bullet
}\right)  ^{\operatorname*{red}}\cdots\left(  A_{k,\bullet}\right)
^{\operatorname*{red}}=\left(  \underbrace{A_{1,\bullet}A_{2,\bullet}\cdots
A_{k,\bullet}}_{=\operatorname*{read}A}\right)  ^{\operatorname*{red}}=\left(
\operatorname*{read}A\right)  ^{\operatorname*{red}}.
\]
This proves (\ref{pf.prop.QSym.secondcomult.alt.helper1}).}.

Also, for every $A\in\mathbb{N}_{\operatorname*{Cred}}^{k,\bullet}$, we
have the logical equivalence%
\begin{equation}
\left(  \left\vert \left(  A_{g,\bullet}\right)  ^{\operatorname*{red}%
}\right\vert =a_{g}\text{ for all }g\right)  \ \Longleftrightarrow\ \left(
\operatorname*{row}A=\alpha\right)
\label{pf.prop.QSym.secondcomult.alt.helper2}%
\end{equation}
\footnote{\textit{Proof of (\ref{pf.prop.QSym.secondcomult.alt.helper2}):} Let
$A\in\mathbb{N}_{\operatorname*{Cred}}^{k,\bullet}$. Then, every $g\in\left\{
1,2,\ldots,k\right\}  $ satisfies
\begin{align*}
\left\vert \left(  A_{g,\bullet}\right)  ^{\operatorname*{red}}\right\vert  &
=\left(  \text{sum of all entries of }\left(  A_{g,\bullet}\right)
^{\operatorname*{red}}\right)  =\left(  \text{sum of all nonzero entries in
}A_{g,\bullet}\right) \\
&  \ \ \ \ \ \ \ \ \ \ \left(  \text{by the definition of }\left(
A_{g,\bullet}\right)  ^{\operatorname*{red}}\right) \\
&  =\left(  \text{sum of all nonzero entries in the }g\text{-th row of
}A\right) \\
&  =\left(  \text{sum of all entries in the }g\text{-th row of }A\right)
=\left(  \text{the }g\text{-th entry of }\operatorname*{row}A\right)  .
\end{align*}
Hence, we have the following chain of equivalences:%
\begin{align*}
&  \ \left(  \underbrace{\left\vert \left(  A_{g,\bullet}\right)
^{\operatorname*{red}}\right\vert }_{=\left(  \text{the }g\text{-th entry of
}\operatorname*{row}A\right)  }=a_{g}\text{ for all }g\right) \\
&  \Longleftrightarrow\ \left(  \left(  \text{the }g\text{-th entry of
}\operatorname*{row}A\right)  =a_{g}\text{ for all }g\right) \\
&  \Longleftrightarrow\ \left(  \operatorname*{row}A=\underbrace{\left(
a_{1},a_{2},\ldots,a_{k}\right)  }_{=\alpha}\right)  =\left(
\operatorname*{row}A=\alpha\right)  .
\end{align*}
This proves (\ref{pf.prop.QSym.secondcomult.alt.helper2}).}.

Also, every $A\in\mathbb{N}_{\operatorname*{Cred}}^{k,\bullet}$ satisfying
$\operatorname*{row}A=\alpha$ belongs to $\mathbb{N}_{\operatorname*{red}%
}^{\bullet,\bullet}$\ \ \ \ \footnote{\textit{Proof.} Let $A\in\mathbb{N}%
_{\operatorname*{Cred}}^{k,\bullet}$ be such that $\operatorname*{row}%
A=\alpha$. We must show that $A\in\mathbb{N}_{\operatorname*{red}}%
^{\bullet,\bullet}$. The sequence $\operatorname*{row}A=\alpha$ is a
composition; hence, $A$ is row-reduced. Since $A$ is also column-reduced
(because $A\in\mathbb{N}_{\operatorname*{Cred}}^{k,\bullet}$), this shows that
$A$ is reduced. Hence, $A\in\mathbb{N}_{\operatorname*{red}}^{\bullet,\bullet
}$, qed.}. Conversely, every $A\in\mathbb{N}_{\operatorname*{red}}%
^{\bullet,\bullet}$ satisfying $\operatorname*{row}A=\alpha$ belongs to
$\mathbb{N}_{\operatorname*{Cred}}^{k,\bullet}$%
\ \ \ \ \footnote{\textit{Proof.} Let $A\in\mathbb{N}_{\operatorname*{red}%
}^{\bullet,\bullet}$ be such that $\operatorname*{row}A=\alpha$. We must show
that $A\in\mathbb{N}_{\operatorname*{Cred}}^{k,\bullet}$. The number of rows
of $A$ is clearly the length of the vector $\operatorname*{row}A$ (where the
\textquotedblleft length\textquotedblright\ of a vector just means its number
of entries). But this length is $k$ (since $\operatorname*{row}A=\alpha
=\left(  a_{1},a_{2},\ldots,a_{k}\right)  $). Therefore, the number of rows of
$A$ is $k$. Also, $A$ is reduced (since $A\in\mathbb{N}_{\operatorname*{red}%
}^{\bullet,\bullet}$) and therefore column-reduced. Hence, $A\in
\mathbb{N}_{\operatorname*{Cred}}^{k,\bullet}$ (since $A$ is column-reduced
and the number of rows of $A$ is $k$), qed.}. Combining these two
observations, we see that
\begin{equation}
\left(
\begin{array}
[c]{c}%
\text{the matrices }A\in\mathbb{N}_{\operatorname*{Cred}}^{k,\bullet}\text{
satisfying }\operatorname*{row}A=\alpha\\
\text{are precisely the matrices }A\in\mathbb{N}_{\operatorname*{red}%
}^{\bullet,\bullet}\text{ satisfying }\operatorname*{row}A=\alpha
\end{array}
\right)  . \label{pf.prop.QSym.secondcomult.alt.helper5}%
\end{equation}

Now, (\ref{pf.prop.QSym.secondcomult.alt.8}) becomes%
\begin{align*}
\xi_{\alpha}\left(  M_{\gamma}\right)   &  =\underbrace{\sum_{\substack{A\in
\mathbb{N}_{\operatorname*{Cred}}^{k,\bullet};\\\left(  A_{1,\bullet}\right)
^{\operatorname*{red}}\left(  A_{2,\bullet}\right)  ^{\operatorname*{red}%
}\cdots\left(  A_{k,\bullet}\right)  ^{\operatorname*{red}}=\gamma
;\\\left\vert \left(  A_{g,\bullet}\right)  ^{\operatorname*{red}}\right\vert
=a_{g}\text{ for all }g}}}_{\substack{=\sum_{\substack{A\in\mathbb{N}%
_{\operatorname*{Cred}}^{k,\bullet}\text{;}\\\left(  \operatorname*{read}%
A\right)  ^{\operatorname*{red}}=\gamma\text{;}\\\operatorname*{row}A=\alpha
}}\\\text{(by (\ref{pf.prop.QSym.secondcomult.alt.helper1}) and
(\ref{pf.prop.QSym.secondcomult.alt.helper2}))}}}M_{\operatorname*{column}A}\\
&  =\underbrace{\sum_{\substack{A\in\mathbb{N}_{\operatorname*{Cred}%
}^{k,\bullet}\text{;}\\\left(  \operatorname*{read}A\right)
^{\operatorname*{red}}=\gamma\text{;}\\\operatorname*{row}A=\alpha}%
}}_{\substack{=\sum_{\substack{A\in\mathbb{N}_{\operatorname*{red}}%
^{\bullet,\bullet}\text{;}\\\left(  \operatorname*{read}A\right)
^{\operatorname*{red}}=\gamma\text{;}\\\operatorname*{row}A=\alpha
}}\\\text{(by (\ref{pf.prop.QSym.secondcomult.alt.helper5}))}}%
}M_{\operatorname*{column}A}=\sum_{\substack{A\in\mathbb{N}%
_{\operatorname*{red}}^{\bullet,\bullet}\text{;}\\\left(  \operatorname*{read}%
A\right)  ^{\operatorname*{red}}=\gamma\text{;}\\\operatorname*{row}A=\alpha
}}M_{\operatorname*{column}A}.
\end{align*}
This proves (\ref{pf.prop.QSym.secondcomult.alt.4}).

Now, forget that we fixed $\alpha$ and $\gamma$. For every $\gamma
\in\operatorname*{Comp}$, we have%
\begin{align}
\beta_{\operatorname*{QSym}\nolimits_{\mathbf{k}}}\left(  M_{\gamma}\right)
&  =\sum_{\alpha\in\operatorname*{Comp}}\underbrace{\xi_{\alpha}\left(
M_{\gamma}\right)  }_{\substack{=\sum_{\substack{A\in\mathbb{N}%
_{\operatorname*{red}}^{\bullet,\bullet}\text{;}\\\left(  \operatorname*{read}%
A\right)  ^{\operatorname*{red}}=\gamma\text{;}\\\operatorname*{row}A=\alpha
}}M_{\operatorname*{column}A}\\\text{(by
(\ref{pf.prop.QSym.secondcomult.alt.4}))}}}\otimes M_{\alpha}%
\ \ \ \ \ \ \ \ \ \ \left(  \text{by the definition of }\beta
_{\operatorname*{QSym}\nolimits_{\mathbf{k}}}\right) \nonumber\\
&  =\sum_{\alpha\in\operatorname*{Comp}}\ \ \sum_{\substack{A\in
\mathbb{N}_{\operatorname*{red}}^{\bullet,\bullet}\text{;}\\\left(
\operatorname*{read}A\right)  ^{\operatorname*{red}}=\gamma\text{;}%
\\\operatorname*{row}A=\alpha}}M_{\operatorname*{column}A}\otimes
\underbrace{M_{\alpha}}_{\substack{=M_{\operatorname*{row}A}\\\text{(since
}\operatorname*{row}A=\alpha\text{)}}}\nonumber\\
&  =\underbrace{\sum_{\alpha\in\operatorname*{Comp}}\ \ \sum_{\substack{A\in
\mathbb{N}_{\operatorname*{red}}^{\bullet,\bullet}\text{;}\\\left(
\operatorname*{read}A\right)  ^{\operatorname*{red}}=\gamma\text{;}%
\\\operatorname*{row}A=\alpha}}}_{=\sum_{\substack{A\in\mathbb{N}%
_{\operatorname*{red}}^{\bullet,\bullet}\text{;}\\\left(  \operatorname*{read}%
A\right)  ^{\operatorname*{red}}=\gamma\text{;}\\\operatorname*{row}%
A\in\operatorname*{Comp}}}}M_{\operatorname*{column}A}\otimes
M_{\operatorname*{row}A}\nonumber\\
&  =\sum_{\substack{A\in\mathbb{N}_{\operatorname*{red}}^{\bullet,\bullet
}\text{;}\\\left(  \operatorname*{read}A\right)  ^{\operatorname*{red}}%
=\gamma\text{;}\\\operatorname*{row}A\in\operatorname*{Comp}}%
}M_{\operatorname*{column}A}\otimes M_{\operatorname*{row}A}.
\label{pf.prop.QSym.secondcomult.alt.-2}%
\end{align}

But every $A\in\mathbb{N}_{\operatorname*{red}}^{\bullet,\bullet}$ satisfies
$\operatorname*{row}A\in\operatorname*{Comp}$\ \ \ \ \footnote{\textit{Proof.}
Let $A\in\mathbb{N}_{\operatorname*{red}}^{\bullet,\bullet}$. Then, the matrix
$A$ is reduced, and therefore row-reduced. In other words,
$\operatorname*{row}A$ is a composition. In other words, $\operatorname*{row}%
A\in\operatorname*{Comp}$, qed.}. Hence, the summation sign $\sum
_{\substack{A\in\mathbb{N}_{\operatorname*{red}}^{\bullet,\bullet}%
\text{;}\\\left(  \operatorname*{read}A\right)  ^{\operatorname*{red}}%
=\gamma\text{;}\\\operatorname*{row}A\in\operatorname*{Comp}}}$ on the right
hand side of (\ref{pf.prop.QSym.secondcomult.alt.-2}) can be replaced by
$\sum_{\substack{A\in\mathbb{N}_{\operatorname*{red}}^{\bullet,\bullet
}\text{;}\\\left(  \operatorname*{read}A\right)  ^{\operatorname*{red}}%
=\gamma}}$. Thus, (\ref{pf.prop.QSym.secondcomult.alt.-2}) becomes%
\begin{align}
\beta_{\operatorname*{QSym}\nolimits_{\mathbf{k}}}\left(  M_{\gamma}\right)
&  =\underbrace{\sum_{\substack{A\in\mathbb{N}_{\operatorname*{red}}%
^{\bullet,\bullet}\text{;}\\\left(  \operatorname*{read}A\right)
^{\operatorname*{red}}=\gamma\text{;}\\\operatorname*{row}A\in
\operatorname*{Comp}}}}_{=\sum_{\substack{A\in\mathbb{N}_{\operatorname*{red}%
}^{\bullet,\bullet}\text{;}\\\left(  \operatorname*{read}A\right)
^{\operatorname*{red}}=\gamma}}}M_{\operatorname*{column}A}\otimes
M_{\operatorname*{row}A}\nonumber\\
&  =\sum_{\substack{A\in\mathbb{N}_{\operatorname*{red}}^{\bullet,\bullet
}\text{;}\\\left(  \operatorname*{read}A\right)  ^{\operatorname*{red}}%
=\gamma}}M_{\operatorname*{column}A}\otimes M_{\operatorname*{row}A}.
\label{pf.prop.QSym.secondcomult.alt.-1}%
\end{align}

On the other hand, every $\gamma\in\operatorname*{Comp}$ satisfies%
\begin{align*}
\underbrace{\Delta_{P}^{\prime}}_{=\tau\circ\Delta_{P}}\left(  M_{\gamma
}\right)   &  =\left(  \tau\circ\Delta_{P}\right)  \left(  M_{\gamma}\right)
=\tau\left(  \underbrace{\Delta_{P}\left(  M_{\gamma}\right)  }%
_{\substack{=\sum_{\substack{A\in\mathbb{N}_{\operatorname*{red}}%
^{\bullet,\bullet};\\\left(  \operatorname*{read}A\right)
^{\operatorname*{red}}=\gamma}}M_{\operatorname*{row}A}\otimes
M_{\operatorname*{column}A}\\\text{(by the definition of }\Delta_{P}\text{)}%
}}\right) \\
&  =\tau\left(  \sum_{\substack{A\in\mathbb{N}_{\operatorname*{red}}%
^{\bullet,\bullet};\\\left(  \operatorname*{read}A\right)
^{\operatorname*{red}}=\gamma}}M_{\operatorname*{row}A}\otimes
M_{\operatorname*{column}A}\right) \\
&  =\sum_{\substack{A\in\mathbb{N}_{\operatorname*{red}}^{\bullet,\bullet
}\text{;}\\\left(  \operatorname*{read}A\right)  ^{\operatorname*{red}}%
=\gamma}}M_{\operatorname*{column}A}\otimes M_{\operatorname*{row}%
A}\ \ \ \ \ \ \ \ \ \ \left(  \text{by the definition of }\tau\right) \\
&  =\beta_{\operatorname*{QSym}\nolimits_{\mathbf{k}}}\left(  M_{\gamma
}\right)  \ \ \ \ \ \ \ \ \ \ \left(  \text{by
(\ref{pf.prop.QSym.secondcomult.alt.-1})}\right)  .
\end{align*}
Since both maps $\Delta_{P}^{\prime}$ and $\beta_{\operatorname*{QSym}%
\nolimits_{\mathbf{k}}}$ are $\mathbf{k}$-linear, this yields $\Delta
_{P}^{\prime}=\beta_{\operatorname*{QSym}\nolimits_{\mathbf{k}}}$ (since
$\left(  M_{\gamma}\right)  _{\gamma\in\operatorname*{Comp}}$ is a basis of
the $\mathbf{k}$-module $\operatorname*{QSym}\nolimits_{\mathbf{k}}$). This
proves Proposition \ref{prop.QSym.secondcomult.alt}.
\end{proof}

The next theorem is an analogue for $\operatorname*{QSym}$ of the Bernstein
homomorphism (\cite[\S 18.24]{HazeWitt}) for the symmetric functions:

\begin{theorem}
\label{thm.QSym.bernstein}Let $\mathbf{k}$ be a commutative ring. Let $H$ be a
commutative connected graded $\mathbf{k}$-Hopf algebra. For every composition
$\alpha$, define a $\mathbf{k}$-linear map $\xi_{\alpha}:H\rightarrow H$ as in
Definition \ref{def.bernstein}. Define a map $\beta_{H}:H\rightarrow
\underline{H}\otimes\operatorname*{QSym}\nolimits_{\mathbf{k}}$ as in
Definition \ref{def.bernstein}.

\textbf{(a)} The map $\beta_{H}$ is a $\mathbf{k}$-algebra homomorphism
$H\rightarrow\underline{H}\otimes\operatorname*{QSym}\nolimits_{\mathbf{k}}$
and a graded $\left(  \mathbf{k},\underline{H}\right)  $-coalgebra homomorphism.

\textbf{(b)} We have $\left(  \operatorname*{id}\otimes\varepsilon_{P}\right)
\circ\beta_{H}=\operatorname*{id}$, where we regard $\operatorname*{id}%
\otimes\varepsilon_{P}:\underline{H}\otimes\operatorname*{QSym}%
\nolimits_{\mathbf{k}}\rightarrow\underline{H}\otimes\mathbf{k}$ as a map from
$\underline{H}\otimes\operatorname*{QSym}\nolimits_{\mathbf{k}}$ to
$\underline{H}$ (by identifying $\underline{H}\otimes\mathbf{k}$ with
$\underline{H}$).

\textbf{(c)} Define a map $\Delta_{P}^{\prime}:\operatorname*{QSym}%
\nolimits_{\mathbf{k}}\rightarrow\operatorname*{QSym}\nolimits_{\mathbf{k}%
}\otimes\operatorname*{QSym}\nolimits_{\mathbf{k}}$ as in Definition
\ref{def.QSym.secondcomult} \textbf{(e)}. The diagram%
\[
\xymatrixcolsep{5pc}\xymatrix{
H \ar[r]^{\beta_H} \ar[d]_{\beta_H} & \underline{H} \otimes \QSym \ar[d]_{\beta_H\otimes\id} \\
\underline{H} \otimes \QSym \ar[r]_-{\id\otimes\Delta'_P} & \underline{H} \otimes \underline{\QSym} \otimes \QSym
}
\]
is commutative.

\textbf{(d)} If the $\mathbf{k}$-coalgebra $H$ is cocommutative, then
$\beta_{H}\left(  H\right)  $ is a subset of the subring $\underline{H}%
\otimes\Lambda_{\mathbf{k}}$ of $\underline{H}\otimes\operatorname*{QSym}%
\nolimits_{\mathbf{k}}$, where $\Lambda_{\mathbf{k}}$ is the $\mathbf{k}%
$-algebra of symmetric functions over $\mathbf{k}$.
\end{theorem}

Parts \textbf{(b)} and \textbf{(c)} of Theorem \ref{thm.QSym.bernstein} can be
combined into \textquotedblleft$\beta_{H}$ makes $H$ into a
$\operatorname*{QSym}\nolimits_{2}$-comodule, where $\operatorname*{QSym}%
\nolimits_{2}$ is the coalgebra $\left(  \operatorname*{QSym},\Delta
_{P}^{\prime},\varepsilon_{P}\right)  $\textquotedblright\ (the fact that this
$\operatorname*{QSym}\nolimits_{2}$ is actually a coalgebra follows from
Proposition \ref{prop.QSym.secondbialg}).

What Hazewinkel actually calls the Bernstein homomorphism in \cite[\S 18.24]%
{HazeWitt} is the $\mathbf{k}$-algebra homomorphism $H\rightarrow
\underline{H}\otimes\Lambda_{\mathbf{k}}$ obtained from our map $\beta
_{H}:H\rightarrow\underline{H}\otimes\operatorname*{QSym}\nolimits_{\mathbf{k}%
}$ by restricting the codomain when $H$ is both commutative and
cocommutative\footnote{Hazewinkel neglects to require the cocommutativity of
$H$ in \cite[\S 18.24]{HazeWitt}, but he uses it nevertheless.}. His
observation that the second comultiplication of $\Lambda_{\mathbf{k}}$ is a
particular case of the Bernstein homomorphism is what gave the original
motivation for the present note; its analogue for $\operatorname*{QSym}%
\nolimits_{\mathbf{k}}$ is our Proposition \ref{prop.QSym.secondcomult.alt}.

\begin{proof}
[Proof of Theorem \ref{thm.QSym.bernstein}.]Set $A=\underline{H}$ and
$\xi=\operatorname*{id}$. Then, the map $\xi_{\alpha}$ defined in Corollary
\ref{cor.ABS.hopf.esc} \textbf{(c)} is precisely the map $\xi_{\alpha}$
defined in Definition \ref{def.bernstein} (because $\xi^{\otimes
k}=\operatorname*{id}\nolimits^{\otimes k}=\operatorname*{id}$). Thus, we can
afford calling both maps $\xi_{\alpha}$ without getting confused.

\textbf{(a)} Corollary \ref{cor.ABS.hopf.esc} \textbf{(a)} shows that there
exists a unique graded $\left(  \mathbf{k},\underline{A}\right)  $-coalgebra
homomorphism $\Xi:H\rightarrow\underline{A}\otimes\operatorname*{QSym}%
\nolimits_{\mathbf{k}}$ for which the diagram
(\ref{eq.cor.ABS.hopf.esc.a.diag}) is commutative. Since $A=\underline{H}$ and
$\xi=\operatorname*{id}$, we can rewrite this as follows: There exists a
unique graded $\left(  \mathbf{k},\underline{H}\right)  $-coalgebra
homomorphism $\Xi:H\rightarrow\underline{H}\otimes\operatorname*{QSym}%
\nolimits_{\mathbf{k}}$ for which the diagram
\begin{equation}%
\xymatrix{
H \ar[rr]^-{\Xi} \ar[dr]_{\id} & & \underline{H} \otimes\QSym\ar
[dl]^{\id_H \otimes\varepsilon_P} \\
& \underline{H} &
}
\label{pf.thm.QSym.bernstein.1}%
\end{equation}
is commutative. Consider this $\Xi$. Corollary \ref{cor.ABS.hopf.esc}
\textbf{(c)} shows that this homomorphism $\Xi$ is given by%
\[
\Xi\left(  h\right)  =\sum_{\alpha\in\operatorname*{Comp}}\xi_{\alpha}\left(
h\right)  \otimes M_{\alpha}\ \ \ \ \ \ \ \ \ \ \text{for every }h\in H.
\]
Comparing this equality with (\ref{eq.def.bernstein.def}), we obtain
$\Xi\left(  h\right)  =\beta_{H}\left(  h\right)  $ for every $h\in H$. In
other words, $\Xi=\beta_{H}$. Thus, $\beta_{H}$ is a graded $\left(
\mathbf{k},\underline{H}\right)  $-coalgebra homomorphism (since $\Xi$ is a
graded $\left(  \mathbf{k},\underline{H}\right)  $-coalgebra homomorphism).

Corollary \ref{cor.ABS.hopf.esc} \textbf{(b)} shows that $\Xi$ is a
$\mathbf{k}$-algebra homomorphism. In other words, $\beta_{H}$ is a
$\mathbf{k}$-algebra homomorphism (since $\Xi=\beta_{H}$). This completes the
proof of Theorem \ref{thm.QSym.bernstein} \textbf{(a)}.

\textbf{(b)} Consider the map $\Xi$ defined in our above proof of Theorem
\ref{thm.QSym.bernstein} \textbf{(a)}. We have shown that $\Xi=\beta_{H}$.

The commutative diagram (\ref{pf.thm.QSym.bernstein.1}) shows that $\left(
\operatorname*{id}\otimes\varepsilon_{P}\right)  \circ\Xi=\operatorname*{id}$.
In other words, $\left(  \operatorname*{id}\otimes\varepsilon_{P}\right)
\circ\beta_{H}=\operatorname*{id}$ (since $\Xi=\beta_{H}$). This proves
Theorem \ref{thm.QSym.bernstein} \textbf{(b)}.

\textbf{(c)} Theorem \ref{thm.QSym.bernstein} \textbf{(a)} shows that the map
$\beta_{H}$ is a $\mathbf{k}$-algebra homomorphism $H\rightarrow
\underline{H}\otimes\operatorname*{QSym}\nolimits_{\mathbf{k}}$ and a graded
$\left(  \mathbf{k},\underline{H}\right)  $-coalgebra homomorphism. Theorem
\ref{thm.QSym.bernstein} \textbf{(a)} (applied to $\operatorname*{QSym}%
\nolimits_{\mathbf{k}}$ instead of $H$) shows that the map $\beta
_{\operatorname*{QSym}\nolimits_{\mathbf{k}}}$ is a $\mathbf{k}$-algebra
homomorphism $\operatorname*{QSym}\nolimits_{\mathbf{k}}\rightarrow
\underline{\operatorname*{QSym}\nolimits_{\mathbf{k}}}\otimes
\operatorname*{QSym}\nolimits_{\mathbf{k}}$ and a graded $\left(
\mathbf{k},\underline{\operatorname*{QSym}\nolimits_{\mathbf{k}}}\right)
$-coalgebra homomorphism. Since $\Delta_{P}^{\prime}=\beta
_{\operatorname*{QSym}\nolimits_{\mathbf{k}}}$ (by Proposition
\ref{prop.QSym.secondcomult.alt}), this rewrites as follows: The map
$\Delta_{P}^{\prime}$ is a $\mathbf{k}$-algebra homomorphism
$\operatorname*{QSym}\nolimits_{\mathbf{k}}\rightarrow
\underline{\operatorname*{QSym}\nolimits_{\mathbf{k}}}\otimes
\operatorname*{QSym}\nolimits_{\mathbf{k}}$ and a graded $\left(
\mathbf{k},\underline{\operatorname*{QSym}\nolimits_{\mathbf{k}}}\right)
$-coalgebra homomorphism.

Applying Corollary \ref{cor.ABS.hopf.esc} \textbf{(a)} to $\underline{H}%
\otimes\underline{\operatorname*{QSym}\nolimits_{\mathbf{k}}}$ and $\beta_{H}$
instead of $A$ and $\xi$, we see that there exists a unique graded $\left(
\mathbf{k},\underline{H}\otimes\underline{\operatorname*{QSym}%
\nolimits_{\mathbf{k}}}\right)  $-coalgebra homomorphism $\Xi:H\rightarrow
\underline{H}\otimes\underline{\operatorname*{QSym}\nolimits_{\mathbf{k}}%
}\otimes\operatorname*{QSym}\nolimits_{\mathbf{k}}$ for which the diagram
\begin{equation}%
\xymatrix{
H \ar[rr]^-{\Xi} \ar[dr]_{\beta_H} & & \underline{H} \otimes\underline{\QSym}
\otimes\QSym\ar[dl]^{\ \ \ \ \id_{\underline{H} \otimes\underline{\QSym}}
\otimes\varepsilon_P} \\
& \underline{H} \otimes\underline{\QSym} &
}
\label{pf.thm.QSym.bernstein.diagram}%
\end{equation}
is commutative. Thus, if we have two graded $\left(  \mathbf{k},\underline{H}%
\otimes\underline{\operatorname*{QSym}\nolimits_{\mathbf{k}}}\right)
$-coalgebra homomorphisms $\Xi:H\rightarrow\underline{H}\otimes
\underline{\operatorname*{QSym}\nolimits_{\mathbf{k}}}\otimes
\operatorname*{QSym}\nolimits_{\mathbf{k}}$ for which the diagram
(\ref{pf.thm.QSym.bernstein.diagram}) is commutative, then these two
homomorphisms must be identical. We will now show that the two homomorphisms
$\left(  \beta_{H}\otimes\operatorname*{id}\right)  \circ\beta_{H}$ and
$\left(  \operatorname*{id}\otimes\Delta_{P}^{\prime}\right)  \circ\beta_{H}$
both fit the bill; this will then yield that $\left(  \beta_{H}\otimes
\operatorname*{id}\right)  \circ\beta_{H}=\left(  \operatorname*{id}%
\otimes\Delta_{P}^{\prime}\right)  \circ\beta_{H}$, and thus Theorem
\ref{thm.QSym.bernstein} \textbf{(c)} will follow.

Recall that $\beta_{H}$ and $\Delta_{P}^{\prime}$ are graded maps. Thus, so
are $\left(  \beta_{H}\otimes\operatorname*{id}\right)  \circ\beta_{H}$ and
$\left(  \operatorname*{id}\otimes\Delta_{P}^{\prime}\right)  \circ\beta_{H}$.
Moreover, $\beta_{H}$ is a $\left(  \mathbf{k},\underline{H}\right)
$-coalgebra homomorphism, and $\Delta_{P}^{\prime}$ is a $\left(
\mathbf{k},\underline{\operatorname*{QSym}\nolimits_{\mathbf{k}}}\right)
$-coalgebra homomorphism. From this, it is easy to see that $\left(  \beta
_{H}\otimes\operatorname*{id}\right)  \circ\beta_{H}$ and $\left(
\operatorname*{id}\otimes\Delta_{P}^{\prime}\right)  \circ\beta_{H}$ are
$\left(  \mathbf{k},\underline{H}\otimes\underline{\operatorname*{QSym}%
\nolimits_{\mathbf{k}}}\right)  $-coalgebra
homomorphisms\footnote{\textit{Proof.} Proposition
\ref{prop.relative.composition.c2} (applied to $H$, $\underline{H}%
\otimes\operatorname*{QSym}\nolimits_{\mathbf{k}}$, $H$, $\operatorname*{QSym}%
\nolimits_{\mathbf{k}}$, $\beta_{H}$ and $\beta_{H}$ instead of $A$, $B$, $H$,
$G$, $f$ and $p$) shows that $\left(  \beta_{H}\otimes\operatorname*{id}%
\right)  \circ\beta_{H}$ is a $\left(  \mathbf{k},\underline{H}\otimes
\underline{\operatorname*{QSym}\nolimits_{\mathbf{k}}}\right)  $-coalgebra
homomorphism. It remains to show that $\left(  \operatorname*{id}\otimes
\Delta_{P}^{\prime}\right)  \circ\beta_{H}$ is a $\left(  \mathbf{k}%
,\underline{H}\otimes\underline{\operatorname*{QSym}\nolimits_{\mathbf{k}}%
}\right)  $-coalgebra homomorphism.
\par
Recall that $\Delta_{P}^{\prime}$ is a $\left(  \mathbf{k}%
,\underline{\operatorname*{QSym}\nolimits_{\mathbf{k}}}\right)  $-coalgebra
homomorphism $\operatorname*{QSym}\nolimits_{\mathbf{k}}\rightarrow
\underline{\operatorname*{QSym}\nolimits_{\mathbf{k}}}\otimes
\operatorname*{QSym}\nolimits_{\mathbf{k}}$. Hence, Proposition
\ref{prop.relative.composition.d} (applied to $\underline{\operatorname*{QSym}%
\nolimits_{\mathbf{k}}}$, $\underline{H}$, $\operatorname*{QSym}%
\nolimits_{\mathbf{k}}$, $\underline{\operatorname*{QSym}\nolimits_{\mathbf{k}%
}}\otimes\operatorname*{QSym}\nolimits_{\mathbf{k}}$ and $\Delta_{P}^{\prime}$
instead of $A$, $B$, $H$, $G$ and $f$) shows that $\operatorname*{id}%
\otimes\Delta_{P}^{\prime}:\underline{H}\otimes\operatorname*{QSym}%
\nolimits_{\mathbf{k}}\rightarrow\underline{H}\otimes
\underline{\operatorname*{QSym}\nolimits_{\mathbf{k}}}\otimes
\operatorname*{QSym}\nolimits_{\mathbf{k}}$ is an $\left(  \underline{H}%
,\underline{H}\otimes\underline{\operatorname*{QSym}\nolimits_{\mathbf{k}}%
}\right)  $-coalgebra homomorphism. Therefore, Proposition
\ref{prop.relative.composition.e} (applied to $\underline{H}$, $\underline{H}%
\otimes\underline{\operatorname*{QSym}\nolimits_{\mathbf{k}}}$, $H$,
$\underline{H}\otimes\operatorname*{QSym}\nolimits_{\mathbf{k}}$,
$\underline{H}\otimes\underline{\operatorname*{QSym}\nolimits_{\mathbf{k}}%
}\otimes\operatorname*{QSym}\nolimits_{\mathbf{k}}$, $\beta_{H}$ and
$\operatorname*{id}\otimes\Delta_{P}^{\prime}$ instead of $A$, $B$, $H$, $G$,
$I$, $f$ and $g$) shows that $\left(  \operatorname*{id}\otimes\Delta
_{P}^{\prime}\right)  \circ\beta_{H}$ is a $\left(  \mathbf{k},\underline{H}%
\otimes\underline{\operatorname*{QSym}\nolimits_{\mathbf{k}}}\right)
$-coalgebra homomorphism. This completes the proof.}.

Now, we shall show that the diagrams%
\begin{equation}%
\xymatrix{
H \ar[rr]^-{\left(\beta_H\otimes\id\right)\circ\beta_H}\ar[dr]_{\beta_H}
& & \underline{H} \otimes\underline{\QSym} \otimes\QSym\ar[dl]^{\ \ \ \ \id
_{\underline{H} \otimes\underline{\QSym}} \otimes\varepsilon_P} \\
& \underline{H} \otimes\underline{\QSym} &
}
\label{pf.thm.QSym.bernstein.diagram.1}%
\end{equation}
and%
\begin{equation}%
\xymatrix{
H \ar[rr]^-{\left(\id\otimes\Delta'_P\right)\circ\beta_H}\ar[dr]_{\beta_H}
& & \underline{H} \otimes\underline{\QSym} \otimes\QSym\ar[dl]^{\ \ \ \ \id
_{\underline{H} \otimes\underline{\QSym}} \otimes\varepsilon_P} \\
& \underline{H} \otimes\underline{\QSym} &
}
\label{pf.thm.QSym.bernstein.diagram.2}%
\end{equation}
are commutative. This follows from the computations%
\[
\underbrace{\left(  \operatorname*{id}\nolimits_{\underline{H}\otimes
\underline{\operatorname*{QSym}\nolimits_{\mathbf{k}}}}\otimes\varepsilon
_{P}\right)  \circ\left(  \beta_{H}\otimes\operatorname*{id}\right)  }%
_{=\beta_{H}\otimes\varepsilon_{P}=\beta_{H}\circ\left(  \operatorname*{id}%
\otimes\varepsilon_{P}\right)  }\circ\beta_{H}=\beta_{H}\circ
\underbrace{\left(  \operatorname*{id}\otimes\varepsilon_{P}\right)
\circ\beta_{H}}_{\substack{=\operatorname*{id}\\\text{(by Theorem
\ref{thm.QSym.bernstein} \textbf{(b)})}}}=\beta_{H}%
\]
and%
\begin{align*}
&  \left(  \underbrace{\operatorname*{id}\nolimits_{\underline{H}%
\otimes\underline{\operatorname*{QSym}\nolimits_{\mathbf{k}}}}}%
_{=\operatorname*{id}\nolimits_{\underline{H}}\otimes\operatorname*{id}%
\nolimits_{\underline{\operatorname*{QSym}\nolimits_{\mathbf{k}}}}}%
\otimes\varepsilon_{P}\right)  \circ\left(  \underbrace{\operatorname*{id}%
}_{=\operatorname*{id}\nolimits_{\underline{H}}}\otimes\underbrace{\Delta
_{P}^{\prime}}_{=\beta_{\operatorname*{QSym}\nolimits_{\mathbf{k}}}}\right)
\circ\beta_{H}\\
&  =\underbrace{\left(  \operatorname*{id}\nolimits_{\underline{H}}%
\otimes\operatorname*{id}\nolimits_{\underline{\operatorname*{QSym}%
\nolimits_{\mathbf{k}}}}\otimes\varepsilon_{P}\right)  \circ\left(
\operatorname*{id}\nolimits_{\underline{H}}\otimes\beta_{\operatorname*{QSym}%
\nolimits_{\mathbf{k}}}\right)  }_{=\operatorname*{id}\nolimits_{\underline{H}%
}\otimes\left(  \left(  \operatorname*{id}%
\nolimits_{\underline{\operatorname*{QSym}\nolimits_{\mathbf{k}}}}%
\otimes\varepsilon_{P}\right)  \circ\beta_{\operatorname*{QSym}%
\nolimits_{\mathbf{k}}}\right)  }\circ\beta_{H}\\
&  =\left(  \operatorname*{id}\nolimits_{\underline{H}}\otimes
\underbrace{\left(  \left(  \operatorname*{id}%
\nolimits_{\underline{\operatorname*{QSym}\nolimits_{\mathbf{k}}}}%
\otimes\varepsilon_{P}\right)  \circ\beta_{\operatorname*{QSym}%
\nolimits_{\mathbf{k}}}\right)  }_{\substack{=\operatorname*{id}\\\text{(by
Theorem \ref{thm.QSym.bernstein} \textbf{(b)},}\\\text{applied to
}\operatorname*{QSym}\nolimits_{\mathbf{k}}\text{ instead of }H\text{)}%
}}\right)  \circ\beta_{H}\\
&  =\underbrace{\left(  \operatorname*{id}\nolimits_{\underline{H}}%
\otimes\operatorname*{id}\right)  }_{=\operatorname*{id}}\circ\beta_{H}%
=\beta_{H}.
\end{align*}

Thus, we know that $\left(  \beta_{H}\otimes\operatorname*{id}\right)
\circ\beta_{H}$ and $\left(  \operatorname*{id}\otimes\Delta_{P}^{\prime
}\right)  \circ\beta_{H}$ are two graded $\left(  \mathbf{k},\underline{H}%
\otimes\underline{\operatorname*{QSym}\nolimits_{\mathbf{k}}}\right)
$-coalgebra homomorphisms $\Xi:H\rightarrow\underline{H}\otimes
\underline{\operatorname*{QSym}\nolimits_{\mathbf{k}}}\otimes
\operatorname*{QSym}\nolimits_{\mathbf{k}}$ for which the diagram
(\ref{pf.thm.QSym.bernstein.diagram}) is commutative (since the diagrams
(\ref{pf.thm.QSym.bernstein.diagram.1}) and
(\ref{pf.thm.QSym.bernstein.diagram.2}) are commutative). But we have shown
before that any two such homomorphisms must be identical. Thus, we conclude
that $\left(  \beta_{H}\otimes\operatorname*{id}\right)  \circ\beta
_{H}=\left(  \operatorname*{id}\otimes\Delta_{P}^{\prime}\right)  \circ
\beta_{H}$. This completes the proof of Theorem \ref{thm.QSym.bernstein}
\textbf{(c)}.

\textbf{(d)} Consider the map $\Xi$ defined in our above proof of Theorem
\ref{thm.QSym.bernstein} \textbf{(a)}. We have shown that $\Xi=\beta_{H}$.

Assume that $H$ is cocommutative. Corollary \ref{cor.ABS.hopf.esc}
\textbf{(d)} then shows that $\Xi\left(  H\right)  $ is a subset of the
subring $\underline{A}\otimes\Lambda_{\mathbf{k}}$ of $\underline{A}%
\otimes\operatorname*{QSym}\nolimits_{\mathbf{k}}$. In other words, $\beta
_{H}\left(  H\right)  $ is a subset of the subring $\underline{H}%
\otimes\Lambda_{\mathbf{k}}$ of $\underline{H}\otimes\operatorname*{QSym}%
\nolimits_{\mathbf{k}}$ (since $\Xi=\beta_{H}$ and $A=H$). This proves Theorem
\ref{thm.QSym.bernstein} \textbf{(d)}.

(Alternatively, we could prove \textbf{(d)} by checking that for any element
$h$ of a commutative cocommutative Hopf algebra $H$, the element $\xi_{\alpha
}\left(  h\right)  $ of $H$ depends only on the result of sorting $\alpha$,
rather than on the composition $\alpha$ itself.)
\end{proof}

\begin{proof}
[Proof of Proposition \ref{prop.QSym.secondbialg}.]Let $\tau$ be the twist map
$\tau_{\operatorname*{QSym}\nolimits_{\mathbf{k}},\operatorname*{QSym}%
\nolimits_{\mathbf{k}}}:\operatorname*{QSym}\nolimits_{\mathbf{k}}%
\otimes\operatorname*{QSym}\nolimits_{\mathbf{k}}\rightarrow
\operatorname*{QSym}\nolimits_{\mathbf{k}}\otimes\operatorname*{QSym}%
\nolimits_{\mathbf{k}}$. This twist map clearly satisfies $\tau\circ
\tau=\operatorname*{id}$. Hence, $\tau\circ\underbrace{\Delta_{P}^{\prime}%
}_{=\tau\circ\Delta_{P}}=\underbrace{\tau\circ\tau}_{=\operatorname*{id}}%
\circ\Delta_{P}=\Delta_{P}$.

Theorem \ref{thm.QSym.bernstein} \textbf{(c)} (applied to
$H=\operatorname*{QSym}\nolimits_{\mathbf{k}}$) shows that the diagram%
\[
\xymatrixcolsep{5pc}\xymatrix{
\QSym \ar[r]^{\beta_{\QSym}} \ar[d]_{\beta_{\QSym}} & \underline{\QSym} \otimes \QSym \ar[d]_{\beta_{\QSym}\otimes\id} \\
\underline{\QSym} \otimes \QSym \ar[r]_-{\id\otimes\Delta'_P} & \underline{\QSym} \otimes \underline{\QSym} \otimes \QSym
}
\]
is commutative. In other words, $\left(  \operatorname*{id}\otimes\Delta
_{P}^{\prime}\right)  \circ\beta_{\operatorname*{QSym}\nolimits_{\mathbf{k}}%
}=\left(  \beta_{\operatorname*{QSym}\nolimits_{\mathbf{k}}}\otimes
\operatorname*{id}\right)  \circ\beta_{\operatorname*{QSym}%
\nolimits_{\mathbf{k}}}$. Since $\beta_{\operatorname*{QSym}%
\nolimits_{\mathbf{k}}}=\Delta_{P}^{\prime}$ (by Proposition
\ref{prop.QSym.secondcomult.alt}), this rewrites as $\left(
\operatorname*{id}\otimes\Delta_{P}^{\prime}\right)  \circ\Delta_{P}^{\prime
}=\left(  \Delta_{P}^{\prime}\otimes\operatorname*{id}\right)  \circ\Delta
_{P}^{\prime}$. Thus, the operation $\Delta_{P}^{\prime}$ is coassociative.
Therefore, the operation $\Delta_{P}=\tau\circ\Delta_{P}^{\prime}$ is also
coassociative (because the coassociativity of a map $H\rightarrow H\otimes H$
does not change if we compose this map with the twist map $\tau_{H,H}:H\otimes
H\rightarrow H\otimes H$). It is furthermore easy to see that the operation
$\varepsilon_{P}$ is counital with respect to the operation $\Delta_{P}$ (see,
for example, \cite[\S 11.45]{HazeWitt}). Hence, the $\mathbf{k}$-module
$\operatorname*{QSym}\nolimits_{\mathbf{k}}$, equipped with the
comultiplication $\Delta_{P}$ and the counit $\varepsilon_{P}$, is a
$\mathbf{k}$-coalgebra. Our goal is to prove that it is a $\mathbf{k}%
$-bialgebra. Hence, it remains to show that $\Delta_{P}$ and $\varepsilon_{P}$
are $\mathbf{k}$-algebra homomorphisms. For $\varepsilon_{P}$, this is again
obvious (indeed, $\varepsilon_{P}$ sends any $f\in\operatorname*{QSym}%
\nolimits_{\mathbf{k}}$ to $f\left(  1,0,0,0,\ldots\right)  $). It remains to
prove that $\Delta_{P}$ is a $\mathbf{k}$-algebra homomorphism.

The map $\beta_{\operatorname*{QSym}\nolimits_{\mathbf{k}}}$ is a $\mathbf{k}%
$-algebra homomorphism $\operatorname*{QSym}\nolimits_{\mathbf{k}}%
\rightarrow\underline{\operatorname*{QSym}\nolimits_{\mathbf{k}}}%
\otimes\operatorname*{QSym}\nolimits_{\mathbf{k}}$ (by Theorem
\ref{thm.QSym.bernstein} \textbf{(a)}, applied to $H=\operatorname*{QSym}%
\nolimits_{\mathbf{k}}$). In other words, the map $\Delta_{P}^{\prime}$ is a
$\mathbf{k}$-algebra homomorphism $\operatorname*{QSym}\nolimits_{\mathbf{k}%
}\rightarrow\operatorname*{QSym}\nolimits_{\mathbf{k}}\otimes
\operatorname*{QSym}\nolimits_{\mathbf{k}}$ (since $\beta
_{\operatorname*{QSym}\nolimits_{\mathbf{k}}}=\Delta_{P}^{\prime}$, and since
$\underline{\operatorname*{QSym}\nolimits_{\mathbf{k}}}=\operatorname*{QSym}%
\nolimits_{\mathbf{k}}$ as $\mathbf{k}$-algebras). Thus, $\Delta_{P}=\tau
\circ\Delta_{P}^{\prime}$ is also a $\mathbf{k}$-algebra homomorphism (since
both $\tau$ and $\Delta_{P}^{\prime}$ are $\mathbf{k}$-algebra homomorphisms).
This completes the proof of Proposition \ref{prop.QSym.secondbialg}.
\end{proof}

\section{Remark on antipodes}

We have hitherto not really used the antipode of a Hopf algebra; thus, we
could just as well have replaced the words \textquotedblleft Hopf
algebra\textquotedblright\ by \textquotedblleft bialgebra\textquotedblright%
\ throughout the entire preceding text\footnote{That said, we would not have
gained anything this way, because any connected graded $\mathbf{k}$-bialgebra
is a $\mathbf{k}$-Hopf algebra (see \cite[Proposition 1.4.16]{Reiner}).}. Let
us now connect the preceding results with antipodes.

The antipode of any Hopf algebra $H$ will be denoted by $S_{H}$.

\begin{proposition}
\label{prop.relative.antipode}Let $\mathbf{k}$ be a commutative ring. Let $A$
be a commutative $\mathbf{k}$-algebra. Let $H$ be a $\mathbf{k}$-Hopf algebra.
Let $G$ be an $A$-Hopf algebra. Then, every $\mathbf{k}$-algebra homomorphism
$f:H\rightarrow G$ which is a $\left(  \mathbf{k},\underline{A}\right)
$-coalgebra homomorphism must also satisfy $f\circ S_{H}=S_{G}\circ f$.
\end{proposition}

\begin{proof}
[Proof of Proposition \ref{prop.relative.antipode}.]We know that $H$ is a
$\mathbf{k}$-Hopf algebra. Thus, $\underline{A}\otimes H$ is an $A$-Hopf
algebra. Its definition by extending scalars yields that its antipode is given
by $S_{\underline{A}\otimes H}=\operatorname*{id}\nolimits_{A}\otimes S_{H}$.

Let $f:H\rightarrow G$ be a $\mathbf{k}$-algebra homomorphism which is a
$\left(  \mathbf{k},\underline{A}\right)  $-coalgebra homomorphism. Then,
$f^{\sharp}:\underline{A}\otimes H\rightarrow G$ is an $A$-coalgebra
homomorphism (since $f$ is a $\left(  \mathbf{k},\underline{A}\right)
$-coalgebra homomorphism) and an $A$-algebra homomorphism (by Proposition
\ref{prop.relative.alg}). Hence, $f^{\sharp}$ is an $A$-bialgebra
homomorphism, thus an $A$-Hopf algebra homomorphism (since every $A$-bialgebra
homomorphism between two $A$-Hopf algebras is an $A$-Hopf algebra
homomorphism). Thus, $f^{\sharp}$ commutes with the antipodes, i.e., satisfies
$f^{\sharp}\circ S_{\underline{A}\otimes H}=S_{G}\circ f^{\sharp}$.

Now, let $\iota$ be the canonical $\mathbf{k}$-module homomorphism
$H\rightarrow\underline{A}\otimes H,\ h\mapsto1\otimes h$. Then, $\left(
\operatorname*{id}\nolimits_{A}\otimes S_{H}\right)  \circ\iota=\iota\circ
S_{H}$. On the other hand, $f^{\sharp}\circ\iota=f$ (this is easy to check).
Thus,%
\begin{align*}
\underbrace{f}_{=f^{\sharp}\circ\iota}\circ S_{H}  &  =f^{\sharp}%
\circ\underbrace{\iota\circ S_{H}}_{=\left(  \operatorname*{id}\nolimits_{A}%
\otimes S_{H}\right)  \circ\iota}=f^{\sharp}\circ\underbrace{\left(
\operatorname*{id}\nolimits_{A}\otimes S_{H}\right)  }_{=S_{\underline{A}%
\otimes H}}\circ\iota=\underbrace{f^{\sharp}\circ S_{\underline{A}\otimes H}%
}_{=S_{G}\circ f^{\sharp}}\circ\iota\\
&  =S_{G}\circ\underbrace{f^{\sharp}\circ\iota}_{=f}=S_{G}\circ f.
\end{align*}
This proves Proposition \ref{prop.relative.antipode}.
\end{proof}

\begin{corollary}
\label{cor.QSym.bernstein.antipode}Let $\mathbf{k}$ be a commutative ring. Let
$H$ be a commutative connected graded $\mathbf{k}$-Hopf algebra. Define a map
$\beta_{H}:H\rightarrow\underline{H}\otimes\operatorname*{QSym}%
\nolimits_{\mathbf{k}}$ as in Definition \ref{def.bernstein}. Then,%
\[
\beta_{H}\circ S_{H}=\left(  \operatorname*{id}\nolimits_{H}\otimes
S_{\operatorname*{QSym}\nolimits_{\mathbf{k}}}\right)  \circ\beta_{H}.
\]

\end{corollary}

\begin{proof}
[Proof of Corollary \ref{cor.QSym.bernstein.antipode}.]Theorem
\ref{thm.QSym.bernstein} \textbf{(a)} shows that the map $\beta_{H}$ is a
$\mathbf{k}$-algebra homomorphism $H\rightarrow\underline{H}\otimes
\operatorname*{QSym}\nolimits_{\mathbf{k}}$ and a graded $\left(
\mathbf{k},\underline{H}\right)  $-coalgebra homomorphism. Thus, Proposition
\ref{prop.relative.antipode} (applied to $A=H$, $G=\underline{H}%
\otimes\operatorname*{QSym}\nolimits_{\mathbf{k}}$ and $f=\beta_{H}$) shows
that $\beta_{H}\circ S_{H}=S_{\underline{H}\otimes\operatorname*{QSym}%
\nolimits_{\mathbf{k}}}\circ\beta_{H}$.

But the $H$-Hopf algebra $\underline{H}\otimes\operatorname*{QSym}%
\nolimits_{\mathbf{k}}$ is defined by extension of scalars; thus, its antipode
is given by $S_{\underline{H}\otimes\operatorname*{QSym}\nolimits_{\mathbf{k}%
}}=\operatorname*{id}\nolimits_{H}\otimes S_{\operatorname*{QSym}%
\nolimits_{\mathbf{k}}}$. Hence,%
\[
\beta_{H}\circ S_{H}=\underbrace{S_{\underline{H}\otimes\operatorname*{QSym}%
\nolimits_{\mathbf{k}}}}_{=\operatorname*{id}\nolimits_{H}\otimes
S_{\operatorname*{QSym}\nolimits_{\mathbf{k}}}}\circ\beta_{H}=\left(
\operatorname*{id}\nolimits_{H}\otimes S_{\operatorname*{QSym}%
\nolimits_{\mathbf{k}}}\right)  \circ\beta_{H}.
\]
This proves Corollary \ref{cor.QSym.bernstein.antipode}.
\end{proof}

\begin{corollary}
\label{cor.QSym.bernstein.antipodeH}Let $\mathbf{k}$ be a commutative ring.
Let $H$ be a commutative connected graded $\mathbf{k}$-Hopf algebra. Define a
map $\beta_{H}:H\rightarrow\underline{H}\otimes\operatorname*{QSym}%
\nolimits_{\mathbf{k}}$ as in Definition \ref{def.bernstein}. Then,%
\[
S_{H}=\left(  \operatorname*{id}\nolimits_{H}\otimes\left(  \varepsilon
_{P}\circ S_{\operatorname*{QSym}\nolimits_{\mathbf{k}}}\right)  \right)
\circ\beta_{H}.
\]

\end{corollary}

\begin{proof}
[Proof of Corollary \ref{cor.QSym.bernstein.antipodeH}.]We have%
\begin{align*}
\underbrace{\left(  \operatorname*{id}\nolimits_{H}\otimes\left(
\varepsilon_{P}\circ S_{\operatorname*{QSym}\nolimits_{\mathbf{k}}}\right)
\right)  }_{=\left(  \operatorname*{id}\nolimits_{H}\otimes\varepsilon
_{P}\right)  \circ\left(  \operatorname*{id}\nolimits_{H}\otimes
S_{\operatorname*{QSym}\nolimits_{\mathbf{k}}}\right)  }\circ\beta_{H}  &
=\left(  \operatorname*{id}\nolimits_{H}\otimes\varepsilon_{P}\right)
\circ\underbrace{\left(  \operatorname*{id}\nolimits_{H}\otimes
S_{\operatorname*{QSym}\nolimits_{\mathbf{k}}}\right)  \circ\beta_{H}%
}_{\substack{=\beta_{H}\circ S_{H}\\\text{(by Corollary
\ref{cor.QSym.bernstein.antipode})}}}\\
&  =\underbrace{\left(  \operatorname*{id}\nolimits_{H}\otimes\varepsilon
_{P}\right)  \circ\beta_{H}}_{\substack{=\operatorname*{id}\\\text{(by Theorem
\ref{thm.QSym.bernstein} \textbf{(b)})}}}\circ S_{H}=S_{H},
\end{align*}
and thus Corollary \ref{cor.QSym.bernstein.antipodeH} is proven.
\end{proof}

\begin{remark}
\label{rmk.embed}What I find remarkable about Corollary
\ref{cor.QSym.bernstein.antipodeH} is that it provides a formula for the
antipode $S_{H}$ of $H$ in terms of $\beta_{H}$ and $\operatorname*{QSym}%
\nolimits_{\mathbf{k}}$. Thus, in order to understand the antipode of $H$, it
suffices to study the map $\beta_{H}$ and the antipode of
$\operatorname*{QSym}\nolimits_{\mathbf{k}}$ well enough.

Similar claims can be made about other endomorphisms of $H$, such as the
Dynkin idempotent or the Eulerian idempotent (when $\mathbf{k}$ is a
$\mathbb{Q}$-algebra). Better yet, we can regard the map $\beta_{H}%
:H\rightarrow\underline{H}\otimes\operatorname*{QSym}\nolimits_{\mathbf{k}}$
as an \textquotedblleft embedding\textquotedblright\ of the $\mathbf{k}$-Hopf
algebra $H$ into the $H$-Hopf algebra $\underline{H}\otimes
\operatorname*{QSym}\nolimits_{\mathbf{k}}\cong\operatorname*{QSym}%
\nolimits_{H}$. Here, I am using the word \textquotedblleft
embedding\textquotedblright\ in scare quotes, since this map is not a Hopf
algebra homomorphism (its domain and its target are Hopf algebras over
different base rings); nevertheless, the map $\beta_{H}$ is injective (by
Theorem \ref{thm.QSym.bernstein} \textbf{(b)}), and the corresponding map
$\left(  \beta_{H}\right)  ^{\sharp}:\underline{H}\otimes H\rightarrow
\underline{H}\otimes\operatorname*{QSym}\nolimits_{\mathbf{k}}$ (sending every
$a\otimes h$ to $a\beta_{H}\left(  h\right)  $) is a graded $H$-Hopf algebra
homomorphism (because it is graded, an $H$-algebra homomorphism and an
$H$-coalgebra homomorphism); this shows that $\beta_{H}$ commutes with various
maps defined canonically in terms of a commutative connected graded Hopf
algebra. It appears possible to use this for proving identities in commutative
connected graded Hopf algebra.
\end{remark}

Note that Corollary \ref{cor.QSym.bernstein.antipodeH} is not completely new.
Indeed, it can also be obtained from Takeuchi's formula (\cite[Proposition
1.4.24]{Reiner}) by breaking up the map $f=\operatorname*{id}-u\epsilon$ into
$\pi_{1}+\pi_{2}+\pi_{3}+\cdots$. However, in its above form, it is more
suited to algebraic applications as discussed in Remark \ref{rmk.embed}.

\section{Questions and final remarks}

I shall finish with some remarks and open questions which may or may not be
worth further study.

\subsection{Dualizing $\operatorname*{QSym}\nolimits_{2}$}

It is well-known (see, e.g., \cite[\S 5.4]{Reiner}) that the graded
Hopf-algebraic dual of the graded Hopf algebra $\operatorname*{QSym}$ is a
graded Hopf algebra $\operatorname*{NSym}$. The second comultiplication
$\Delta_{P}$ and the second counit $\varepsilon_{P}$ on $\operatorname*{QSym}$
dualize to a second multiplication $m_{P}$ and a second unit $u_{P}$ on
$\operatorname*{NSym}$, albeit $u_{P}$ is not a map from $\mathbf{k}$ to
$\operatorname*{NSym}$ but rather a map from $\mathbf{k}$ to the completion
$\widehat{\operatorname*{NSym}}$ (specifically, to the completion of
$\operatorname*{NSym}$ with respect to its grading). We denote the
\textquotedblleft almost-$\mathbf{k}$-bialgebra\textquotedblright\ $\left(
\operatorname*{NSym},m_{P},u_{P},\Delta,\varepsilon\right)  $
(\textquotedblleft almost\textquotedblright\ because $u_{P}$ does not go into
$\operatorname*{NSym}$) by $\operatorname*{NSym}\nolimits^{\left(  2\right)
}$. Explicitly, its operations are given as follows:

\begin{itemize}
\item Its multiplication $m_{P}$ is given by%
\[
m_{P}\left(  H_{\beta}\otimes H_{\gamma}\right)  =\sum_{\substack{A\in
\mathbb{N}_{\operatorname*{red}}^{\bullet,\bullet};\\\operatorname*{row}%
A=\beta;\\\operatorname*{column}A=\gamma}}H_{\left(  \operatorname*{read}%
A\right)  ^{\operatorname*{red}}}\ \ \ \ \ \ \ \ \ \ \text{for all }%
\beta,\gamma\in\operatorname*{Comp},
\]
where $\left(  H_{\alpha}\right)  _{\alpha\in\operatorname*{Comp}}$ is the
basis of $\operatorname*{NSym}$ dual to the basis $\left(  M_{\alpha}\right)
_{\alpha\in\operatorname*{Comp}}$ of $\operatorname*{QSym}$. Thus, it is
precisely the \emph{internal product} $\ast$ introduced in \cite[Section
5.1]{NCSF1} (by \cite[Proposition 5.1]{NCSF1}). The canonical projection
$\operatorname*{NSym}\rightarrow\Lambda$ (which sends each $H_{\alpha}$ to the
complete homogeneous symmetric function $h_{\alpha}$) intertwines this
internal product $m_{P}$ with the Kronecker multiplication on $\Lambda$.

\item Its unit $u_{P}$ sends $1\in\mathbf{k}$ to the element%
\[
u_{P}=H_{\left(  {}\right)  }+H_{\left(  1\right)  }+H_{\left(  2\right)
}+H_{\left(  3\right)  }+\cdots
\]
of the completion $\widehat{\operatorname*{NSym}}$ of $\operatorname*{NSym}$.
\end{itemize}

Forgetting $\Delta$ and $\varepsilon$ for a moment, we can identify the
\textquotedblleft almost-algebra\textquotedblright\ $\operatorname*{NSym}%
\nolimits^{\left(  2\right)  }=\left(  \operatorname*{NSym},m_{P}%
,u_{P}\right)  $ with the direct sum of the \emph{descent algebras} of the
symmetric groups $S_{0},S_{1},S_{2},\ldots$ (see, e.g., \cite[Section
5.1]{NCSF1}).

We can (more or less) dualize Theorem \ref{thm.QSym.bernstein}. As a result,
instead of a $\operatorname*{QSym}\nolimits_{2}$-comodule structure on every
commutative graded connected Hopf algebra $H$, we obtain an
$\operatorname*{NSym}\nolimits^{\left(  2\right)  }$-module structure on every
cocommutative graded connected Hopf algebra $H$. This structure is rather
well-known: It has $H_{\alpha}\in\operatorname*{NSym}\nolimits^{\left(
2\right)  }$ act as the convolution product
\[
\pi_{a_{1}}\star\pi_{a_{2}}\star\cdots\star\pi_{a_{k}}\in\operatorname*{End}H
\]
for every composition $\alpha=\left(  a_{1},a_{2},\ldots,a_{k}\right)  $
(where $\star$ denotes the convolution product in $\operatorname*{End}H$).
This $\operatorname*{NSym}\nolimits^{\left(  2\right)  }$-module structure is
well-known; it appears implicitly in \cite[Th\'{e}or\`{e}me II.7]{Patras94}
(and is the map $\mathfrak{W}$ in \cite[Hint to Exercise 5.4.6]{Reiner},
although \cite{Reiner} does not prove that it is an action). It provides a way
to transfer information from \textbf{the} descent algebra
$\operatorname*{NSym}\nolimits_{2}$ to \textbf{a} descent algebra $\left(
\operatorname*{End}\nolimits_{\text{graded}}H,\circ\right)  $ of a
cocommutative graded connected Hopf algebra $H$.

\begin{question}
\label{quest.NSym}Is it possible to prove that this works using universal
properties like I have done above for Theorem \ref{thm.QSym.bernstein}? (Just
saying \textquotedblleft dualize Theorem \ref{thm.QSym.bernstein}%
\textquotedblright\ is not enough, because dualization over arbitrary
commutative rings is a heuristic, not a proof strategy; there does not seem to
be a general theorem stating that \textquotedblleft the dual of a correct
result is correct\textquotedblright, at least when the result has assumptions
about gradedness and similar things.)

If the answer is positive, can we use this to give a slick proof of Solomon's
Mackey formula? (I am not saying that there is a need for slick proofs of this
formula -- not after those by Gessel and Bidigare --, but it would be
interesting to have a new one. I am thinking of letting both
$\operatorname*{NSym}^{\left(  2\right)  }$ and the symmetric groups act on
the tensor algebra $T\left(  V\right)  $ of an infinite-dimensional free
$\mathbf{k}$-module $V$; one then only needs to check that the actions match.)
\end{question}

Note that if $u$ and $v$ are two elements of $\operatorname*{NSym}%
\nolimits^{\left(  2\right)  }$, then the action of the $\operatorname*{NSym}%
$-product $uv$ (not the $\operatorname*{NSym}\nolimits^{\left(  2\right)  }%
$-product!) on $H$ is the convolution of the actions of $u$ and $v$. So the
action map $\operatorname*{NSym}\nolimits^{\left(  2\right)  }\rightarrow
\operatorname*{End}H$ takes the multiplication of $\operatorname*{NSym}%
\nolimits^{\left(  2\right)  }$ to composition, and the multiplication of
$\operatorname*{NSym}$ to convolution.

\subsection{Natural transformations}

\begin{question}
\label{quest.allfunctorial}In Question \ref{quest.NSym}, we found a
$\mathbf{k}$-algebra homomorphism $\operatorname*{NSym}\nolimits^{\left(
2\right)  }\rightarrow\left(  \operatorname*{End}H,\circ\right)  $ for every
cocommutative connected graded Hopf algebra $H$. This is functorial in $H$,
and so is really a map (i.e., natural transformation) from the constant
functor $\operatorname*{NSym}\nolimits^{\left(  2\right)  }$ to the functor%
\begin{align*}
\left\{  \text{cocommutative connected graded Hopf algebras}\right\}   &
\rightarrow\left\{  \mathbf{k}\text{-modules}\right\}  ,\\
H  &  \mapsto\operatorname*{End}H.
\end{align*}
Does the image of this action span (up to topology) the whole functor? I guess
I am badly abusing categorical language here, so let me restate the question
in simpler terms: If a natural endomorphism of the $\mathbf{k}$-module $H$ is
given for every cocommutative connected graded Hopf algebra $H$, and this
endomorphism is known to annihilate all homogeneous components $H_{m}$ for
sufficiently high $m$ (this is what I mean by \textquotedblleft up to
topology\textquotedblright), then must there be an element $v$ of
$\operatorname*{NSym}\nolimits^{\left(  2\right)  }$ such that this
endomorphism is the action of $v$ ?

If the answer is \textquotedblleft No\textquotedblright, then does it change
if we require the endomorphism of $H$ to be graded? If we require $\mathbf{k}$
to be a field of characteristic $0$ ?

What if we restrict ourselves to commutative cocommutative connected graded
Hopf algebras? At least then, if $\mathbf{k}$ is a finite field $\mathbb{F}%
_{q}$, there are more natural endomorphisms of $H$, such as the Frobenius
morphism $x\mapsto x^{q}$ and its powers. One can then ask for the graded
endomorphisms of $H$, but actually it is also interesting to see how the full
$\mathbf{k}$-algebra of natural endomorphisms looks like (how do the
endomorphisms coming from $\operatorname*{NSym}\nolimits^{\left(  2\right)  }$
interact with the Frobenii?). And what about characteristic $0$ here?
\end{question}

\subsection{Dropping commutativity}

\begin{question}
\label{quest.noncocomm}What are the natural endomorphisms of connected graded
Hopf algebras, without any cocommutativity or commutativity assumption? I
suspect that they will form a connected graded Hopf algebra, with two
multiplications (one for composition and the other for convolution), but now
with a basis indexed by \textquotedblleft mopiscotions\textquotedblright%
\ (i.e., pairs $\left(  \alpha,\sigma\right)  $ of a composition $\alpha$ and
a permutation $\sigma\in\mathfrak{S}_{\ell\left(  \alpha\right)  }$). Is this
a known combinatorial Hopf algebra?
\end{question}

\subsection{Other combinatorial Hopf algebras?}

\begin{question}
\label{quest.bigger}Can we extend the map $\beta_{H}:H\rightarrow
\underline{H}\otimes\operatorname*{QSym}\nolimits_{\mathbf{k}}$ to a map
$H\rightarrow\underline{H}\otimes U$ for some combinatorial Hopf algebra $U$
bigger than $\operatorname*{QSym}\nolimits_{\mathbf{k}}$ ? What if we require
some additional (say, dendriform?) structure on $H$ ? Can we achieve
$U=\operatorname*{NCQSym}\nolimits_{\mathbf{k}}$ or
$U=\operatorname*{DoublePosets}\nolimits_{\mathbf{k}}$ (the combinatorial Hopf
algebra of double posets, which is defined for $\mathbf{k}=\mathbb{Z}$ and
denoted by $\mathbb{Z}\mathbf{D}$ in \cite{Mal-Reu-DP}, and can be similarly
defined over any $\mathbf{k}$) ? (I am singling out these two Hopf algebras
because they have fairly nice internal comultiplications. Actually, the
internal comultiplication of $\operatorname*{NCQSym}\nolimits_{\mathbf{k}}$ is
the key to Bidigare's proof of Solomon's Mackey formula \cite[\S 2]{Schocker},
and I feel it will tell us more if we listen to it.)

Aguiar suggests that the map $H\rightarrow\underline{H}\otimes
\operatorname*{NCQSym}\nolimits_{\mathbf{k}}$ I am looking for is the dual of
his action of the Tits algebra on Hopf monoids \cite[Proposition 88]{Aguiar13}.
\end{question}

\subsection{Further consequences?}

\begin{question}
\label{quest.relhopf}Do we gain anything from applying Corollary
\ref{cor.QSym.bernstein.antipode} to $H=\operatorname*{QSym}%
\nolimits_{\mathbf{k}}$ (thus getting a statement about $\Delta_{P}^{\prime}$)
? Probably not much for $\Delta_{P}^{\prime}$ that the Marne-la-Vall\'{e}e
school has not already discovered using virtual alphabets (the dual version is
the statement that $S\left(  a\ast b\right)  =a\ast S\left(  b\right)  $ for
all $a,b\in\operatorname*{NSym}\nolimits_{\mathbf{k}}$, where $\ast$ is the
internal product).
\end{question}

\begin{question}
\label{quest.split}From Theorem \ref{thm.QSym.bernstein} \textbf{(a)} and
Proposition \ref{prop.QSym.secondcomult.alt}, we can conclude that $\Delta
_{P}^{\prime}$ is a $\left(  \mathbf{k},\underline{\operatorname*{QSym}%
\nolimits_{\mathbf{k}}}\right)  $-coalgebra homomorphism. If I am not
mistaken, this can be rewritten as the equality%
\[
\left(  AB\right)  \ast G=\sum_{\left(  G\right)  }\left(  A\ast G_{\left(
1\right)  }\right)  \left(  B\ast G_{\left(  2\right)  }\right)
\]
(using Sweedler's notation) for any three elements $A$, $B$ and $G$ of
$\operatorname*{NSym}$. This is the famous splitting formula.

Now, it is known from \cite[\S 7]{NCSF7} that the same splitting formula holds
when $A$ and $B$ are elements of $\operatorname*{FQSym}$ (into which
$\operatorname*{NSym}$ is known to inject), as long as $G$ is still an element
of $\operatorname*{NSym}$ (actually, it can be an element of the bigger
Patras-Reutenauer algebra, but let us settle for $\operatorname*{NSym}$ so
far). Can this be proven in a similar vein? How much of the
Marne-la-Vall\'{e}e theory follows from Theorem \ref{thm.ABS.hopf}?
\end{question}

\end{document}